\documentclass{amsart}
\usepackage{amssymb, amscd, amsmath, float, graphicx, longtable,
  color, enumerate, bm, stmaryrd} 
\usepackage{mathrsfs}
\usepackage[all]{xy}
\SelectTips{eu}{} 
\xyoption{curve}

\newcommand{\hide}[1]{}

\usepackage[hypertex]{hyperref}

\DeclareMathOperator{\Cliff}{Cliff}

\DeclareMathOperator{\coh}{{{coh}}}

\DeclareMathOperator{\cok}{cok}
\DeclareMathOperator{\cone}{{\mathsf{cone}}}

\DeclareMathOperator{\End}{End}
\DeclareMathOperator{\Ext}{Ext}

\DeclareMathOperator{\GL}{GL}

\DeclareMathOperator{\Hom}{Hom}
\DeclareMathOperator{\id}{id}
\DeclareMathOperator{\idmap}{id}

\DeclareMathOperator{\Proj}{Proj}
\DeclareMathOperator{\rad}{rad}
\DeclareMathOperator{\rank}{rank}
\DeclareMathOperator{\Rep}{{\mathfrak{rep}}}
\DeclareMathOperator{\Spec}{Spec}

\DeclareMathOperator{\Sym}{Sym}

\renewcommand{\phi}{\varphi}
\newcommand{\vp}{\varphi}

\renewcommand{\tilde}{\widetilde}

\renewcommand{\leq}{\leqslant}
\renewcommand{\le}{\leqslant}
\renewcommand{\geq}{\geqslant}
\renewcommand{\ge}{\geqslant}

\renewcommand{\to}{\longrightarrow}
\newcommand\lto{{\longrightarrow}}
\newcommand\into{{\hookrightarrow}}
\newcommand\xto{\ -\negmedspace\negmedspace\negthinspace\xrightarrow}
\newcommand\onto{\twoheadrightarrow}

\renewcommand{\L}{{\textstyle\bigwedge}}
\newcommand{\dd}[1]{\frac{\partial}{\partial #1}}
\newcommand{\svee}{{\scriptscriptstyle \vee}}
\newcommand{\sbullet}{{\scriptscriptstyle \bullet}}

\newcommand{\m}{{\mathfrak{m}}}

\newcommand{\cata}{\mathsf{A}}

\newcommand{\K}{\mathsf{K}}
\newcommand{\Q}{\mathsf{Q}}
\newcommand{\Qtilde}{\tilde{\mathsf{Q}}}

\newcommand{\EE}{{\mathbb E}}

\newcommand{\KK}{{\mathbb K}}
\newcommand{\LL}{{\mathbb L}}
\newcommand{\NN}{{\mathbb N}}
\newcommand{\PP}{{\mathbb P}}
\newcommand{\QQ}{{\mathbb Q}}

\renewcommand{\SS}{{\mathbb S}}

\newcommand{\ZZ}{{\mathbb Z}}

\newcommand{\calb}{{\mathcal B}}

\newcommand{\cald}{{\mathcal D}}
\newcommand{\cale}{{\mathcal E}}
\newcommand{\calf}{{\mathcal F}}
\newcommand{\calg}{{\mathcal G}}
\newcommand{\calh}{{\mathcal H}}

\newcommand{\calm}{{\mathcal M}}

\newcommand{\calo}{{\mathcal O}}

\newcommand{\calq}{{\mathcal Q}}
\newcommand{\calr}{{\mathcal R}}

\newcommand{\calt}{{\mathcal T}}
\newcommand{\calu}{{\mathcal U}}
\newcommand{\calv}{{\mathcal V}}

\newcommand{\caly}{{\mathcal Y}}
\newcommand{\calz}{{\mathcal Z}}

\newcommand{\caloy}{{\calo_\caly}}
\newcommand{\caloz}{{\calo_\calz}}
\newcommand{\calop}{{\calo_\PP}}
\newcommand{\F}{\mathcal{F{}}}
\newcommand{\G}{\mathcal{G{}}}

\newcommand{\tilt}{{\mathcal T}}  

\def\cHom{{\mathscr H}\!om}
\newcommand{\rHom}[1][{}]{{\mathbf R}^{#1}\!\Hom}
\newcommand{\R}{{\mathbf R}}

\newcommand{\quiverQ}{\xymatrix@C=5pc{
1 &
2 \ar@<2ex>[l]_{\vdots}|{u_1} \ar@<-2ex>[l]|{u_n}  & 
\cdots \cdots \ar@<2ex>[l]_{\vdots}|{u_1} \ar@<-2ex>[l]|{u_n} & 
n \ar@<2ex>[l]_{\vdots}|{u_1} \ar@<-2ex>[l]|{u_n} 
}}

\newcommand{\quiverQlambdar}{\xymatrix@C=5pc{
1 &
2 \ar@<2ex>[l]_{\vdots}|{\lambda_1} \ar@<-2ex>[l]|{\lambda_n}  & 
\cdots \cdots \ar@<2ex>[l]_{\vdots}|{\lambda_1} \ar@<-2ex>[l]|{\lambda_n} & 
r \ar@<2ex>[l]_{\vdots}|{\lambda_1} \ar@<-2ex>[l]|{\lambda_n} 
}}

\newcommand{\quiverQtilde}{\xymatrix@C=5pc{
1 \ar@/_/[r]|{v_1} \ar@/_/@<-4ex>[r]^{\vdots}|{v_n}  & 
2 \ar@/_/[r]|{v_1} \ar@/_/@<-4ex>[r]^{\vdots}|{v_n}
\ar@/_/[l]_{\vdots}|{u_1} \ar@/_/@<-4ex>[l]|{u_n}  &  
\cdots \cdots \ar@/_/[l]_{\vdots}|{u_1} \ar@/_/@<-4ex>[l]|{u_n}
\ar@/_/[r]|{v_1} \ar@/_/@<-4ex>[r]^{\vdots}|{v_n}&  
n  \ar@/_/[l]_{\vdots}|{u_1} \ar@/_/@<-4ex>[l]|{u_n}
}}

\newcommand{\Qtildelambdageer}{\xymatrix@C=5pc{
1 \ar@/_/[r]|{g_1} \ar@/_/@<-4ex>[r]^{\vdots}|{g_n}  & 
2 \ar@/_/[r]|{g_1} \ar@/_/@<-4ex>[r]^{\vdots}|{g_n}
\ar@/_/[l]_{\vdots}|{\lambda_1} \ar@/_/@<-4ex>[l]|{\lambda_m}  &  
\cdots \cdots \ar@/_/[l]_{\vdots}|{\lambda_1} \ar@/_/@<-4ex>[l]|{\lambda_m}
\ar@/_/[r]|{g_1} \ar@/_/@<-4ex>[r]^{\vdots}|{g_n}&  
m  \ar@/_/[l]_{\vdots}|{\lambda_1} \ar@/_/@<-4ex>[l]|{\lambda_m}
}}

\newcommand{\quiverQinfty}{\xymatrix@C=4pc{
\cdots \ar@/_/[r]|{v_1} \ar@/_/@<-4ex>[r]^{\vdots}|{v_n}  & 
0 \ar@/_/[r]|{v_1} \ar@/_/@<-4ex>[r]^{\vdots}|{v_n}
 \ar@/_/[l]_{\vdots}|{u_1} \ar@/_/@<-4ex>[l]|{u_n}  &  
1 \ar@/_/[r]|{v_1} \ar@/_/@<-4ex>[r]^{\vdots}|{v_n}
 \ar@/_/[l]_{\vdots}|{u_1} \ar@/_/@<-4ex>[l]|{u_n}  &  
\cdots \ar@/_/[l]_{\vdots}|{u_1} \ar@/_/@<-4ex>[l]|{u_n}
 \ar@/_/[r]|{v_1} \ar@/_/@<-4ex>[r]^{\vdots}|{v_n} & 
m \ar@/_/[r]|{v_1} \ar@/_/@<-4ex>[r]^{\vdots}|{v_n}
 \ar@/_/[l]_{\vdots}|{u_1} \ar@/_/@<-4ex>[l]|{u_n}  & 
\cdots  \ar@/_/[l]_{\vdots}|{u_1} \ar@/_/@<-4ex>[l]|{u_n}
}}

\newcommand{\quiverQinftylabels}{\xymatrix@C=4pc{
\cdots \ar@/_/[r]|{g_1} \ar@/_/@<-4ex>[r]^{\vdots}|{g_n}  & 
0 \ar@/_/[r]|{g_1} \ar@/_/@<-4ex>[r]^{\vdots}|{g_n}
 \ar@/_/[l]_{\vdots}|{\lambda_1} \ar@/_/@<-4ex>[l]|{\lambda_m}  &  
1 \ar@/_/[r]|{g_1} \ar@/_/@<-4ex>[r]^{\vdots}|{g_n}
 \ar@/_/[l]_{\vdots}|{\lambda_1} \ar@/_/@<-4ex>[l]|{\lambda_m}  &  
\cdots \ar@/_/[l]_{\vdots}|{\lambda_1} \ar@/_/@<-4ex>[l]|{\lambda_m}
 \ar@/_/[r]|{g_1} \ar@/_/@<-4ex>[r]^{\vdots}|{g_n} & 
m \ar@/_/[r]|{g_1} \ar@/_/@<-4ex>[r]^{\vdots}|{g_n}
 \ar@/_/[l]_{\vdots}|{\lambda_1} \ar@/_/@<-4ex>[l]|{\lambda_m}  & 
\cdots  \ar@/_/[l]_{\vdots}|{\lambda_1} \ar@/_/@<-4ex>[l]|{\lambda_m}
}}

\newcommand{\simplequiver}{\xymatrix@=1.5cm{
0 \ar@<1ex>[r]^{G} & 1\ar@<1ex>[l]^{F^{*}} \ar@<1ex>[r]^{G} &
\ar@<1ex>[l]^{F^{*}} \cdots \ar@<1ex>[r]^{G} &\ar@<1ex>[l]^{F^{*}} n-1
}}

\theoremstyle{plain}
\newtheorem{theorem}{Theorem}

\newtheorem{prop}[theorem]{Proposition}
\newtheorem{proposition}[theorem]{Proposition}

\newtheorem{lemma}[theorem]{Lemma}

\newtheorem{cor}[theorem]{Corollary}

\newtheorem*{theorem*}{Theorem}
\newtheorem*{prop*}{Proposition}

\theoremstyle{definition}
\newtheorem{defn}[theorem]{Definition}

\newtheorem{remark}[theorem]{Remark}

\newtheorem{example}[theorem]{Example}

\newtheorem{sit}[theorem]{}
\newenvironment{nsit}{\begin{sit}\textit}{\end{sit}}
\numberwithin{theorem}{section}
\numberwithin{equation}{theorem}

\begin{document}

\title[Desingularization of Determinantal Varieties]{%
Non-commutative Desingularization of Determinantal Varieties I}

\author[R.-O. Buchweitz]{Ragnar-Olaf Buchweitz}
\address{Dept.\ of Computer and Mathematical Sciences, University of
Toronto at Scarborough, Toronto, Ont.\ M1C 1A4, Canada}
\email{ragnar@math.utoronto.ca}

\author[G.J. Leuschke]{Graham J. Leuschke}
\address{Dept.\ of Math., Syracuse University,
Syracuse NY 13244, USA}
\email{gjleusch@math.syr.edu}
\urladdr{http://www.leuschke.org/}

\author[M. Van den Bergh]{Michel Van den Bergh}
\address{Departement WNI, Universiteit Hasselt, 3590
  Diepenbeek, Belgium} 
\email{michel.vandenbergh@uhasselt.be}

\thanks{The first author was partly supported by NSERC grant
3-642-114-80. The second author was partly supported by NSA grant
H98230-05-1-0032 and NSF grant DMS~0556181.  The third author is
director of research at the FWO}

\date{\today}

\subjclass[2000]{Primary: 
  13C14, 
  14A22, 
  14E15, 
  14C40, 
  16S38; 
Secondary:
  13D02, 
  12G50, 
  16G20
}

\begin{abstract}
  We show that determinantal varieties defined by maximal minors of a
  generic matrix have a non-commutative desingularization, in that we
  construct a maximal Cohen-Macaulay module over such a variety whose
  endomorphism ring is Cohen-Macaulay and has finite global dimension.
  In the case of the determinant of a square matrix, this gives a
  non-commutative crepant resolution.
\end{abstract}
\maketitle

\setcounter{tocdepth}{1}
{\footnotesize\tableofcontents} 

\section{Introduction}\label{sect:intro}
Let $K$ be a field and $X = (x_{ij})$ an $(m \times n)$-matrix of
indeterminates over $K$ having $n\ge m$.  With $S=K[x_{ij}]$ the
polynomial ring in the $x_{ij}$, the matrix $X$ determines the
\emph{generic $S$-linear map} $\phi\colon S^n \to S^m$.  Let $\Spec R$
be the locus in $\Spec S$ where $\phi$ has non-maximal rank;
equivalently $R$ is the quotient of $S$ given by the maximal minors of
$X$.

The classical $R$-modules $M_a=\cok \L_S^a \phi$ are familiar objects
in commutative algebra.  In particular it is known that they are
maximal Cohen--Macaulay and are resolved by the Buchsbaum-Rim complex
(\cite[Corollary 2.6]{Buchsbaum-Rim:1964}, see also
\cite{Vetter:1992}).
In this paper we show that the $(M_a)_a$ conspire to yield a kind
of \emph{non-commutative desingularization} of the singular variety
$\Spec R$.  More precisely we prove the following result.

{
 \def\thetheorem{A}
\begin{theorem}[Thm.\ \ref{thm:mcm}] \label{thm:A}
For $1\le a \le m$ put $M_a=\cok \L_S^a \phi$ and $M=\bigoplus_a M_a$.
Then the endomorphism algebra $E = \End_R(M)$ is maximal
Cohen--Macaulay as an $R$-module, and has moreover finite global
dimension.
\end{theorem}
} 
%
\nocite{Buchweitz-Leuschke:MCMsdetX}

If $m=n$ then $R$ is the hypersurface ring $R=S/(\det \phi)$ and hence
$R$ is Gorenstein.  In this case our non-commutative desingularization
is an example of a \emph{non-commutative crepant resolution} as
defined in~\cite{vandenbergh:crepant}.  Non-commutative
desingularizations occurred probably first in theoretical physics
(e.g.\  \cite{berenstein-leigh:2001}) but they have recently been
encountered in a number of purely mathematical contexts (e.g.\
\cite{Bezrukavnikov:2006, iyama-reiten:2006, kaledin:2005,
Leuschke:finrepdim, szendroi:2007}).

The next result is a description by generators and relations
of the non-commuta\-tive resolution $E$. 
{
\def\thetheorem{B}
\begin{theorem}[Rem.\ \ref{rem:cubic}, Thm.\ \ref{thm:qCliffisoEndM}]
\label{thm:B}
As a $K$-algebra, $E$ is isomorphic to the path algebra $K\Qtilde$ of
the quiver
\[
\Qtilde: \qquad 
\Qtildelambdageer
\]
modulo relations 
\begin{align*}
\lambda_i \lambda_j + \lambda_j \lambda_i &= 0 = \lambda_i^2  &\text{for $i,j = 1, \dots, m$;}\\
g_i g_j + g_j g_i &= 0 = g_i^2 \qquad &\text{for $i,j = 1, \dots, n$;}\\
\lambda_k(\lambda_i g_j + g_j \lambda_i) &= (\lambda_i g_j + g_j
\lambda_i)\lambda_k \qquad &\text{for $i,k=1, \dots, 
  m$, $j=1, \dots, n$; and}\\
g_l(\lambda_i g_j + g_j \lambda_i) &= (\lambda_i g_j + g_j
\lambda_i)g_l \qquad &\text{for $i=1, \dots, 
  m$, $j,l=1, \dots, n$.}
\end{align*}
(terms in those relations which go outside the quiver are silently
suppressed, see \S\ref{sit:extBquiver}).
\end{theorem}
} 
\smallskip
Despite the fact that Theorems~\ref{thm:A} and~\ref{thm:B} have
purely algebraic statements, we will prove them
by relying on algebraic geometry. 
In our proofs we use the classical fact that $\Spec R$ has a
\emph{Springer type resolution of singularities}. To be precise,
define the incidence variety
\[
\calz = \lbrace \,([\lambda],\theta)\in \PP^{m-1}(K) \times M_{m\times n}(K)\,|\,
\lambda\theta=0\,\rbrace 
\]
with projections $p'\colon \calz \to \PP^{m-1}$ and $q'\colon \calz \to
\Spec R$. The following theorem contains  the key geometric facts we use. 
{
\def\thetheorem{C}
\begin{theorem}[Thm.\ \ref{thm:geometric}, Thm.\ \ref{tilting},
    Thm.\ \ref{thm:mcm}]
\label{thm:C}
The scheme $\calz$ is projective over $\Spec R$, which is of finite
type over $K$.  The $\caloz$-module
\[
\tilt := p'^{*}\left(\bigoplus_{a=1}^{m}
\left(\L^{a-1}\Omega_{\PP^{m-1}}\right)(a)\right) 
\]
is a classical tilting bundle on $\calz$ in the sense of
\cite{Hille-VdB:2007}, i.e.\
\begin{enumerate}
\item $\tilt$ is a locally free sheaf, in particular, a perfect
  complex on $\calz$,
\item $\tilt$ generates the derived category
  $\cald(\operatorname{Qch}(\calz))$, in that
  $\Ext^{\bullet}_{\calo_\calz}(\tilt,C)=0$ for a complex $C$ in
  $\cald(\operatorname{Qch}(\calz))$ implies $C\cong 0$, and
\item $\Hom_{\calo_\calz}(\tilt,\tilt[i]) = 0$ for $i \neq 0$.
\end{enumerate}
Furthermore we have 
\begin{enumerate}
\setcounter{enumi}{2}
\item $M\cong \R q'_\ast \calt$, and
\item $E\cong \End_{\calz}(\calt)$.
\end{enumerate}
\end{theorem}
} 
This theorem implies in particular that the geometric resolution
$\calz$ and the non-commutative resolution $E$ are derived equivalent
\cite{rickard:1989}.  Hence $\calz$ parametrizes certain objects in
the derived category of $E$.  
The
following result gives a more precise interpretation of this idea.  
{
\def\thetheorem{D}
\begin{theorem}[Thm.\ \ref{thm:representQtilde}]
\label{thm:D}
The variety $\calz$ is the fine moduli space for the\/
 $\Qtilde$-rep\-re\-sen\-ta\-tions $W$ of dimension vector $(1, m-1,
 \binom{m-1}{2}, \dots, 1)$ that are generated by the last component $W_m$.
\end{theorem}
}
\smallskip 
The proof of Theorem \ref{thm:C} is based on the explicit (and
characteristic-free) computation of the cohomology of certain
homogeneous bundles on $\PP^{m-1}$. More precisely, for
\[
\calm^b_a=\cHom_{\calo_{\PP^{m-1}}}\left((\L^{b-1}\Omega)(b),
  (\L^{a-1}\Omega)(a)\right)
\]
we compute in Theorem~\ref{thm:directimage} the cohomology of
$\calm^b_a(-c)$ for $c\in \ZZ$.  (The interested reader may wish to
compare the Appendix by Weyman to
\cite{Eisenbud-Schreyer-Weyman:2003}, which, by different methods,
computes as a special case $\Ext^i(\L^p \Omega, \L^q \Omega)$ for all
$i \geq 0$.)  This result is used in Theorem \ref{thm:hdi} to compute
the shape of the minimal $S$-projective resolution of $q'_\ast
p^{\prime \ast} \calm^b_a(-c)$ in many cases. This yields in
particular a large supply of maximal Cohen-Macaulay $R$-modules.

To prove Theorem \ref{thm:hdi} we use a new ``degeneracy criterion for
sparse spectral sequences'' (see Proposition \ref{projresfromSS})
which we think is interesting in its own right. Under mild boundedness
hypotheses this result asserts that if a page of a spectral sequence
has projective entries then we can obtain from it a projective
resolution of its limit.

\smallskip 
Two additional results occupy the last two sections: In
\S\ref{sec:explicit} we give an explicit minimal $S$-free presentation
for the maximal Cohen-Macaulay $R$-modules $\Hom_R(M_a,M_b)$ in terms
of certain minors in~$X$, and in \S\ref{sect:char0simples} we compute
(in characteristic zero) the shape of the minimal graded free
resolution of the graded simples of $E$.

\medskip

In characteristic zero we know how to generalize most of our results to
arbitrary determinantal varieties. This will be covered in a sequel to
the current paper. The present paper is largely characteristic-free.

\section{Notation}\label{sect:notation}


\setlongtables
\begin{longtable}{p{.2\linewidth} p{.75\linewidth}}
\bfseries Symbol & \bfseries Meaning\\
\endfirsthead
$K$ & a commutative base ring, most often a field\\
$F,G$ & projective $K$-modules of finite ranks $m \leq n$\\
$\L^a F$, $F_a$ & indicated exterior power of $F$ \\
$|F|$ & determinant of $F$, $\L^{\rank F}F$\\
$\SS^b =\Sym_K^b$ & symmetric power \\
$H$ & $\Hom_K(G,F)$\\
$S$ & $\SS(H^\svee)$, a polynomial algebra over $K$\\
${-}^\svee$ & dual over $K$ or $S$, as context implies\\
$\cata$, $\K^-(\PP\cata)$ & abelian category with
  enough projectives and the homotopy category of its projectives\\
$\G, \F$ & free $S$-modules induced from $G$ and $F$\\
$\phi\colon \G \to \F$ & the generic $S$-linear map\\
$X = (x_{ij})$ & generic $(m\times n)$-matrix of local coordinates on
  $\Spec S$\\ 
$R$ & the quotient of $S$ determined by the maximal minors of $X$\\
$\PP = \PP(F^\svee)$ & $K$-projective space on the dual of $F$\\
$\pi\colon\PP \to K$ & structure morphism\\
$\caly$ & $\PP \times_{\Spec K} H$, with projections
  $p\colon\caly\to\PP$, $q\colon\caly\to H$ \\
$\calz$ & the incidence variety desingularizing $\Spec R$, with
  inclusion $j\colon \calz \to \caly$\\
$\coh$ and $\operatorname{Qch}$ & categories of coherent and
quasi-coherent sheaves\\
$\cald^b_f$ & bounded derived category of complexes with finite
homology \\
$\KK(\id_F)$ & affine tautological Koszul complex over $\id_F$\\
$\KK$ & projective tautological Koszul complex\\
$F_a^b$&$\Hom_K(F_b,F_a)=F_a\otimes F^\svee_b$\\
$\Omega = \Omega_{\PP/K}$ & cotangent bundle on $\PP$ over $K$\\
$\Omega^a = \L_{\calop}^a\Omega$ & $\calop$-module of degree-$a$ differential
  forms\\
$\calu=\Omega(1)$ & the tautological subbundle of rank $m-1$ in $\pi^\ast F$\\
$\cale=\calu^\ast$ & the tautological quotient bundle of rank $m-1$ of
  $\pi^\ast F^\svee$\\ 
$\KK_{>a}$, $\KK_{\leq a}$ & certain (shifted) truncations of $\KK$\\
$\calm_a^b(-c)$ & $\cHom_\PP(\Omega^{b-1}(b),\Omega^{a-1}(a))(-c)$\\
$\tilt_a$ and $\tilt$ & $p'^*(\Omega^{a-1}(a))$ and $\bigoplus_a
  \tilt_a$\\
$M_a$ and $M$
& $\cok (\L^a\phi)$ and $ \bigoplus_a M_a$\\
$E$
& $\End_R(M)$
\\
$\Q$ & Be\u\i linson quiver on $F$\\
$B$ & path algebra of $\Q$\\
$\Qtilde$ & doubled Be\u\i linson quiver on $F$ and $G$\\
$C$ & quiverized Clifford algebra, path algebra of $\Qtilde$\\
$\Q^\infty$ and $C^\infty$ & infinite doubled Be\u\i linson quiver and
its path algebra\\
$\Cliff(b)$ & Clifford algebra on the quadratic map $b$\\
$\Rep(\Gamma)$ & abelian category of finite-dimensional
  representations of $\Gamma$\\
$\calr$, $\tilde\calr$ & certain representations of $\Q$
  and $\Qtilde$\\
$L_\alpha$ & Schur module corresponding to the partition $\alpha$\\
\end{longtable}

\[
\xymatrix{
\calz  \ar@{_{(}->}[dr]^{j} \ar@/^4ex/[drr]^{p'} \ar[dd]_{q'}\\
& \caly = H\times \PP \ar[r]^{p} \ar[d]_{q} & \PP=\PP(F^\svee)\ar[r]^-\pi& 
 \Spec K\\
\Spec R \ar@{^{(}->}[r] & H = \Hom_K(G,F)\\
}
\]

\section{Direct Images of \texorpdfstring{$\cHom$}{} Between
  Bundles of Differential Forms}
\label{sect:higherdirect}
In Sections~\ref{sect:higherdirect} through~\ref{sect:di-det} we prove
the technical results needed for the proofs of the theorems stated in
the Introduction.  At first reading the reader may wish to go directly
to Section~\ref{sec:geometricmethods} (after a pit stop in
\S\S\ref{nsit:genmorph}--\ref{nsit:resRjOZ} to pick up the notation)
where the applications start of the results obtained here.

\medskip

The aim of the present section is to determine the higher direct images of
the twisted bundles of homomorphisms between the modules of relative
differential forms on a projective bundle.  
The result is surely not new; it contains, for example, Bott's formula
for the twists of the differential forms themselves and the fact,
first exploited by Be{\u\i}linson in \cite{Beilinson:1978}, that the
direct sum $\bigoplus_{i}\Omega^{i}_{\PP/K}(i)$ is a tilting bundle with
its endomorphism ring isomorphic to a triangular version of the
exterior algebra. 

Not being aware of a complete, concise and explicit treatment of this
general case in the literature, although it is certainly contained in
the even more general treatment in \cite{Weyman:book}, as well as to
be able to use the ingredients of the proof later on, we recall here
the argument that relies entirely on properties of the tautological
Koszul complex, with the only challenge to keep the combinatorics at
bay. To this end we first introduce compact notation we use throughout
and then embark upon the actual computation after stating the result
as Theorem \ref{thm:directimage}.
\begin{nsit}{Notation.}
We fix in this section a commutative base ring $K$ and a projective 
$K$-module $F$ of \emph{constant finite rank} $m >0$ in the sense
that the $K$-module $\L_{K}^{m}F$ is invertible and faithful,
equivalently, locally free of rank one.

The considerations to come will involve various multilinear operations
on $F$ and we choose abbreviated notation as follows.
 
\begin{enumerate}[\quad\rm $\bullet$]
\item Unadorned tensor products are understood over $K$.
\item $M^{\svee}$ denotes the $K$-dual of the $K$-module $M$. 
\item $F_{a}=\L^{a}_{K}F$ represents the indicated \emph{exterior
power} of $F$ over $K$. It is a finite projective $K$-module, of
constant rank $\binom{m}{a}$, non-zero for $0\le a\le m$.

The resulting abbreviation $F^{\svee}_{a}$ is unambiguous, as
$\L^{a}(F^{\svee})\cong (\L^{a}F)^{\svee}$ via a canonical and natural
isomorphism of $K$-modules; see \cite[XIX, Prop. 1.5]{Lang}.
\item $F^{b}_{a}=\Hom_{K}(F_{b},F_{a})\cong F_{a}\otimes
  F^{\svee}_{b}$. The lower index thus indicates the covariant, the
  upper one the contravariant argument in the space of $K$-linear
  maps involved. 
\item $|F| =F_{m}$ denotes the \emph{determinant} of $F$, a projective
  $K$-module of (constant) rank $1$ by assumption. Again,
  $|F|^{\svee}\cong |F^{\svee}|$ canonically. 
\item $\SS^{b} = \Sym_{K}^{b}(F)$ represents the indicated
  \emph{symmetric power} of $F$ over $K$. 
It is again a projective $K$-module, of constant rank
  $\binom{m-1+b}{b}$, non-zero for $b\ge 0$. 
We write $\SS=\bigoplus_{b\ge 0}\SS^{b}=\Sym_{K}F$ for the
\emph{symmetric algebra} on $F$ over $K$, endowed with its canonical
grading that places $\SS^{b}$ into (internal) degree $b$. 
\end{enumerate}
Complexes will be graded \emph{cohomologically}, so that the
differential increments the complex degree by $1$. Recall that the
(simple) \emph{translation} of a complex, denoted~$[1]$, then
shifts a complex one place against the direction of the differential
and changes the sign of said differential. 
\end{nsit}

\begin{nsit}{The tautological Koszul complex.}
\label{nsit:affineKoszul}
Exterior and symmetric algebra on $F$ over $K$ combine to define the
(affine) \emph{tautological Koszul complex} $\KK(\id_{F})$ over the
identity map on $F$; see \cite[9.3 AX.151]{ALGX}.  That Koszul complex
often plays a dominant role in (co-)homological considerations, and
this instance is no exception.

Recall that the underlying bigraded ${\Sym_{K}}F$-module of
$\KK(\id_{F})$ is $\L_{K}(F[1])\otimes {\Sym_{K}}F$ and that the
differential can be described in a coordinate-free manner through the
comultiplication on the exterior algebra and the multiplication on the
symmetric algebra. Namely, denote $\Delta^{a-1,1}\colon F_{a}\to
F_{a-1}\otimes F$ the indicated bihomogeneous component of the
comultiplication defined by applying the exterior algebra functor to
$\Delta\colon F\to F\oplus F$ followed by the canonical isomorphism
$\L_{K}(F\oplus F)\cong \L_{K}F\otimes \L_{K}F$. With $\mu^{1,b}\colon
F\otimes \SS^{b}=\SS^{1}\otimes\SS^{b}\to \SS^{b+1}$ the indicated
bihomogeneous component of the multiplication on the symmetric
algebra, the differential $\partial$ is then simply the direct sum of
its bihomogeneous components
\begin{align*}
\partial_{a}^{b}\colon
F_{a}\otimes\SS^{b}\xto{\Delta^{a-1,1}\otimes\SS^{b}}F_{a-1}\otimes
F\otimes\SS^{b} \xto{F_{a-1}\otimes \mu^{1,b}} F_{a-1}\otimes
S^{b+1}\,.
\end{align*}
\end{nsit}

We continually use the following basic fact.  

\begin{prop}[{cf. \cite[9.3 Prop. 3]{ALGX}}]
The homogeneous strand of the Koszul complex $\KK(\id_{F})$ in
internal degree $a\in\ZZ$ is a complex of finite projective
$K$-modules of constant rank
\begin{align}
\label{eq:strand}
0\to |F|\otimes\SS^{a-m}\to F_{m-1}\otimes\SS^{a-m+1}\to\cdots\to
F \otimes\SS^{a-1} \to \SS^{a}\to 0\,.
\end{align}
It is supported on the integral interval $[-\min\lbrace a,m\rbrace ,0]$ and,
unless $a=0$, it is \emph{exact}, thus, even \emph{split exact} as its
terms are finite projective $K$-modules. If $a=0$, the complex
reduces to the single copy of $K\cong \SS^{0}$ placed in
(cohomological) degree $0$.\qed
\end{prop}

\begin{nsit}{The projective tautological Koszul complex.}
\label{nsit:projKoszul}
Now we turn to $\PP = \PP(F^{\svee})= \Proj_{K}(\Sym_{K}F)$, the
projective space of linear forms on $F$ over $K$ with structure
morphism $\pi\colon\PP\to \Spec K$ and its canonical very ample line bundle
$\calop(1)$. If $M$ is any $K$-module, we write $M\otimes\calop$ for
the induced $\calop$-module $\pi^{*}M$, and even
$M(i)=M\otimes\calop(i)$ for any integer $i$.

The $\calop$-linear \emph{Euler derivation} $e\colon F \otimes \calop(-1)
\to \calop$ corresponds to the identity on $F$ under the standard
identifications
\[\Hom_{\PP}(F \otimes \calop(-1), \calop)
\cong \Hom_{\PP}(\pi^{*}F, \calop(1)) \cong
\Hom_{K}(F,\pi_{*}\calop(1))\cong  \Hom_{K}(F,F) 
\]
It gives rise to the \emph{(projective) tautological Koszul complex}
of $\calop$-modules 
\begin{equation}
\label{eq:sheafKoszul}
\KK\quad \equiv\quad 0 \to F_{m}(-m) \to \cdots \to F (-1) \to \calop \to 0\,,
\end{equation}
where we place $\calop$ in (cohomological) degree zero, so that the
complex is supported again on the interval $[-m,0]$. This complex on
$\PP$ is the sheafification of the affine tautological Koszul complex
$\KK(\idmap_{F})$ in \ref{nsit:affineKoszul}, and, conversely, if one
applies $\pi_{*}$ to $\KK(a)$ for some integer $a$, the result is the
homogeneous strand of the affine Koszul complex displayed in
(\ref{eq:strand}) above.
\end{nsit}

\begin{nsit}{Differential forms.}
\label{diffforms}
The Koszul complex $\KK$ on $\PP$ is exact and decomposes into short
exact sequences 
\begin{equation}
\label{eq:defOmega-a}
0 \to \Omega^{a} \to F_{a}(-a) \to \Omega^{a-1} \to 0
\end{equation}
where $\Omega^{a}=\L_\PP^a \Omega^{1}_{\PP/K}$ denotes
the $\calop$-module of relative K\"ahler differential forms on $\PP$
of degree $a$, equivalently, the locally free sheaf of sections of the
$a^\text{th}$ exterior power of the cotangent bundle on $\PP$ relative
to $K$.

Recall as well that the locally free $\calop$-module
$\Omega^{m-1}\cong F_{m}(-m) = |F|(-m)$ of rank $1$ represents
$\omega_{\PP/K}$, the \emph{relative dualizing $\calop$-module} for
the projective morphism $\pi$.
\end{nsit}

\begin{nsit}{The canonical (co-)resolutions of the differential
	forms.}
\label{nsit:(co)resns}
Twisting, truncating, and translating the Koszul complex
(\ref{eq:sheafKoszul}) appropriately provides locally free resolutions
and coresolutions of the $\calop$-modules $\Omega^{a}(a')$, for any
integers $a, a'$. These (co-)resolutions are represented by the
following quasi-isomorphisms of complexes, where we view
$\Omega^{a}(a')[0]$ as a complex concentrated in degree zero,
\begin{align}
\label{Klea}
\xymatrix{
&\Omega^{a}(a')[0] \ar[d]_{i_{a}}^{\simeq}\\
\big(0 \ar[r] &F_{a}(a'{-}a) \ar[r]& \cdots\ar[r]&  F(a'{-}1) \ar[r]&
\calop(a')\ar[r]& 0\big)[-a]} 
\intertext{and}
\label{K>a}
\xymatrix{
\big(0\ar[r]& |F|(a'{-}m)\ar[r] &\cdots\ar[r]
&F_{a+1}(a'{-}a{-}1)\ar[r]\ar[d]_{p_{a}}^{\simeq} & 0\big)[-a-1]\\ 
          &&&\Omega^{a}(a')[0]}
\end{align}

Denote $\KK_{\leqslant a}(a')$ the locally free \emph{coresolution}
displayed in (\ref{Klea}). It is thus the (cochain) complex
concentrated on the interval $[0,a]$ with non-zero terms  
$$
\KK_{\leqslant a}(a')^{i}=F_{a-i}(a'-a+i)
$$ 
for $i=0,\ldots,a$, and with $\calh^{0}(\KK_{\leqslant a}(a')) \cong
\Omega^{a}(a')$ the only possibly non-vanishing cohomology
$\calop$-module. 

Analogously the locally free \emph{resolution} displayed in
(\ref{K>a}) is denoted $\KK_{> a}(a')$. It is a (chain) complex
concentrated on the interval $[a-m+1,0]$ with terms
$$
\KK_{> a}(a')^{j} = F_{a+1-j}(a'-a+j-1)
$$ 
for $j=a-m+1,\ldots,0$, and with $\calh^{0}(\KK_{>a}(a')) \cong
\Omega^{a}(a')$ the only possibly non-vanishing cohomology. 

The proper signs of the differentials in these (co-)resolutions are
uniquely determined by the requirement that the mapping cone over the
composition 
\begin{align*}
i_{a}p_{a}\colon\KK_{> a}(a')\xto{p_{a}}\Omega^{a}(a')[0]
\xto{i_{a}}\KK_{\leqslant a}(a') 
\end{align*}
returns exactly $\KK(a')[-a]$.
\end{nsit}

\begin{nsit}{The higher direct images.}
The (co-)resolutions displayed in (\ref{Klea}) and (\ref{K>a}) combine
to produce, for any integers $a,b$, and $c$,  four ways to represent
the locally free $\calop$-modules\footnote{%
The curious looking notation will be justified later.}
\begin{align*}
\calm^{b}_{a}(-c)&=
\cHom_{\calop}(\Omega^{b-1}(b),\Omega^{a-1}(a))(-c)\\
&\cong \cHom_{\calop}(\Omega^{b-1}(b-1),\Omega^{a-1}(a-1))(-c)
\end{align*} 
as sole cohomology sheaf in total degree zero of a bicomplex\footnote{
 bicomplex = total complex obtained from the corresponding (naive)
 double complex.}  
of locally free $\calop$-modules, each supported in exactly one of
the four quadrants in the plane, representing a suitably twisted
``cut-out'' of the endomorphism complex $\End_{\calop}(\KK)$ of the
projective tautological Koszul complex. Choosing the appropriate
bicomplex, the total derived direct image of $\calm^{b}_{a}(-c)$ can
be obtained as the cohomology of just the (dual of the) direct image
of that bicomplex. 
\end{nsit}

The work is reduced considerably in view of the following.

\begin{lemma}
\label{lem:iden}
For any integers $a,b$, and $c$, there are canonical isomorphisms of
locally free $\calop$-modules
\begin{align}
\label{lem:iden1}
\calm^{b}_{a}(-c)&\cong \calm^{m+1-a}_{m+1-b}(-c)
\end{align}
and
\begin{align}
\label{lem:iden2}
\cHom_{\calop}(\calm^{b}_{a}(-c),\Omega^{m-1})&\cong
\calm^{a}_{b}(c-m)\otimes_{\calop}\pi^{*}|F|\,. 
\end{align}
\end{lemma}

\begin{proof}
The non-degenerate pairing resulting from exterior multiplication
\begin{equation}
\label{eq:pairing}
-\wedge-\colon \Omega^{a-1}(a)\otimes_{\calop}
\Omega^{m-a}(-a)\to \Omega^{m-1}
\end{equation}
induces for each integer $a$ a natural isomorphism
\begin{align}
\label{eq:dualityomega}
 \Omega^{m-a}(m+1-a)\xto{\cong}
 \cHom_{\calop}(\Omega^{a-1}(a),\Omega^{m-1}(m+1))
\end{align}
whence applying the contravariant functor
$\cHom_{\calop}(-,\Omega^{m-1}(m+1))$ to each argument returns an
isomorphism 
\begin{align*}
\xymatrix{
\calm^{b}_{a}(-c) =
\cHom_{\calop}(\Omega^{b-1}(b),\Omega^{a-1}(a))(-c)\ar[d]^{\cong}\\ 
\calm^{m+1-a}_{m+1-b}(-c)=
\cHom_{\calop}(\Omega^{m-a}(m+1-a),\Omega^{m-b}(m+1-b))(-c) 
}
\end{align*}
as desired. Similarly one obtains from the definition of
$\calm^{b}_{a}(-c)$ and adjunction the first three isomorphisms in
\begin{align*}
\cHom_{\calop}(\calm^{b}_{a}(-c),\Omega^{m-1})
\cong\ &
\cHom_{\calop}(\cHom_{\calop}(\Omega^{b-1}(b),\Omega^{a-1}(a)),\Omega^{m-1})(c)\\    
\cong\ & 
\cHom_{\calop}(\Omega^{a-1}(a)\otimes_{\calop}\cHom_{\calop}(\Omega^{b-1}(b),\calop), \Omega^{m-1})(c)\\
\cong\ & 
\cHom_{\calop}(\Omega^{a-1}(a),\Omega^{b-1}(b)\otimes_{\calop} \Omega^{m-1})(c)\\
\cong\ &
\cHom_{\calop}(\Omega^{a-1}(a),\Omega^{b-1}(b))(c-m)\otimes_{\calop}\pi^{*}|F|\\
=\ &
\calm^{a}_{b}(c-m)\otimes_{\calop}\pi^{*}|F|
\end{align*}
while the fourth one uses the isomorphism $\Omega^{m-1}\cong |F|(-m)$
recalled in \ref{diffforms}, and the final equality substitutes the
definition of $\calm^{a}_{b}$.
\end{proof}

After these preliminary considerations we turn now to the
determination of the higher direct images of $\calm^{b}_{a}(-c)$ with
respect to the projective morphism $\pi\colon\PP\to \Spec K$. The
result is as follows, and the remainder of this section contains its
detailed proof, followed by some immediate consequences.
 
\begin{theorem}
\label{thm:directimage}
For any integers $a,b,c$, and each $\nu\in\ZZ$, the higher direct
image $\R^{\nu}\pi_{*}(\calm^{b}_{a}(-c))$ of the locally free
$\calop$-module
$$
\calm^{b}_{a}(-c)=\cHom_{\calop}\left(\Omega^{b-1}(b),\Omega^{a-1}(a)\right)(-c)
$$ 
is a finite projective $K$-module. 
In particular, the higher direct images
$\R^{\nu}\pi_{*}(\calm^{b}_{a}(-c))$ are non-zero for at most one value
of $\nu$. In case $a+b\ge m+1$, the precise situation is as follows.
\begin{enumerate}[\quad\rm (1)]
\item For $c < 0$, only the direct image
  $\R^{0}\pi_{*}(\calm^{b}_{a}(-c))= \pi_{*}(\calm^{b}_{a}(-c))$
  itself is non-zero. 
\item
\label{thm:di2}
For $0\le c < a-b$, for $m-b< c < a$, and for $a+m-b < c\le m$, all
(higher) direct images vanish.
\item For $\max\lbrace 0,a-b\rbrace \le c \le m-b$, the only non-vanishing higher
direct image is
\begin{align*}
\R^{c}\pi_{*}\left(\calm^{b}_{a}(-c)\right)\cong F^{\svee}_{c+b-a}\,.
\end{align*}
\item For $a\le c \le \min\lbrace a+m-b,m\rbrace $, the only non-vanishing higher
  direct image is 
\begin{align*}
\R^{c-1}\pi_{*}(\calm^{b}_{a}(-c))\cong F^{\svee}_{c+b-a}\,.
\end{align*}
\item For $m < c$, only the highest direct image is non-zero, and it satisfies
\begin{align*}
\R^{m-1}\pi_{*}(\calm^{b}_{a}(-c))\cong
\pi_{*}(\calm^{a}_{b}(c-m))^{\svee}\otimes |F|^{\svee}\,. 
\end{align*}
\end{enumerate}
The case $a+b < m+1$ reduces to the previous one in light of
Lemma~\ref{lem:iden}.
\end{theorem}

\begin{remark}
\label{rem:ext}
As the target of $\pi$ is affine and each $\calm^{b}_{a}(-c)$ is
locally free, the usual local-global spectral sequence yields natural
isomorphisms of $K$-modules 
\begin{align*}
\R^{\nu}\pi_{*}(\calm^{b}_{a}(-c))\cong
\Ext^{\nu}_{\calop}(\Omega^{b-1}(b'),\Omega^{a-1}(a')) 
\end{align*}
for any integers $\nu$ and $a',b'$ with $b'-a'=c+b-a$.

On the other hand, as the calculation of higher direct images is local
in the base, the reader may as well replace $\Spec K$ in Theorem
\ref{thm:directimage} by an arbitrary scheme with a locally free sheaf
$F$ of constant rank $m$ on it to obtain the analogous result for the
higher direct images relative to a projective bundle over an arbitrary
base scheme. 
\end{remark}

\begin{remark}
\label{rem:involution}
The results of the Theorem are invariant under the involution
$(a,b,c)\leftrightarrow(m+1-b,m+1-a,c)$ in view of Lemma
\ref{lem:iden}. Note further that either the first or the last
interval in Theorem \ref{thm:directimage}(\ref{thm:di2}) is {\em
empty}, depending on whether or not $a\le b$.
\end{remark}


\begin{proof}[Proof of Theorem \ref{thm:directimage}]
In view of the foregoing remark, we may assume without loss of
generality that $a+b\ge m+1$.  Using Grothendieck-Serre duality\
for the projective morphism $\pi$ we show next that \emph{it suffices
  to establish the claims for $c\le \frac{1}{2}(a-b+m)$}.  

Namely, assume we have shown that in the indicated range the higher
direct images are finite projective $K$-modules and that at most one
higher direct image $\R^{\nu}\pi_{*}(\calm^{b}_{a}(-c))$ is not zero
for given $a,b,c$. Using
\begin{align*}
\R^{\nu}\pi_{*}(\Omega^{m-1}) \cong
\begin{cases}
0&\text{for $\nu\neq m-1$}\\
K&\text{for $\nu = m-1$}
\end{cases}
\end{align*}
the duality theorem for projective morphisms yields that the natural
pairing of complexes of $K$-modules 
\begin{align*}
\R\pi_{*}(\calm^{b}_{a}(-c))\otimes^{\LL}
\R\pi_{*}(\cHom_{\calop}(\calm^{b}_{a}(-c),\Omega^{m-1}))\to 
\R\pi_{*}(\Omega^{m-1})\simeq K[1-m]
\end{align*}
is non-degenerate. The isomorphism (\ref{lem:iden2}) together with the
projection formula $\R\pi_{*}(-\otimes_{\calop}\pi^{*}|F|)\cong
\R\pi_{*}(-)\otimes |F|$ let us rewrite this pairing as
\begin{align*}
\R\pi_{*}(\calm^{b}_{a}(-c))\otimes^{\LL}
\R\pi_{*}(\calm^{a}_{b}(c-m))\otimes^{\LL} |F| \to K[1-m]\,.
\end{align*}

Accordingly, if the total direct image is represented by a single
finite projective $K$-module in cohomological degree $d$, so that
$\R\pi_{*}(\calm^{b}_{a}(-c))\simeq
\R^{d}\pi_{*}\calm^{b}_{a}(-c)[-d]$, we read off
\begin{align*}
\R^{\nu}\pi_{*}(\calm^{a}_{b}(c-m))\cong
\begin{cases}
0&\text{if $\nu\neq m-1-d$}\\
(\R^{d}\pi_{*}\calm^{b}_{a}(-c))^{\svee}\otimes |F|^{\svee}&\text{if
  $\nu =  m-1-d$}\,.
\end{cases}
\end{align*}

Under the involution $(a,b,c)\mapsto (a',b',c') = (b,a,m-c)$, the range 
\begin{align*}
a+b\ge m+1\quad&,\quad \max\lbrace 0,a-b\rbrace \le c\le m-b 
\intertext{is interchanged with the range}
a'+b'\ge m+1\quad&,\quad a'\le  c'\le \min\lbrace a'-b'+m, m\rbrace \,.
\end{align*}
Now assuming that the conclusion of (3) holds, one finds on the one hand
\begin{align*}
(\R^{c}\pi_{*}(\calm^{b}_{a}(-c))^{\svee}\otimes |F|^{\svee}
&\cong 
\R^{m-1-c}\pi_{*}(\calm^{a}_{b}(c-m))\\
&\cong 
\R^{c'-1}\pi_{*}(\calm^{b'}_{a'}(-c'))
\end{align*}
and on the other
\begin{align*}
(F^{\svee}_{c+b-a})^{\svee}\otimes |F|^{\svee}
&\cong 
F_{a'-b'+m-c'}\otimes |F|^{\svee}\cong F_{b'-a'+c'}^{\svee}
\end{align*}
with the last isomorphism due to the pairing induced by exterior
multiplication among the exterior powers of $F$. 

In this way, (3) and (4) are seen to be dual statements. Similarly,
the statements in (1) and (5) are dual to each other, while in (2) the
statements for the first and third interval are interchanged, the
statement for the middle one being selfdual.

It thus remains to prove the theorem for the range $a+b\ge m+1$ and
$c\le \frac{1}{2}(a-b+m)$.  The outline of the argument here is as
follows.




Depending on whether $c \le m-b$ or $m-b < c$, we choose a different
bicomplex to represent $\calm^{b}_{a}(-c)$ in the derived category of
$\PP$. The choice is made so that the individual terms of the
representing bicomplex are \emph{$\pi_{*}$-acyclic}, that is, the
direct image itself will be the only non-vanishing (higher) direct
image, and the bicomplex resulting from applying $\pi_{*}$ will
represent $\R\pi_{*}(\calm^{b}_{a}(-c))$ as a complex of finite
projective $K$-modules. We then analyse this bicomplex along its
``rows''. Each of these is the tensor product over $K$ of a finite
projective $K$-module with a (subcomplex of a) homogeneous strand of
the affine tautological Koszul complex $\KK(\idmap_{F})$, thus, is a
complex with easily determined cohomology. It then remains to assemble
the information so gained. 

Now we turn to the details, where we freely use the well known results
on the higher direct images of the locally free $\calop$-modules
$\calop(i), i\in \ZZ$; see \cite[Prop. 2.1.12]{EGAIII} for the general
case treated here. 

\begin{sit}
\label{sit:E2prop}
With $a+b\ge m+1$, assume first $c\le m-b$. Choosing for each of
$\Omega^{a-1}(a)$ and $\Omega^{b-1}(b)$ the appropriate coresolution
(\ref{Klea}), one can represent $\calm^{b}_{a}(-c)$ by  the following
bicomplex with non-zero terms concentrated in the fourth quadrant: 
\begin{align*}
\EE^{i,-j}_{\scriptscriptstyle{{+}{-}}} &=
\cHom_{\calop}(\KK_{\leqslant b-1}(b), \KK_{\leqslant
  a-1}(a))^{i,-j}(-c)\\ 
&\cong F_{a-1-i}(i+1)\otimes F^{\svee}_{b-1-j}(-1-j)(-c) \\
&\cong F^{b-1-j}_{a-1-i}(i-j-c)\quad\text{for $i,j\ge 0$}\,.
\end{align*}
This bicomplex evidently has the following properties:
\begin{enumerate}[\quad\rm (a)]
\item
\label{E2prop:1}
Each term $\EE^{i,-j}_{\scriptscriptstyle{+-}}$ is the twist by a
power of the distinguished very ample line bundle on $\PP$ of an
$\calop$-module induced from a finite projective $K$-module; 
\item
\label{E2prop:2}
The twist occurring in $\EE^{i,-j}_{\scriptscriptstyle{+-}}$ depends
only upon the total degree $i-j$;
\item 
\label{E2prop:3}
The bicomplex is  supported on the rectangle $[0,a-1]\times [-b+1,0]$
in the $(i,j)$-plane. 
\item 
\label{E2prop:4}
The twists occurring in non-zero terms range over the integers from
$1-b-c$ to $a-1-c$, an integral interval of length $a+b-2$. 
\item
\label{E2prop:5}
As $c \le m-b$ and $a+b\ge m+1$ by assumption, the possible twists
$?(t)$ occurring in non-zero terms of
$\EE^{i,-j}_{\scriptscriptstyle{+-}}$ satisfy $t \ge 1-m$, whence for
each such term the higher direct images \emph{vanish},
$\R^{\nu}\pi_{*}(\EE^{i,-j}_{\scriptscriptstyle{+-}}) =0$ for $\nu\neq
0$.
\end{enumerate}

Property (\ref{E2prop:5}) implies in particular that the total higher
direct image of $\calm^{b}_{a}(-c)$ is represented in the derived
category of $K$ by $\pi_{*}$ of this bicomplex,
\begin{align*}
\R\pi_{*}(\calm^{b}_{a}(-c)) \simeq
\pi_{*}\EE^{\scriptscriptstyle{\bullet,\bullet}}_{\scriptscriptstyle{+-}}\,. 
\end{align*}
\end{sit}

\begin{sit}
The cohomology of $
\pi_{*}\EE^{\scriptscriptstyle{\bullet,\bullet}}_{\scriptscriptstyle{+-}}$
is now readily determined by looking first at the ``rows'' of the
bicomplex. Fixing $j\in [0,b-1]$, the corresponding (row) complex
$\pi_{*}\EE^{\sbullet,-j}_{\scriptscriptstyle{{+}{-}}}$ is
concentrated on the line segment $[0,a-1]\times\lbrace -j\rbrace $ in the
$(i,j)$-plane and has the form
\begin{align}
\label{rowcomplex}
(0\to F_{a-1}\otimes\SS^{-j-c}\to F_{a-2}\otimes\SS^{1-j-c}
\to\cdots\to\SS^{a-1-j-c}\to 0)\otimes F^{\svee}_{b-1-j}[j]\,.
\end{align}
This complex is, up to the signs of the differentials, the translation
by $[j]$ of the tensor product over $K$ of $F^{\svee}_{b-1-j}$ with a
subcomplex of the homogeneous strand in internal degree $a-1-j-c$ in
the affine tautological Koszul complex $\KK(\idmap_{F})$ recalled in
\ref{nsit:projKoszul}. That strand of $\KK(\idmap_{F})$ is exact
except possibly at its ends. More precisely, the situation is as
follows.

\begin{lemma}
For $(i,j)\in [0,a-1]\times [0,b-1]$ and $c\le m-b$, the cohomology
$H^{i,-j}(\pi_{*}\EE^{\sbullet,-j}_{\scriptscriptstyle{{+}{-}}})$ of
the row complex just displayed in (\ref{rowcomplex}) is non-zero only
if
\begin{enumerate}[\quad\rm (1)]
\item $(i,-j)=(0,-j)$ with $-j>c$, and then
\begin{align*}
H^{0,-j}(\pi_{*}\EE^{\sbullet,-j}_{\scriptscriptstyle{{+}{-}}})
&\cong
\pi_{*}(\Omega^{a-1}(-j-c-1))\otimes F^{\svee}_{b-1-j}\\
&\cong
\ker\left(F_{a-1}\otimes\SS^{-j-c}\to
F_{a-2}\otimes\SS^{1-j-c}\right)\otimes F^{\svee}_{b-1-j}\\ 
&\cong
\cok\left(F_{a+1}\otimes\SS^{-2-j-c}\to
F_{a}\otimes\SS^{-1-j-c}\right)\otimes F^{\svee}_{b-1-j} 
\end{align*}
or
\item $(i,-j)=(a-1,c- a+1)$, in which case
\begin{align*}
H^{a-1,c-a+1}(\pi_{*}\EE^{\sbullet,-j}_{\scriptscriptstyle{{+}{-}}})
 \cong F^{\svee}_{b-1-j} 
 \cong F^{\svee}_{c+b-a}\,.
\end{align*}
\end{enumerate}
In either case, the cohomology is a finite projective $K$-module.
\qed
\end{lemma}
\end{sit}

\begin{nsit}{Visualization.}
\label{diag:1}
The reader might find it helpful to contemplate the following
visualisations in the $(i,j)$-plane, where the place with non-zero
homology with respect to the horizontal differential is marked
$\mathsf x$, those places with
$\pi_{*}\EE^{i,j}_{\scriptscriptstyle{{+}{-}}}\neq 0$ but no
horizontal homology are marked by $\bullet$, and the symbol $\circ$
refers to entries where
$\pi_{*}\EE^{i,j}_{\scriptscriptstyle{{+}{-}}}$ is zero.

We begin with the simplest case, when $0\le c \le m-b$, whence in
particular $0\le c \le a-1$.  We get then the following picture:
\begin{equation}
\label{pict:simplestpicture}
\pi_{*}\EE^{\scriptscriptstyle{\bullet,\bullet}}_{\scriptscriptstyle{+-}}
\quad\equiv\quad
\begin{array}{|c|cccccccc|}
\hline
\scriptstyle{0}&\circ&\cdots&\circ&\bullet&\bullet&\cdots&\bullet&\bullet\\
&\circ&\cdots&\circ&\circ&\bullet&\cdots&\bullet&\bullet\\
&\vdots&\ddots&\vdots&\vdots&\vdots&\ddots&\vdots&\vdots\\
&\circ&\cdots&\circ&\circ&\circ&\cdots&\bullet&\bullet\\
\scriptstyle{c-a+1}&\circ&\cdots&\circ&\circ&\circ&\cdots&\circ&\mathsf x\\
&\circ&\cdots&\circ&\circ&\circ&\cdots&\circ&\circ\\
&\vdots&\ddots&\vdots&\vdots&\vdots&\ddots&\vdots&\vdots\\
\scriptstyle {1-b}&\circ&\cdots&\circ&\circ&\circ&\cdots&\circ&\circ\\
\hline &\scriptstyle 0&&&\scriptstyle{c}&&&&
\scriptstyle {a-1}\\
\hline
\end{array}
\end{equation}
In other words, there is at most one non-vanishing cohomology group
occurring in those rows, whence the entire bicomplex equally only
carries this cohomology. Note that that cohomology indeed appears if,
and only if, $\max\lbrace 0,a-b\rbrace \le c$.
\end{nsit}

In summary, we read off the following result that settles the claims
in Theorem \ref{thm:directimage} for $c$ in the interval $[0,m-b]$.

\begin{proposition}
\label{prop:summary}
For $a+b\ge m+1$ and $0\le c < a-b$, the higher direct images of
$\calm^{b}_{a}(-c)$ all vanish, while for $\max\lbrace 0,a-b\rbrace \le c\le m-b$
the only non-vanishing one is the finite projective $K$-module
\begin{align*}
\R^{c}\pi_{*}(\calm^{b}_{a}(-c))\cong F^{\svee}_{c+b-a}\,.
\end{align*}
\qed
\end{proposition}

\begin{sit}
In case $c<0$, the corresponding diagram has the form
\[
\pi_{*}\EE^{\scriptscriptstyle{\bullet,\bullet}}_{\scriptscriptstyle{+-}}
\quad\equiv\quad
\begin{array}{|c|ccccccc|}
\hline
\scriptstyle{0}&\mathsf x&\bullet&\bullet&\cdots&\bullet&\bullet&\bullet\\
&\vdots&\vdots&\vdots&\ddots&\vdots&\vdots&\vdots\\
\scriptstyle c&\mathsf x&\bullet&\bullet&\cdots&\bullet&\bullet&\bullet\\
&\circ&\bullet&\bullet&\cdots&\bullet&\bullet&\bullet\\
&\circ&\circ&\bullet&\cdots&\bullet&\bullet&\bullet\\
&\vdots&\vdots&\vdots&\ddots&\vdots&\vdots&\vdots\\
&\circ&\circ&\circ&\cdots&\bullet&\bullet&\bullet\\
&\circ&\circ&\circ&\cdots&\circ&\bullet&\bullet\\
\scriptstyle{c-a+1}&\circ&\circ&\circ&\cdots&\circ&\circ&\mathsf x\\
&\circ&\circ&\circ&\cdots&\circ&\circ&\circ\\
&\vdots&\vdots&\vdots&\ddots&\vdots&\vdots&\vdots\\
\scriptstyle {1-b}&\circ&\circ&\circ&\cdots&\circ&\circ&\circ\\
\hline &\scriptstyle 0&&&&&&
\scriptstyle a-1\\
\hline
\end{array}
\]
with non-zero cohomology along the rows thus occurring only for total
degrees in the interval $[\max\lbrace 1-b,c\rbrace ,0]$.

Now $\R^{\nu}\calm=0$ for $\nu < 0$ and any $\calop$-module $\calm$,
whence the bicomplex
$\pi_{*}\EE^{\scriptscriptstyle{\sbullet,\sbullet}}_{\scriptscriptstyle{+-}}$
that represents $\R\pi_{*}(\calm^{b}_{a}(-c))$ only admits cohomology
in non-negative degrees. Combining these two facts, there can be at
most a single degree, namely $0$, in which there is non-vanishing
cohomology. This amounts to the following result.
\end{sit}

\begin{proposition}
\label{c<0}
For $a+b\ge m+1$ and $c < 0$, the higher direct images of
$\calm^{b}_{a}(-c)$ vanish except possibly\footnote{We will see below
  in Remark~\ref{rem:backhanded} that they are indeed non-zero!} 
for the direct image
$\pi_{*}(\calm^{b}_{a}(-c))$ itself, a finite projective $K$-module
of constant rank.
\end{proposition}

\begin{proof}
We already explained before stating the proposition why the higher
direct images necessarily vanish in degrees different from zero. The
final statement follows then from the universality of the
construction: the determination of the higher direct images of
$\calm^{b}_{a}(-c)$ as described is compatible with base change in the
base $\Spec K$, and the fact that the higher direct images are
concentrated in degree zero is independent of that base. It follows
that $\pi_{*}(\calm^{b}_{a}(-c))$ is $K$-flat, thus finite projective
over $K$ as it is finitely presented. Its rank can be computed through
the Euler characteristic of the ranks of the terms of the bicomplex,
whence the result is still constant across $\Spec K$.
\end{proof}
At this stage, we have established the claims in Theorem
\ref{thm:directimage} for $a+b\ge m+1$ and $c\le m-b$.
\begin{sit} 
\label{sit:concrete}
For further use we give in Lemma \ref{lem:concrete} below
a concrete  interpretation of the isomorphism
\begin{align}
\label{eq:concrete}
\R^{c}\pi_{*}(\calm^{b}_{a}(-c))\cong F^{\svee}_{c+b-a}\,.
\end{align}
for $a+b\ge m-1$ and $0\le c\le m-b$ (see Proposition
\ref{prop:summary}).

It will be convenient to regard $\KK=\L(F(-1)[1])$ as a
$\calo_{\PP}$-linear differentially graded algebra with differential
$d$ obtained by extending the Euler map $F(-1)\to
\calo_{\PP}$.  For $u,v\in \ZZ$ we regard $\KK(u)[v]$ as a
$\calo_{\PP}$-linear $\KK$-DG-bimodule.

Let $\lambda\in F^\svee$. By extending the linear map $\lambda\colon
\KK_1=F(-1)\to \calo_{\PP}(-1)= \KK(-1)_0$ we obtain a derivation
$\KK\to \KK(-1)[1]$ which we denote by $\partial_\lambda$.  The
commutator $d\partial_\lambda+\partial_\lambda d$ is a derivation and
since it is zero on generators it follows $
d\partial_\lambda+\partial_\lambda d=0$. A similar argument shows
$\partial_{\lambda}\partial_{\lambda'} +
\partial_{\lambda'}\partial_{\lambda}=0$.

If $\lambda^1\wedge\cdots \wedge\lambda^p\in F^\svee_p$ then we obtain a
corresponding differential operator $\partial_{\lambda^1}\cdots
\partial_{\lambda^p}\colon  \KK\to \KK(-p)[p]$ commuting with $d$.
This yields a map of complexes
\[
 F^\svee_p\otimes \KK\to \KK(-p)[p]\,.
\]
Put $p=b+c-a$. We obtain a map
\[
F^\svee_{b+c-a}\otimes \KK(b)[-b+1]\to \KK(a-c)[c-a+1]\,.
\]
Truncating in homological degree $0$ and taking into account the shift
incorporated into the definition of $\KK_{\le b+1}$ (see
\S\ref{nsit:(co)resns}) we obtain a map
\begin{equation}\label{eq:premap}
 F^\svee_{b+c-a}\otimes \KK_{\le b-1}(b)\to \KK_{\le -1+a-c}(a-c)
\subset \KK_{\le a-1}(a)(-c)[c]\,.
\end{equation}

\begin{lemma}\label{lem:concrete}
The map
\begin{equation}\label{eq:minidelta}
F^\svee_{b+c-a}\to \Hom_{\calo_{\PP}}(\KK_{\le b-1}(b), \KK_{\le
a-1}(a)(-c))[c] =\R\pi_{*}(\calm^{b}_{a}(-c))[c]
\end{equation}
obtained from \eqref{eq:premap} is a quasi-isomorphism.
\end{lemma}
\begin{proof}
Filtering the double complex $\Hom_{\calo_{\PP}}(\KK_{\le b-1}(b),
\KK_{\le a-1}(a))(-c)$ by rows as before, it is sufficient to show that 
the induced map to the only row carrying non-trivial cohomology is 
a quasi-isomorphism. Looking at the picture
\eqref{pict:simplestpicture} we see that we  
have to show that 
\begin{align*}
 F^\svee_{b+c-a}&\to 
\Hom_{\calo_{\PP}}(\KK_{\le b-1}(b)^{a-c-1}, \KK_{\le a-1}(a)^{a-1})(-c)\\
&=\pi_\ast(F^\svee_{b-1-(a-c-1)}(-1-(a-c-1))\otimes \calo_{\PP}(1+(a-1))(-c))\\
&=\pi_\ast(F^\svee_{b+c-a}\otimes \calo_{\PP})\\
&=F^\svee_{b+c-a}
\end{align*}
is an isomorphism. This is an easy verification.
\end{proof}
\end{sit}
\begin{sit}
\label{E++}
Now we turn to the higher direct images of $\calm^{b}_{a}(-c)$ in the
range $m-b<c < a\le m$, still under the assumption that $a+b\ge
m+1$. To this end, we choose the locally free coresolution for
$\Omega^{a-1}(a)$ as in (\ref{Klea}), but the locally free {\em
resolution} for $\Omega^{b-1}(b)$ as in (\ref{K>a}), to represent
$\calm^{b}_{a}(-c)$ by the resulting bicomplex concentrated in the
first quadrant. It has the terms
\begin{align*}
\EE^{i,j}_{\scriptscriptstyle{{+}{+}}} &= \cHom_{\calop}(\KK_{>
  b-1}(b), \KK_{\leqslant a-1}(a))^{i,j}(-c)\\ 
&\cong F_{a-1-i}(i+1)\otimes F^{\svee}_{b+j}(j)(-c)\\
&\cong F^{b+j}_{a-1-i}( i+j+1-c)\quad\text{for $i,j\ge 0$}\,.
\end{align*}
Take note of the following properties, analogous to the properties
(\ref{E2prop:1}) through  (\ref{E2prop:5}) in \S\ref{sit:E2prop} above. 
\begin{enumerate}[\quad\rm (a)]
\item
\label{Eprop:1}
Each term $\EE^{i,j}_{\scriptscriptstyle{++}}$ is the twist of a power
of the canonical line bundle on $\PP$ with an $\calop$-module induced
from a finite projective $K$-module;
\item
\label{Eprop:2}
The twist occurring in $\EE^{i,j}_{\scriptscriptstyle{++}} $ only
depends upon the total degree $i+j$;
\item 
\label{Eprop:3}
The bicomplex is supported on the rectangle $[0,a-1]\times [0,m-b]$ in
the $(i,j)$-plane.
\item 
\label{Eprop:4}
The twists occurring in non-zero terms range over the integers from
$1-c$ to $a-b-c+m$, an integral interval of length $a-b + m-1$.
\item
\label{Eprop:5}
In view of the preceding point, and as $0\le m-b < c \le m$ by our
current assumption, the possible twists $?(t)$ occurring in non-zero
terms of $\EE^{i,j}_{\scriptscriptstyle{++}}$ satisfy $a > t \ge 1-m$
and the lower bound shows again that for each such term the higher
direct images \emph{vanish},
$\R^{\nu}\pi_{*}(\EE^{i,j}_{\scriptscriptstyle{++}}) =0$ for $\nu\neq
0$.
\end{enumerate}
\end{sit}

As before, property (\ref{Eprop:5}) implies in particular that the
total derived direct image of $\calm^{b}_{a}(-c)$ is represented by
the direct image under $\pi_{*}$ of this bicomplex, 
\begin{align*}
\R\pi_{*}\calm^{b}_{a}(-c) \simeq
\pi_{*}\EE^{\scriptscriptstyle{\bullet,\bullet}}_{\scriptscriptstyle{++}}
\end{align*}
and we will determine its cohomology once again by looking first at
the corresponding ``row'' complexes.  Fixing therefore $j\in [0,m-b]$,
the complex
$\pi_{*}\EE^{{\scriptscriptstyle\bullet},j}_{\scriptscriptstyle{{+}{+}}}$
is concentrated on the line segment $[c-j-1,a-1]\times\lbrace j\rbrace $ in the
$(i,j)$-plane and has the form
\begin{align*}
(0\to F_{a+j-c}\to F_{a+j-c-1}\otimes \SS^{1}\to
  \cdots\to\SS^{a+j-c}\to 0)\otimes F^{\svee}_{b+j}[-j]\,. 
\end{align*}
Up to the signs of the differentials, this is the translation by
$[-j]$ of the tensor product over $K$ of $F^{\svee}_{b+j}$ with the
(entire!) homogeneous strand in internal degree $a+j-c$ in the affine
tautological Koszul complex $\KK(\idmap_{F})$ recalled in
\ref{nsit:projKoszul}.

Note that $a-m\le a+j-c \le a-1$ by the assumptions $c\in[m-b+1,m]$
and $j\in[0,m-b]$, whence either
\begin{enumerate}[\quad\rm(i)]
\item $a+j-c <0$, and this complex has \emph{no non-zero terms}, or 
\item $a+j-c =0$, and the strand of the affine Koszul complex has
cohomology equal to its only non-zero term, isomorphic to $K$, in
bidegree $(a-1,c-a)$, thus total degree $c-1$, or
\item $0<a+j-c \le a-1$, and the strand of the Koszul complex is
  \emph{exact}. 
\end{enumerate}

Depicting the situation again, for $m-b < c < a$ the resulting diagram
is of the form
\[
\pi_{*}\EE^{\scriptscriptstyle{\bullet,\bullet}}_{\scriptscriptstyle{++}}
\quad\equiv\quad
\begin{array}{|c|ccccccccc|}
\hline
\scriptstyle{m-b}&\circ&\cdots&\circ&\bullet&\bullet&\cdots&\bullet&\cdots&\bullet\\ 
&\circ&\cdots&\circ&\circ&\bullet&\cdots&\bullet&\cdots&\bullet\\
&\vdots&\ddots&\vdots&\vdots&\vdots&\ddots&\vdots&\ddots&\vdots\\
\scriptstyle 0&\circ&\cdots&\circ&\circ&\circ&\cdots&\bullet&\cdots&\bullet\\
\hline &\scriptstyle 0&&&\scriptstyle{c'}&&&&&
\scriptstyle {a-1}\\
\hline
\end{array}
\]
where we have set $c'=c-(m-b+1)$, satisfying $0\le c' < a+b-m-1 \le
a-1$.  In other words, all the rows here are already exact, so there are
\emph{no non-zero higher direct images}.  We record this as the
following result that covers the claims in Theorem
\ref{thm:directimage} for the range $m-b<c<a$.

\begin{proposition}
For $a+b\ge m+1$ and $m-b < c < a$, all higher direct images of
$\calm^{b}_{a}(-c)$ vanish, $\R^{\nu}\pi_{*}(\calm^{b}_{a}(-c))= 0$
for each integer $\nu$.\qed
\end{proposition}

The proof of Theorem \ref{thm:directimage} is now complete, in view of
the duality considerations at the beginning.
\end{proof}

For the benefit of the reader and for further use below, we visualize
the results in Theorem \ref{thm:directimage} as follows.

\begin{sit}
For $a+b\ge m+1$ and $m\ge a\ge b\ge 1$, depicting by $\bullet$ the
nonzero terms $\R^{\nu}\pi_{*}(\calm^{b}_{a}(-c))$, by $\circ$ or just
empty spaces the vanishing ones, and setting $m'=b-m, a'=b-a$ for
formatting purposes, results in the following diagram in the
$(-c,\nu)$-plane
\[
\begin{array}{|c|c|ccc|ccc|ccc|ccc|cc|}
\hline\scriptstyle\nu&&&&&&&&&&&&&&&\\
\hline
\scriptstyle{m-1}&\cdots&\bullet&&&&&&&&&&&&&\\
&&&\ddots&&&&&&&&&&&&\\
\scriptstyle {a-1}&&&&\bullet&\circ&&&&&&&&&&\\
&&&&&&\ddots&&&&&&&&&\\
&&&&&&&\circ&&&&&&&&\\
\scriptstyle {m-b}&&&&&&&&\bullet&&&&&&&\\
&&&&&&&&&\ddots&&&&&&\\
\scriptstyle {a-b}&&&&&&&&&&\bullet&&&&&\\
&&&&&&&&&&&\circ&&&&\\
&&&&&&&&&&&&\ddots&&&\\
\scriptstyle{0}&&&&&&&&&&&&&\circ&\bullet&\cdots\\
\hline &\scriptstyle {-c}&\scriptstyle{{-}m}&&\scriptstyle{-a}&&&&\scriptstyle{m'}&&
\scriptstyle{a'}&&&\scriptstyle {0}&&\scriptstyle {-c>0}\\
\hline
\end{array}
\]
while the corresponding diagram for $a+b\ge m+1$ and $m\ge b\ge a\ge
1$ is obtained from the above through a halfturn. It looks as follows,
where we have set this time $m'=b-a-m$ and $b'=b-m$,
\[
\begin{array}{|c|cc|ccc|ccc|ccc|ccc|c|}
\hline\scriptstyle\nu&&&&&&&&&&&&&&&\\
\hline
\scriptstyle{m-1}&\cdots&\bullet&\circ&&&&&&&&&&&&\\
&&&&\ddots&&&&&&&&&&&\\
&&&&&\circ&&&&&&&&&&\\
\scriptstyle {m-1-b+a}&&&&&&\bullet&&&&&&&&&\\
&&&&&&\hphantom{\ddots}&\ddots&&&&&&&&\\
\scriptstyle {a-1}&&&&&&&&\bullet&&&&&&&\\
&&&&&&&&&\circ&&&&&&\\
&&&&&&&&&&\ddots&&&&&\\
\scriptstyle {m-b}&&&&&&&&&&&\circ&\bullet&&&\\
&&&&&&&&&&&&&\ddots&&\\
\scriptstyle{0}&&&&&&&&&&&&&&\bullet&\cdots\\
\hline &\scriptstyle
	   {-c}&&\scriptstyle{{-}m}&&&\scriptstyle{m'}&&\scriptstyle{{-}a}& 
&&&\scriptstyle {b'}&&\scriptstyle {0}&\scriptstyle {-c>0}\\
\hline
\end{array}
\]
\end{sit}

Theorem \ref{thm:directimage} can also be reformulated in the
following terms, which are the most useful for the application we have
in mind.

\begin{cor}
\label{cor:van}
With notations as in \ref{thm:directimage}, for arbitrary integers
$a,b,c,d$, the higher direct image $\R^{d}\pi_{*}(\calm^{b}_{a}(-c))$
is \emph{not zero} only in the following cases:
\begin{enumerate}[\quad\rm (1)]
\item If $d-c > 0$, then $d=0$ and, necessarily, $c<0$.
\item If $d-c = 0$, then $c+b\in [\max\lbrace a,b\rbrace ,\min\lbrace m,a+b-1\rbrace ]$.
\item If $d-c = -1$, then $c-a\in  [\max\lbrace 0,m-a-b-1\rbrace ,\min\lbrace m-b,m-a\rbrace ]$.
\item If $d-c < -1$, then $d = m-1$ and, necessarily, $c>m$.
\end{enumerate}
\end{cor}

Using Remark \ref{rem:ext}, for $c=0$, Theorem \ref{thm:directimage}
returns the following well known fact, namely that the sequence
$\calop(-1)\cong \Omega^{m-1}(m-1),\ldots,\Omega^{0}(0)= \calop$ of
locally free $\calop$-modules ``between $\calop(-1)$ and $\calop$''
is \emph{strongly exceptional} in the sense of Bondal~\cite{Bondal:1989}.
\begin{cor}
In case $c=0$, we have $\calm^{b+1}_{a+1}\cong \cHom_{\calop}(
\Omega^{b}(b), \Omega^{a}(a))$ and \ref{thm:directimage} yields
\begin{align*}
\Ext^{i}_{\calop}(\Omega^{b}(b), \Omega^{a}(a))&\cong
\begin{cases}
\L_{K}^{b-a}F^{\svee}&\text{\rm if $i=0$ and $b\ge a$, and}\\
0&\text{\rm otherwise.}\\
\end{cases}
\end{align*}
\end{cor}


To end this section, we determine the ranks of the finite projective
$K$-modules occurring in Theorem \ref{thm:directimage}. They can be
easily determined in closed form by means of the Hilbert-Serre,
a.k.a. the Hirzebruch-Riemann-Roch Theorem and the result is as
follows.

\begin{cor} Let $m\ge a, b\ge 1$ be integers with $a+b\ge m+1$. Denote
  $r^{b}_{a}(z)\in \QQ[z]$ the unique polynomial of degree at most
  $m-1$ that at the integers in the interval $[-m,0]$ takes on the
  values 
\begin{align*}
r^{b}_{a}(-c) =
\begin{cases}
0 &\text{for $0\le c < a-b$,}\\
(-1)^{c}\binom{m}{c+b-a}&\text{for $\max\lbrace 0,a-b\rbrace \le c \le m-b$,}\\
0 &\text{for $m-b < c < a$,}\\
(-1)^{c-1}\binom{m}{c+b-a}&\text{for $a\le c \le \min\lbrace a+m-b,m\rbrace $,}\\
0 &\text{for $a+m-b < c \le m$.}
\end{cases}
\end{align*}
The ranks of the higher direct images of $\calm^{b}_{a}(-c)$ are then
determined uniquely through
\begin{align*}
\sum_{\nu}(-1)^{\nu}\rank_{K}\R^{\nu}\pi_{*}(\calm^{b}_{a}(-c)) =
r^{b}_{a}(c) 
\end{align*}
as for each triple $a,b,c$ at most one term in the sum on the left is
nonzero. 
\end{cor}

\section{Interlude: Projective Resolutions from Sparse Spectral
  Sequences}\label{sect:interlude} 
In this section we record a ``degeneracy result'' that allows one to
obtain a projective resolution of a bicomplex from one of the
associated spectral sequences, provided the corresponding first page
is ``sparse'' with projective terms. The result applies to bicomplexes
in any abelian category $\cata$ with enough projectives, whence we
assume here this setting.

\begin{nsit}{Categorical notation.}
Let $\K^{-}=\K^{-}(\PP\cata)$ denote the homotopy category of
complexes of projectives from $\cata$ that are bounded in the
direction of the differential, and, for an arbitrary complex $C$ over
$\cata$, denote by $\K^{-}/C$ the corresponding \emph{comma category};
see \cite[II.6]{MacL}. Its objects are thus homotopy classes of
morphisms of complexes $\vp\colon P\to C$ with $P\in \K^{-}$, and its
morphisms from $\vp\colon P\to C$ to $\vp'\colon P'\to C$ are those homotopy
classes of morphisms of complexes $\psi\colon P\to P'$, for which
$\vp'\psi=\vp$ in $\K^{-}$.

Recall that a morphism of complexes is a \emph{quasi-isomorphism} if
it induces an isomorphism in cohomology. If $C$ is any complex over
$\cata$, then a \emph{projective resolution} of $C$ is any
quasi-isomorphism $\vp\colon P\to C$ with source in $\K^{-}$. Such a
projective resolution, if it exists, is an object in $\K^{-}/C$, and in
there it is unique up to isomorphism.
\end{nsit}

\begin{nsit}{Assumptions.}
\label{nsit:assumptionsC}
Fix henceforth a bicomplex $C=(C^{i,j},d)$ supported on the \emph{upper
half-plane} $(i,j)\in\ZZ\times \NN$ and whose associated total
(direct sum) complex exists in the given abelian category
$\cata$. Equivalently, the (countable) direct sums
$C^{n}=\bigoplus_{i+j=n}C^{i,j}$ exist in $\cata$ for each integer
$n$. As $C$ is a bicomplex, the differential of $C$ decomposes as
$d=d_{h}+d_{v}$, where
\begin{alignat*}{2}
d_{h}&= \bigoplus_{i,j}d_{h}^{i,j}&\quad,\quad&d_{h}^{i,j}\colon C^{i,j}\lto
C^{i+1,j}\\ 
d_{v}&= \bigoplus_{i,j}d_{v}^{i,j}&\quad,\quad&d_{v}^{i,j}\colon C^{i,j}\lto
C^{i,j+1} 
\end{alignat*}
represent, respectively, the {horizontal} and {vertical}
components.

Filtering the bicomplex according to column degree, the resulting
spectral sequence \emph{converges} against the cohomology of $C$, as
the bicomplex is supported on the upper half-plane, and it displays on
its first page the vertical cohomology groups. In short,
\begin{align*}
E^{i,j}_{1} = H^{i,j}_{v}(C) =
H^{i,j}(C,d_{v})\quad\Longrightarrow\quad H^{i+j}(C)\,. 
\end{align*}
\end{nsit}

Below we will use the following basic fact. 

\begin{lemma}
\label{lem:projhom}
Let $D=(D,d)$ be a complex whose cohomology objects $H^{n}(D)$ are
\emph{projective} in $\cata$ for each integer $n$. Viewing the graded
object $H=\bigoplus_{n}H^{n}(D)[-n]$ over $\cata$ as a complex with
zero differential, there exists a quasi-isomorphism from it to $D$. In
other words, the cohomology itself constitutes a projective resolution
of $D$.
\end{lemma}

\begin{proof}
Indeed, let $Z$ denote the complex of cycles, which sits naturally
as a subcomplex of $D$ with zero differential.  The natural
epimorphism $Z\onto H$ of graded objects, or complexes with zero
differentials, admits a section $H\into Z$, as the components of $H$
are projective. The resulting composition $H\into Z\into D$ provides
for the desired quasi-isomorphism.
\end{proof}

Now we can formulate the ``degeneracy criterion''.
\begin{prop} 
\label{projresfromSS}
With $C$ as in \ref{nsit:assumptionsC}, suppose that each of its
vertical cohomology groups $E^{i,j}_{1}$ is \emph{projective} and
assume further that there exist an integer $a$ and a strictly
decreasing sequence of integers $i_{a}>i_{a-1}>\cdots$ such that
$E^{i,j}_{1}=0$ for
\begin{enumerate}[\quad\rm $\bullet$]
\item  $i > i_{a}$ and all $j$, and for
\item  $i+j\neq n$ when $i_{n-1}< i \le i_{n}$.
\end{enumerate}
In this case,
\begin{enumerate}[\quad\rm (1)]
\item for each integer $n$, the direct sum
$P^{n}=\bigoplus_{i+j=n}E^{i,j}_{1}$ exists and is \emph{projective}
in $\cata$; note that $P^{n}=0$ if $n > a$;
\item there exist morphisms $\lbrace \partial^{n}\colon P^{n}\to
P^{n+1}\rbrace _{n\in\ZZ}$ with $\partial^{n+1}\partial^{n}=0$, whence
$P=(P^{n},\partial^{n})$ constitutes a complex in $\K^{-}$;
\item
there exists a quasi-isomorphism $\vp=\lbrace \vp^{n}\rbrace _{n\in\ZZ}\colon P\to C$.
\end{enumerate}
In particular, $\vp\colon P\to C$ constitutes a projective resolution of $C$.
\end{prop}

\begin{proof}
Consider the (naive) ascending and exhaustive filtration
\begin{align*}
\xymatrix{
F_{a+1}=(C^{i,j})_{i > i_{a},j} \ar@{^(->}[r]
& \cdots\ar@{^(->}[r]
& F_{n}= (C^{i,j})_{i> i_{n-1},j}\ar@{^(->}[r]
& \cdots\ar@{^(->}[r]
& C
}
\end{align*}
on the bicomplex $C$. Each bicomplex $F_{\nu}$, for $\nu\le a+1$,
satisfies the same hypotheses as those assumed for $C$, and we first
establish the theorem for these bicomplexes by descending
induction. The proof will then be finished by passing to the limit.

The bicomplex $F_{a+1}$ is \emph{exact}: indeed, its vertical
cohomology $E^{i,j}_{1}(F_{a+1})$ vanishes by assumption, whence the
(total) cohomology of $F_{a+1}$ is equally $0$ as the associated
spectral sequence, essentially concentrated in the first quadrant,
converges. Accordingly, for $a'\ge a$, we get
$P^{a'}=\bigoplus_{i+j=a'}E^{i,j}_{1} = 0$, and so $\partial^{a'}=0$,
with $\vp_{a+1}\colon 0\to F_{a+1}$ a quasi-isomorphism. This establishes the
initial step of the induction.

Now assume that for some integer $\nu\le a$, 
\begin{enumerate}[\quad\rm (i)]
\item the terms $P^{\nu'}= \bigoplus_{i+j=\nu'}E^{i,j}_{1}$ exist and
  are projective in $\cata$ for $\nu' > \nu$, 
\item we have constructed a complex 
\begin{align*}
\PP^{\nu+1}\equiv\xymatrix{
(0\ar[r]&P^{\nu+1}\ar[r]^{\partial^{\nu+1}} &
P^{\nu+2}\ar[r]&\cdots\ar[r]&P^{a}\ar[r]^{\partial^{a}}&0)}
\end{align*}
\item
and a quasi-isomorphism $\vp_{\nu+1}\colon  \PP^{\nu+1}\to F_{\nu+1}$.
\end{enumerate}
By definition of the filtration, the quotient $F_{\nu}/F_{\nu+1}$ is a
bicomplex concentrated on the vertical strip $[i_{\nu-1}+1,
i_{\nu}]\times\NN$. On this strip of finite width, the vertical
cohomology is by hypothesis concentrated in total degree $\nu$, and so
involves only finitely many terms. Accordingly,
$P^{\nu}=\bigoplus_{i+j=\nu}E^{i,j}_{1}\cong
\bigoplus_{i=i_{\nu-1}+1}^{i_{\nu}}E^{i,\nu-i}_{1}$ is a finite direct
sum of projectives, thus, exists and is itself projective in $\cata$.
 
Moreover, $P^{\nu}[-\nu]$ represents the (total) cohomology of
$F_{\nu}/F_{\nu+1}$, as the spectral sequence
$E^{i,j}_{1}(F_{\nu}/F_{\nu+1})\Longrightarrow
H^{i+j}(F_{\nu}/F_{\nu+1})$ collapses on its first page, due to the
lack of cohomology outside the diagonal $i+j=\nu$. As $P^{\nu}$ is
projective, it follows from Lemma \ref{lem:projhom} that there exists
then a quasi-isomorphism of complexes $\chi^{\nu}\colon
P^{\nu}[-\nu]\xto{\cong} F_{\nu}/F_{\nu+1}$ from the complex with
$P^{\nu}$ as sole possibly non-zero term in degree $\nu$ to
$F_{\nu}/F_{\nu+1}$; it constitutes a projective resolution of
$F_{\nu}/F_{\nu+1}$.

The semi-split exact sequence of complexes 
\[0\to F_{\nu+1}\to F_{\nu}\to F_{\nu}/F_{\nu+1}\to 0
\] 
defines an exact triangle in the derived category $\cald(\cata)$ that
together with the already constructed quasi-isomorphisms accounts for
the solid arrows in
\begin{align*}
\xymatrix{
\PP^{\nu+1}\ar[d]^{\simeq}_{\vp^{\nu+1}}\ar@{..>}[r]&
\cone(\delta^{\nu})\ar@{..>}[r]\ar@{..>}[d]^{\simeq}_{\vp^{\nu}} & 
P^{\nu}[-\nu]\ar[d]^{\chi^{\nu}}_{\simeq}\ar@{-->}[r]^{\delta^{\nu}}&
\PP^{\nu+1}[1]\ar[d]_{\simeq}^{\vp^{\nu+1}[1]}\\
F_{\nu+1}\ar[r]&F_{\nu}\ar[r]& F_{\nu}/F_{\nu+1}\ar[r]^{\epsilon}&F_{\nu+1}[1]
}
\end{align*}
A morphism of complexes $\delta^{\nu}$ that lifts $\epsilon\circ
\chi^{\nu}$ through $\vp^{\nu+1}[1]$ as indicated then exists in
$\cald(\cata)$, as $P^{\nu}[-\nu]$ is in $\K^{-}$ and $\vp^{\nu+1}[1]$ is
a quasi-isomorphism.  Completing the upper row by
$\cone(\delta^{\nu})$, the mapping cone over $\delta^{\nu}$, to an
exact triangle, there exists next in the triangulated category
$\cald(\cata)$ a morphism $\vp^{\nu}\colon \cone(\delta^{\nu})\to F^{\nu}$ as
indicated so that the triple $(\vp^{\nu+1},\vp^{\nu}, \chi^{\nu})$
constitutes a morphism of exact triangles. As the other two components
are isomorphisms in $\cald(\cata)$, the same necessarily holds true for
$\vp^{\nu}$, and, finally, that isomorphism in $\cald(\cata)$ can be
represented by an actual quasi-isomorphism of complexes, as
$\cone(\delta^{\nu})$ is by construction a complex in $\K^{-}$.

It remains to observe that the morphism of complexes $\delta^{\nu}$
involves at most a single non-zero component, represented by a morphism
from $P^{\nu}\to P^{\nu +1}$, due to the support of the complexes
involved. Indeed, this component is nothing but the morphism induced
in cohomology by the composition
\[
F_{\nu}/F_{\nu+1}\xto{\epsilon} F_{\nu+1}[1]\onto F_{\nu+1}/F_{\nu+2}[1]\,.
\] 
It follows in particular that $\PP^{\nu}=\cone(\delta^{\nu})$ has the
desired form, with $\partial^{\nu}$ that single non-zero component of
$\delta^{\nu}$, up to the sign dictated by the convention on
differentials in mapping cones. This completes the inductive step.

As an aside, the reader may note that the preceding argument can as
well be made directly on the level of morphisms of complexes by
invoking the appropriate version of the \emph{horseshoe lemma} to
construct the quasi-isomorphism $\vp^{\nu}$ with source $\PP^{\nu}$ of
the form claimed.

So far, we have constructed a diagram of morphisms of complexes
\begin{align*}
\xymatrix{
\PP^{a+1}\ar@{^(->}[r]\ar[d]^{\simeq}_{\vp^{a+1}}&
\PP^{a}\ar@{^(->}[r]\ar[d]^{\simeq}_{\vp^{a}}&\cdots\ar@{^(->}[r]&
\PP^{n}\ar@{^(->}[r]\ar[d]^{\simeq}_{\vp^{n}}&\cdots\\
F^{a+1}\ar@{^(->}[r]&F^{a}\ar@{^(->}[r] & 
\cdots\ar@{^(->}[r]&F^{n}\ar@{^(->}[r]&\cdots 
}
\end{align*}
and it remains to take the (essentially constant) direct limit
\begin{align*}
\vp=\varinjlim_{n}\vp^{n}\colon P=\varinjlim_{n}\PP^{n}\xto{\simeq}
\varinjlim_{n}F^{n}\cong C 
\end{align*}
to finish the proof.
\end{proof}
We add a few remarks about the essence of this degeneracy criterion.

\begin{remark}
The point is that, under the assumptions made, each differential 
\[
d^{i,n-i}_{r}\colon E^{i,n-i}_{r}\to E^{i+r,n+1-i-r}_{r}
\]
on the later pages $E^{\sbullet,\sbullet}_{r}$ for $r\ge 1$, by
definition a morphism from a \emph{subquotient} of $E^{i,n-i}_{1}$ to
one of $E^{i+r,n+1-i-r}_{1}$, is already induced by the relevant
component of $\partial^{n}\colon  P^{n}=\bigoplus_{i}E^{i,n-i}\to
\bigoplus_{i}E^{i,n+1-i}=P^{n+1}$. Conversely, if there exist such
morphisms $\partial^{n}$ that induce the higher differentials in the
spectral sequence and that satisfy $\partial^{n+1}\partial^{n}=0$,
then projectivity of the $P^{n}$ ensures that the resulting complex is
quasi-isomorphic to $C$, thus, constitutes a projective resolution.

Moreover, the proof shows that the construction of the projective
resolution of $C$ is effective and natural. It suffices to replace
successively the connecting morphisms $F_{\nu}/F_{\nu+1}\to
F_{\nu+1}/F_{\nu+2}[1]$ by the morphisms $P^{\nu}[-\nu]\to
P^{\nu+1}[-\nu-1]$ they induce in cohomology.
\end{remark}

\begin{remark}
\label{remark:canonicalspecialcase}
It seems worthwhile to single out the simplest case. Assume that the
bicomplex $C$ not only satisfies the hypotheses of Proposition
\ref{projresfromSS} but that furthermore there exists for each $n$ at
most one $i'_{n}$ with $i_{n-1}<i'_{n}\le i_{n}$ and
$E^{i'_{n},n-i'_{n}}_{1}\neq 0$. The spectral sequence then
degenerates into a single complex
\begin{align}
\label{oneE1}
\cdots \to E^{i'_{n},n-i'_{n}}_{1}\xto{\partial^{n}}
E^{i'_{n+1},n+1-i'_{n+1}}_{1}\to\cdots\to  
E^{i'_{a},a-i'_{a}}_{1}\to 0
\end{align}
with projective terms that is quasi-isomorphic to $C$, thus constitutes
a projective resolution of $C$ as postulated in Proposition
\ref{projresfromSS}.

The differential $\partial^{n}$ is simply obtained from the
differential $d^{i'_{n},n-i'_{n}}_{r}\colon E^{i'_{n},n-i'_{n}}_{r}\to
E^{i'_{n+1},n+1-i'_{n+1}}_{r}$ on the $r^\text{th}$ page of the spectral
sequence, for $r= i'_{n+1}-i'_{n}$, through the composition
\begin{align*}
\partial^{n}\colon E^{i'_{n},n-i'_{n}}_{1}\onto
E^{i'_{n},n-i'_{n}}_{r}\xto{d^{i'_{n},n-i'_{n}}_{r}} 
E^{i'_{n+1},n+1-i'_{n+1}}_{r}\into
E^{i'_{n+1},n+1-i'_{n+1}}_{1} 
\end{align*}
where the first morphism is necessarily an epimorphism and the last
one a monomorphism as the assumptions guarantee that there are no
nonzero differentials with source equal to $E^{i'_{n},n-i'_{n}}_{r'}$
or with target equal to $E^{i'_{n+1},n+1-i'_{n+1}}_{r'}$ on any
earlier page $r'<r$.
\end{remark}

\section{Direct Images on the Determinantal Variety}
\label{sect:di-det}
We now come to one of the central results. 
\begin{nsit}{The generic morphism.}\label{nsit:genmorph}
In addition to the projective $K$-module $F$ of constant rank $m$,
let $G$ be a second projective $K$-module, of constant rank $n\ge
m$. The $K$-module $H =\Hom_K(G,F)$ is then still projective, of
constant rank $mn$. We view $H$ as the affine $K$-variety of
$K$-rational points of $S = \Sym_K(H^{\svee})$, locally isomorphic to
a polynomial ring over $K$ in $mn$ variables and naturally graded by
the symmetric powers which are in turn finite projective $K$-modules.

The projective $K$-modules $F$ and $G$ extend under $-\otimes S$ to
projective $S$-modules $\F$ and $\G$ respectively\footnote{If one
wishes to keep track of the $S$-grading, $\G$ should be identified
with the graded $S$-module $G\otimes S(-1)$ generated in degree $1$,
while $\F=F\otimes S$ is generated in degree $0$.}. The evaluation
homomorphism $\Hom_{K}(G,F)\otimes G\to F$ yields by adjunction the
$K$-linear inclusion $G\ \into\ F\otimes H^{\svee}\subseteq F\otimes
S$ that induces the \emph{generic morphism} $\vp\colon \G\to \F$ between
these projective $S$-modules. Taking the $m^\text{th}$ exterior power over
$S$ and using that $|\F|=\L^{m}_{S}\F$ is invertible with inverse
$|\F^{\svee}|=\L^{m}_{S}\F^{\svee}= |F^{\svee}|\otimes S$, there
results an $S$-linear form
\begin{align*}
\L^{m}_{S}\G\otimes_{S}|\F^{\svee}| \xto{(\L^{m}_{S}\vp)
  \otimes_{S}1} |\F|\otimes_S|\F^{\svee}| \to S 
\end{align*}
whose image is the defining ideal of the locus where the generic
morphism drops rank and whose cokernel we denote by $R$. Locally,
$\Spec R$ is described by the vanishing of the maximal minors of the
generic $(m\times n)$-matrix.  The $K$-algebra $R$ inherits the
grading from $S$, and its graded components are still finite
projective $K$-modules, as follows from the classical
Gaeta-Eagon-Northcott complex \cite{Gaeta, EN} that resolves $R$
projectively as an $S$-module.  In particular, $R$ is a perfect
$S$-module of grade equal to $n-m+1$.  The singular locus of $\Spec R$
is locally defined by the \emph{submaximal} minors $I_{m-1}(X)$,
whence has codimension $n-m+3$ in $\Spec R$.  In particular, $R$ is
smooth in codimension $2$, a fact we shall exploit below.

Recall as well that $\pi\colon\PP\to\Spec K$ denotes the structure
morphism from the projective space $\PP = \PP(F^\svee) \cong
\PP^{m-1}$ of $K$-linear forms on $F$ to the base scheme.

Set $\caly = \PP \times_{\Spec K} H$, with the canonical projections
$p\colon \caly \to \PP$ and $q\colon \caly \to H$. Note that $q$ can
be identified with $\pi\times_{\Spec K}H$, whence we may view it as
the structure map of the projective bundle $\caly \cong
\Proj_{H}(\F^{\svee})\to H$. In particular, the results of the
Section~\ref{sect:higherdirect} apply, if one replaces there $K$ by
$H$ and $F$ by $\calf$.
\end{nsit}

\begin{nsit}{The incidence variety and its resolution.}
\label{nsit:resRjOZ}
Define as in the Introduction the \emph{incidence variety} 
\[
\calz = \lbrace \,([\lambda],\theta) \in \PP\times_{\Spec K}
H\;|\;\lambda\theta = 0\,\rbrace  \subseteq \caly
\]
and denote by $j$ the natural inclusion $\calz \to \caly$. The
composition $q' = q j\colon \calz \to H$ is then a birational
isomorphism from $\calz$ onto its image $q'(\calz) = \Spec R$, while
$p' =p j\colon \calz\to \PP$ is a vector bundle (with zero
section $\theta=0$). In particular, $p'$ is smooth, thus flat.

The vector bundle $\calz$ admits a compact description in terms of the
bundle of differential forms $\calu =\Omega_\PP^1(1)$.  Since an
element of the fiber $\Omega^1(1)_\lambda$ over a closed point
$\lambda \in \PP$ sits in an exact sequence
\[
0 \to \Omega^1(1)_\lambda \to F \to K \to 0\,,
\]
we obtain a closed
point of $\calz$ by tensoring with $G^\svee$:
\[
0 \to \Omega^1(1)_\lambda \otimes G^\svee \to F \otimes G^\svee \to
G^\svee \to 0
\]
and see thereby that 
\begin{equation}
\label{eq:zdescription}
\calz \cong \underline{\Spec}\left(\Sym_{\PP(F^\svee)}(\Omega^1(1)^\svee
\otimes G)\right)\,.
\end{equation}

The morphism $j\colon \calz \to \caly$ is a \emph{regular immersion}
of codimension $n$, zero-locus of the cosection
\begin{equation}
\label{eq:cosection}
\Phi\colon  q^{*}\G \to p^{*}(\calop(1)) = \caloy(1)
\end{equation}
which corresponds by adjunction to the generic morphism $\G\to q_{*}
\caloy(1)\cong \F$ and is determined locally through
\[
\Phi(q^{*}g_j) = \sum_{i=1}^m f _i \otimes x_{ij}\,.
\]
Put differently, the $S$-module of sections of $\calop(1)\otimes
\calo_H$, isomorphic to $\F$, contains the $K$-linear subspace $F
\otimes H^{\svee}$ and this subspace in turn contains $G$ canonically.
Then $\calz$ is the complete intersection in $\caly=\PP \times_{\Spec
K} H$ given locally by a basis of $n$ sections of $G \subseteq
\Gamma(\caly, \calop(1) \otimes \calo_H)$.

Accordingly, the direct image $j_*\caloz$ is resolved by locally free
$\caloy$-modules through the Koszul complex
\begin{equation}
\label{eq:koszulcomplex}
j_*\caloz \simeq \left(\textstyle{\L}_\caly (q^{*}\G \otimes_\caloy
p^{*}\calop(-1)), \partial_{\Phi(-1)}\right)\,.
\end{equation}
on the $\caloy$-linear form $\Phi(-1)$.  As $j$ is a finite morphism,
indeed a closed immersion, $j_*\caloz$ represents already the total
direct image $\R j_*\caloz$.

\end{nsit}

We now analyse the higher direct images
$(\R^{\nu}q'_{*})p'^{*}(\calm^{b}_{a}(-c))$, using the degeneracy
criterion from the foregoing section. As before, in view of Lemma
\ref{lem:iden}, it suffices to treat the case $a+b\ge m+1$.

\begin{theorem}
\label{thm:hdi}
With $\calm^{b}_{a}(-c)=
\cHom_{\calop}(\Omega^{b-1}(b),\Omega^{a-1}(a))(-c)$ as before, the
complex $(\R^{\sbullet}q'_{*})p'^{*}(\calm^{b}_{a}(-c))$ admits a
projective resolution in $\cald(S)$ by a perfect complex that is supported
on $[-n,m-1]\subseteq\ZZ$ and of amplitude at most $n$.

The higher direct images $(\R^{\nu}q'_{*})p'^{*}(\calm^{b}_{a}(-c))$ 
with $\nu\neq 0$ vanish as soon as
\begin{align*}
c&\le 0\quad\text{or}\quad \text{$c=1$ and $b=m$ or
  $a=1$}\quad\text{or}\quad \text{$c=2, b=m$ and $a=1$.} 
\end{align*}
In these cases, the direct image $q'_{*}p'^{*}(\calm^{b}_{a}(-c))$
admits a resolution
\begin{align*}
0 \to P^{-d} \to \cdots \to P^{0} \to q'_{*}p'^{*}(\calm^{b}_{a}(-c)) \to 0
\end{align*}
by finite projective $S$-modules $P^{\mu}$. 

For $a+b\ge m+1$, the non-vanishing projective modules $P^{\mu}$ are of
the form

\begin{align}
\label{table:mus}
\begin{array}{|c|c|c|}
\hline \mu& P^{\mu}&c\\
\hline 
[m-n-1, c-2]&\R^{m-1}\pi_{*}\calm^{b}_{a}{\scriptstyle(-c+\mu-m+1)}
\otimes_{S}\L^{m-1-\mu}\G&\ge m-n+1
\vphantom{\displaystyle\frac{1}{2}}\\
c-1&\bigoplus_{k=0}^{\min\lbrace m-a,m-b\rbrace }\L^{b+k}\F^{\svee}\otimes_{S}\L^{a-c+k}\G&
\ge a-n  
\vphantom{\displaystyle\frac{1}{2}}\\
c&\bigoplus_{k=\max\lbrace a,b\rbrace }^{m}\L^{k-a}\F^{\svee}\otimes_{S}\L^{k-b-c}\G&
[\max\lbrace a-b-n,-n\rbrace ,0]
\vphantom{\displaystyle\frac{1}{2}}\\
{[c+1,0]}&\pi_{*}\calm^{b}_{a}{\scriptstyle(-c+\mu)}\otimes_{S}\L^{-\mu}\G&
[-n, 0]
\vphantom{\displaystyle\frac{1}{2}}\\
{[-n,0]}&\pi_{*}\calm^{b}_{a}{\scriptstyle(-c+\mu)}\otimes_{S}\L^{-\mu}\G&
< -n
\vphantom{\displaystyle\frac{1}{2}}\\
\hline
\end{array}
\end{align}
\medskip
\medskip

\noindent
Accordingly, the projective dimension $d$ of
$q'_{*}p'^{*}(\calm^{b}_{a}(-c))$, with $a+b\ge m+1$, is given by
\smallskip

\begin{align*}
\begin{array}{|c|c|c|}
\hline d& c &(a,b)\\
\hline n-m+1& 2 & (1,m)\\
 n-m+1& 1 & b=m\\
n-m+1&[m-n,0]&\\
-c+1&[a-n,m-n]&\\
-c&[a-b-n,a-n-1]&a > b\\
-c&[-n,a-n-1]&a \le b\\
-c-1&[-n,a-b-n-1]&a > b\\
 n&<-n&\\
 \hline
\end{array}
\end{align*}
\smallskip

\noindent
In particular, for arbitrary integers $a,b,c$, the $S$-module
$q'_{*}p'^{*}(\calm^{b}_{a}(-c))$ is \emph{perfect} of grade equal to
$n-m+1$ for
\begin{align*}
&\text{$c=m-n-1$ and $a=m$ or $b=1$,}\quad\text{or}\\
&m-n\le c\le 0,\quad\text{or}\\
&\text{$c=1$ and $b=m$ or $a=1$,}\quad\text{or}\\
&\text{$c=2$ and $b=m$ and $a=1$.}
\end{align*}
\end{theorem}

\begin{proof}
Observe that $q'_{*}p'^{*}=q_{*}j_{*}j^{*}p^{*}$, whence we can
calculate the desired derived direct image as
\begin{align*}
(\R q'_{*})p'^{*}(\calm^{b}_{a}(-c)) 
\simeq 
\R q_{*}(\R j_{*}(j^{*}p^{*}\calm^{b}_{a}(-c)))\,.
\end{align*}
To evaluate the term on the right, we have first 
\[
j^{*}p^{*}\calm^{b}_{a}(-c) \cong
p^{*}\cHom_{\calop}\left(\Omega^{b-1}(b),\Omega^{a-1}(a)\right)(-c)\otimes_{\caloy}\caloz 
\]
and then
\begin{align*}
\R j_{*}(j^{*}p^{*}\calm^{b}_{a}(-c)) & 
\cong p^{*}\cHom_{\calop}\left(\Omega^{b-1}(b),\Omega^{a-1}(a)\right)(-c)\otimes_{\caloy}
\R j_{*}\caloz
\end{align*}
by the projection formula, as $p^{*}\calm^{b}_{a}(-c)$ is locally free
on $\caly$.  Replacing $\R j_{*}\caloz$ by its locally free
$\caloy$-resolution described in \ref{nsit:resRjOZ} above, we find
that $\R j_{*}(j^{*}p^{*}\calm^{b}_{a}(-c))$ is represented in the
derived category of $\caly$ by a (chain) complex $C$ with terms
\[
C^{-i} =
p^{*}\cHom_{\calop}\left(\Omega^{b-1}(b),\Omega^{a-1}(a)\right)(-c-i)\otimes_{\caloy}  
p^{*}\pi^{*}\L^{i}G\quad,\quad i=0,\ldots,n;
\]
concentrated on the interval $[-n,0]$. We can now determine the higher
direct images under $q_{*}$ of $\R j_{*}p'^{*}\calm^{b}_{a}(-c)$ by
means of the hypercohomology spectral sequence defined by this
complex. The first page $E^{i,j}_{1}$ of that spectral sequence is
concentrated in the second quadrant, supported on the rectangle
$[-n,0]\times[0,m-1]$ in the $(i,j)$-plane, with
\begin{align*}
E^{i,j}_{1}= \R^{j} q_{*}(C^{i}) \Longrightarrow
\R^{i+j}q'_{*}(p'^{*}\calm^{b}_{a}(-c))\,. 
\end{align*} 
Using the projection formula once more and noting that taking (higher)
direct images commutes with flat base change, we obtain next that
\[
E^{i,j}_{1} = \R^{j} q_{*}(C^{i})
\cong \R^{j}\pi_{*}(\calm^{b}_{a}(i-c))\otimes_S \L^{-i}\G\,.
\]
In view of Theorem~\ref{thm:directimage}, for fixed $i\in [-n,0]$,
there is at most one index $j$ for which $E^{i,j}_{1}$ is not zero,
and these terms are finite projective $S$-modules. In particular, the
assumptions of Proposition~\ref{projresfromSS} are satisfied and the
hypercohomology spectral sequence degenerates into a projective
resolution of $(\R q'_{*})p'^{*}(\calm^{b}_{a}(-c))$.

The first page of the spectral sequence is concentrated in total
degrees $[-n,m-1]$, with at most $n$ degrees supporting non-zero terms,
whence the claims about support and amplitude of the projective
resolution follow.

For the detailed analysis of the projective resolutions we exhibit
their terms by means of Theorem \ref{thm:directimage}. Recall that we
assume as there that $a+b\ge m+1$. We proceed by cases.

\begin{enumerate}[\ \rm (1)]
\item
\label{hdi:case1} 
For (total) degree $\mu \le c-2$, Theorem \ref{thm:directimage}, or
Corollary \ref{cor:van}, shows that in the direct sum
\begin{align*}
P^{\mu} &=  \bigoplus_{j=0}^{m-1}E^{\mu-j,j}_{1}
\cong
\bigoplus_{j=0}^{m-1}\R^{j}\pi_{*}(\calm^{b}_{a}(\mu-j-c)) 
\otimes_S \L^{j-\mu}\G\,, 
\end{align*}
only the highest occurring direct image
$\R^{m-1}\pi_{*}(\calm^{b}_{a}(\mu-m+1-c))$ can possibly be non-zero,
\begin{align*}
P^{\mu} &= E^{\mu-m+1,m-1}_{1}\cong
\R^{m-1}\pi_{*}(\calm^{b}_{a}(-c+\mu-m+1))\otimes_S \L^{m-1-\mu}\G\,.
\end{align*}
Moreover, the first factor in the tensor product is indeed non-zero if,
and only if, $c-\mu-1>0$ and the other factor is clearly non-zero if,
and only if, $0\le m-1-\mu\le n$.  This yields $P^{\mu}\neq 0$ exactly
for
\begin{align*}
m-n-1\le \mu \le \min\lbrace m-1,c-2\rbrace \,.
\end{align*}

\item
\label{hdi:case2}  
In total degree $i+j=c-1$, Theorem \ref{thm:directimage}(4) yields
\begin{align*}
P^{c-1}&= \bigoplus_{j=0}^{m-1}E^{i,j}_{1}
\cong
\bigoplus_{j=0}^{m-1}\R^{j}\pi_{*}(\calm^{b}_{a}(-j-1))
\otimes_S \L^{j+1-c}\G\\ 
&\cong
\bigoplus_{j+1=a}^{\min\lbrace a-b+m,m\rbrace }
\F^{\svee}_{b-a+j+1}\otimes_{S}\L^{j+1-c}\G\\
&=\bigoplus_{k=0}^{\min\lbrace m-b,m-a\rbrace }\F^{\svee}_{b+k}\otimes_{S}\L^{a+k-c}\G
\end{align*}
and the second factor in the tensor product is non-zero if, and only
if, $0\le a+k-c\le n$.  Combined with the range $0\le k \le
\min\lbrace m-b,m-a\rbrace $ of the summation, $P^{c-1}$ is thus seen to be
nonzero if, and only if,
\begin{align*}
\max\lbrace c-a,0\rbrace \le \min\lbrace m-a,m-b,n-a+c\rbrace \,,
\intertext{equivalently,}
\max\lbrace c,a\rbrace \le \min\lbrace m,m+a-b,n+c\rbrace \,.
\end{align*}
In case $c\le a$, this condition just becomes $a-n\le c$ as $m-a,m-b\ge 0$.

\item
\label{hdi:case3}  
In total degree $i+j=c$ we obtain from Theorem
\ref{thm:directimage}(3) that
\begin{align*}
P^{c}&= \bigoplus_{j=0}^{m-1}E^{i,j}_{1}\cong
\bigoplus_{j=0}^{m-1}\R^{j}\pi_{*}(\calm^{b}_{a}(-j)) 
\otimes_S \L^{j-c}\G\\
&\cong
\bigoplus_{j=\max\lbrace 0,a-b\rbrace }^{m-b}\F^{\svee}_{b-a+j}\otimes_{S}\L^{j-c}\G\\ 
&=\bigoplus_{k=\max\lbrace a,b\rbrace }^{m}\F^{\svee}_{k-a}\otimes_{S}\L^{k-b-c}\G\,.
\end{align*}
Taking into account that the second factor in the tensor product is
nonzero if, and only if, $0\le k-b-c\le n$ and comparing this with the
range of the summation, it follows that $P^{c}\neq 0$ if, and only if,
\begin{align*}
\max\lbrace a,b,b+c\rbrace  \le \min\lbrace m,n+b+c\rbrace \,.
\end{align*}
If $c\le 0$, this inequality becomes equivalent to
$\max\lbrace a-b-n,-n\rbrace \le c\le 0$.  Note also that $P^{c}= 0$ for $b+c >
m$.

\item
\label{hdi:case4}  
Finally assume that the total degree satisfies $\mu=i+j >c$. In that
case thus $c-i < j$ and Corollary~\ref{cor:van}
shows $E^{i,j}_{1}=0$ for $j\neq 0$, whence
\begin{align*}
P^{\mu}&= E^{\mu,0}\cong \pi_{*}(\calm^{b}_{a}(\mu-c))\otimes_S
\L^{-\mu}\G
\end{align*}
In turn, this term is non-zero if, and only if, $\max\lbrace -n,c+1\rbrace  \le \mu\le 0$.
\end{enumerate}

It remains to exhibit when $P^{\mu}\neq 0$ for some $\mu >0$. By case
(\ref{hdi:case1}), this will occur if $0 < \min\lbrace m-1,c-2\rbrace $, thus, for
$c>2$ (and $m\ge 2$).

If $c=2$, then $P^{1}= P^{c-1}\neq 0$ if, and only if, $\max\lbrace 2,a\rbrace \le
\min\lbrace m,m-b+a, n+2\rbrace $ by case (\ref{hdi:case2}) above. As we always
have $\max\lbrace 2,a\rbrace  \le m < n+2$ and $a\le m-b+a$, the inequality fails
only for $a=1, b=m$. In the latter case, indeed each $P^{\mu}=0$ for
$\mu >0$.

If $c=1$, then the above results yield immediately $P^{\mu} = 0$ for
all $\mu >1 $, and case (\ref{hdi:case3}) shows that $P^{1}=0$ if, and
only if, $\max\lbrace a,b,b+1\rbrace  = \max\lbrace a,b+1\rbrace  > \min\lbrace m,n+b+1\rbrace =m$, which
in turn holds if, and only if, $b=m$.
\end{proof}

\begin{remark}\label{rem:backhanded}
For any $n\ge m$, a projective resolution for
$q'_{*}p'^{*}\calm^{b}_{a}$ cannot be shorter than
displayed. Inspecting $P^{-n+m-1}$, this implies, in a backhanded way,
that $\R^{m-1}\pi_{*}\calm^{b}_{a}(-c)\neq 0$ for each $c > m$, whence
also $\pi_{*}\calm^{b}_{a}(-c)\neq 0$ for $c<0$.
\end{remark}

\begin{example}\label{eg:presOmega-a}
To derive an $S$-presentation for $q'_\ast
p^{\prime\ast}\Omega^{a-1}(a)$, consider that we have 
\begin{align*}
\calm^m_a& =
\cHom_{\calo_\PP}\left(\Omega^{m-1}(m),\Omega^{a-1}(a)\right)\\  
&=\cHom_{\calo_\PP}\left(\L^{m}F\otimes\calop,\Omega^{a-1}(a)\right)\\
&=\L^{m}F^\svee \otimes \Omega^{a-1}(a)
\end{align*}
and hence
\begin{equation}
\label{eq:tmrel}
\Omega^{a-1}(a)=|F|\otimes \calm^m_a\,,
\end{equation}
where $|F|$ is the determinant of $F$; recall~\ref{diffforms}.  From
the second and the third line of the table~\eqref{table:mus}, applied
with $c=0$, we find that $q'_\ast p^{\prime\ast}\calm^m_a$ has a
presentation
\[
\xymatrix{
\L^{m} \calf^\svee \otimes_S \L^a\calg \ar[r] &
\L^{m-a}\calf^\svee \ar[r] &
q'_\ast p^{\prime\ast}\calm^m_a \ar[r] &
0\,.
}
\]
Tensoring with $|F|$ we get a presentation
\begin{equation*}
\xymatrix{
\L^a\calg   \ar[r]^\rho  &
\L^a\calf   \ar[r] &
q'_\ast p^{\prime \ast} \Omega^{a-1}(a)   \ar[r] &
0\,.
}
\end{equation*}
We confirm the identity of $\rho$ below in Theorem~\ref{thm:geometric}.
\end{example}

\section{From Algebra to Geometry}\label{sec:geometricmethods}
We now use the homological results from
Sections~\ref{sect:higherdirect}--\ref{sect:di-det} to prove the
results asserted in the Introduction. 

\begin{nsit}{The non-commutative desingularization.}\label{nsit:defM}
We retain the notations from~\ref{nsit:genmorph}
and~\ref{nsit:resRjOZ}, but from now on $K$ will always be a
\emph{field.}  As in the Introduction, we put 
\[
M_a=\cok\L^a_S\phi
\]
for $1\leq a \leq m$, and $M = \bigoplus_a M_a$.  Set $E = \End_R(M)$,
our intended non-commutative desingularization of $\Spec R$.
\end{nsit}

First we obtain a geometric description of $M_a$.

\begin{theorem}
\label{thm:geometric}
There is an isomorphism $ q'_\ast (p^{\prime\ast}\Omega^{a-1}(a))\cong
M_a$, which fits in the following commutative diagram
\[
\xymatrix{
q'_\ast q^{\prime \ast} \L^a\calf \ar@{->>}[r] & q'_\ast
    p^{\prime\ast}\Omega^{a-1}(a) \\
\L^a\calf  \ar@{->>}[u] \ar@{->>}[r] & 
M_a \ar[u]_{\cong}
}
\]
where the leftmost vertical map is the canonical one, the lower horizontal
map comes from the definition of $M_a$, and the upper horizontal map
is derived from the exact sequence in~\ref{eq:defOmega-a}.
\end{theorem}
\begin{proof}
Let $i\colon \Spec R\to \Spec S$ be the inclusion. We will construct a
more elaborate version of the claimed diagram
\begin{equation}
\label{eq:bigdiagram}
\xymatrix{
i^\ast q'_\ast\L^a q^{\prime \ast}(\calg)  
    \ar[rr]^{i^\ast(q_\ast\bigwedge^a q^{\prime\ast}(\phi))}   &&
i^\ast q'_\ast\L^a q^{\prime \ast}(\calf)
    \ar[r] &
i^\ast q'_\ast (p^{\prime\ast}\Omega^{a-1}(a))
    \ar[r] &
0
\\
i^\ast\L^a\calg
    \ar[rr]_{i^\ast(\bigwedge^a\phi)} \ar[u]^{\cong}  &&
i^\ast\L^a\calf
    \ar[r] \ar[u]_{\cong}  &
i^\ast M_a 
    \ar[r] \ar[u]_{\cong} &
0
}
\end{equation}
where the two leftmost vertical maps are the canonical ones.

For brevity we will drop below most of the applications of $i^\ast$
from the notations.

Let $H_0\subset H$ be the locus where the rank of $\phi$ is exactly
$m-1$ and put $\calz_0=(q')^{-1}(H_0)$.  Then $q'$ restricted to
$\calz_0$ is an isomorphism.

The map $\phi\colon \calf\to \calg$ pulls back to a map $q^{\prime
\ast} (\phi)\colon q^{\prime \ast}(\calg) \to q^{\prime \ast}
(\calf)$. By looking at fibers it is easy to see that it factors as
\begin{equation} \label{epimono}
\xymatrix{
q^{\prime\ast}(\phi)\colon \quad
q^{\prime \ast}(\calg) \ar[r] &
p^{\prime\ast}(\Omega(1)) \ar@{^{(}->}[r] &
p^{\prime\ast}\pi^\ast F = q^{\prime \ast}(\calf)\,.
}
\end{equation}
Since the exterior product preserves subbundles we get an
factorization
\[
\xymatrix{
\L^a q^{\prime\ast}(\phi)\colon\quad 
\L^a q^{\prime \ast}(\calg) \ar[r] &
p^{\prime\ast}(\Omega^a(a))\ar@{^{(}->}[r] &
\L^a p^{\prime\ast}\pi^\ast F  = \L^a q^{\prime \ast}(\calf)
}
\]
and hence combining this with the pullback of a suitably shifted
version of \eqref{eq:defOmega-a} under $p'$
we get a complex
\begin{equation}\label{complex}
\xymatrix{
\L^a q^{\prime \ast}(\calg) \ar^{\bigwedge^a q^{\prime\ast}(\phi)}[rr] &&
\L^a q^{\prime \ast}(\calf) \ar[r] &
p^{\prime \ast}(\Omega^{a-1}(a)) \ar[r] &
0\,.
}
\end{equation}
Since the first map in \eqref{epimono} is an epimorphism when restricted
to $\calz_0$ and since exterior powers also preserve epimorphisms, we get
that \eqref{complex} is exact when restricted to $\calz_0$. 

It follows that we have a complex
\begin{equation}\label{eq:exactonH0}
\xymatrix{
q'_\ast\L^a q^{\prime \ast}(\calg) 
  \ar^{q'_\ast\bigwedge^a  q^{\prime\ast}(\phi)}[rr]  &&
q'_\ast\L^a q^{\prime \ast}(\calf) \ar[r]  &
q'_\ast p^{\prime\ast}\Omega^{a-1}(a) \ar[r] & 
0
}
\end{equation}
exact on $H_0$.  Comparing this with the right-exact sequence on
$\Spec R$
\[
\xymatrix{
i^\ast\L^a\calg
   \ar^{i^\ast \bigwedge^a\phi}[rr]  &&
i^\ast\L^a\calf \ar[r] &
M_a \ar[r] &
0
}
\]
we obtain \eqref{eq:bigdiagram}, with a uniquely defined rightmost
vertical map.  It remains to show that the vertical maps are
isomorphisms. We will only consider the rightmost one, as the others
are similar but easier.

Since \eqref{eq:exactonH0} is exact on $H_0$ 
we find
\[
 q'_\ast p^{\prime\ast}\Omega^{a-1}(a)\bigl|_{H_0}=M_a\bigl|_{H_0}\,.
\]
Now $q'_\ast p^{\prime\ast}\Omega^{a-1}(a)$ is $R$-torsion free and
$M_a$ is maximal Cohen-Macaulay over $R$ (and hence $R$-reflexive) by
\cite[Corollary 2.6]{Buchsbaum-Rim:1964}. Since the codimension of the
complement of $H_0$ in $\Spec R$ is at least $2$ we obtain that the
induced map $M_a\to q'_\ast p^{\prime\ast}\Omega^{a-1}(a)$ is an
isomorphism.
\end{proof}



\begin{nsit}{A tilting bundle.}\label{sec:tilting} 
Put $\calt_a=p^{\prime \ast}\Omega^{a-1}(a)$ and
 $\calt=\bigoplus_{a=1}^m \calt_a$, bundles on the incidence variety
 $\calz$.  It follows from Theorem~\ref{thm:geometric} that
 $\End_R(q'_\ast \calt) \cong \End_R(M) =E$.  We can now prove
 Theorems~\ref{thm:A} and~\ref{thm:C} from the Introduction.

\begin{theorem}\label{tilting}  
  We have $\calt^{\perp}=0$ in $\cald(\operatorname{Qch}(\calz))$ and
  $\Ext^i_{\calo_\calz}(\calt,\calt)=0$ for $i>0$.  In other words,
  $\calt$ is a classical tilting bundle on $\calz$ in the sense
  of~\cite{Hille-VdB:2007}.
\end{theorem}

\begin{proof} The condition $\calt^{\perp}=0$ follows immediately by
  considering the adjoint pair $(p^{\prime\ast},p'_\ast)$ and the
  fact, due to Be\u\i linson \cite{Beilinson:1978}, that
  $\bigoplus_{a=1}^m\Omega^{a-1}(a)$ is a tilting bundle on
  $\PP(F^\svee)$. The vanishing of $\Ext$ follows from Theorem
  \ref{thm:hdi} applied with $c=0$.
\end{proof}

\begin{theorem}\label{thm:mcm}
We have $E\cong\End_{\calo_\calz}(\calt)$.  Furthermore $E$ is
noetherian on both sides, is finite over its centre, has finite
global dimension and is a maximal Cohen-Macaulay $R$-module.
\end{theorem}

\begin{proof}  Put ${E}'=\End_{\calo_\calz}(\calt)$.
Since $\calt$ is a tilting bundle on $\calz$ we obtain
$\cald^b(\coh(\calz))\cong \cald^b_f({E}')$. Since
$\calz$ is smooth it follows from \cite[Theorem 7.6]{Hille-VdB:2007} that ${E}'$ has finite global
dimension.

From Theorem \ref{thm:hdi}, applied again with $c=0$, it follows that
${E}'$ is Cohen-Macaulay.

We now have maps
\begin{equation}
\label{eq:tmp1}
{E}' = 
\End_{\calo_\calz}(\calt)  \to
\End_{S}(q'_\ast \calt) \cong 
\End_{S}(M)=
E\,.
\end{equation}
The locus where $\phi$ is not an isomorphism has codimension at least
$2$ in both $\Spec R$ and $\calz$, whence \eqref{eq:tmp1} is an
isomorphism in codimension one. Since both source and target of
\eqref{eq:tmp1} are reflexive (the former e.g.\ by \cite[Lemma
4.2.1]{vandenbergh:flops}) we obtain that \eqref{eq:tmp1} is an
isomorphism.
\end{proof}
\end{nsit}

\section{The Quiverized Clifford Algebra}\label{sect:qCliff}
In this section we compute the algebra structure of the
non-commutative de\-sing\-ularization $E$ defined in~\ref{nsit:defM},
giving in particular an explicit description of $E$ as a path algebra
of a certain quiver with relations derived in a natural way from a
Clifford algebra.

\begin{nsit}{Notation.}\label{nsit:notationfield}
Our setting will be as in Section~\ref{sec:geometricmethods}, so in
particular $K$ is a field.  In addition we fix ordered bases $\lbrace f_1,
\dots, f_m\rbrace $ and $\lbrace g_1, \dots, g_n\rbrace $ for $F$ and $G$, and let
$\lbrace \lambda_1, \dots, \lambda_m\rbrace $, $\lbrace \mu_1, \dots, \mu_n\rbrace $ be the
associated dual bases for $F^\svee$ and $G^\svee$.

We again set $S = \Sym_K(\Hom_K(G,F)^\svee) = \Sym_K(F^\svee\otimes
G)$, which is canonically isomorphic to the polynomial ring over $K$
in the variables $x_{ij} = \lambda_i \otimes g_j$. We let $X$ be the
generic $(m\times n)$-matrix with entries $(x_{ij})_{ij}$, so that $X$
is the matrix of the map $\phi$ when expressed in terms of the bases
$\lbrace g_1, \dots, g_n\rbrace $, $\lbrace f_1,\dots, f_m\rbrace $.

By $\Cliff_S(q_\phi)$ we will denote the Clifford algebra over $S$
associated to the quadratic form $q_\phi\colon \F^\svee \oplus \G \to
S$ which is such that $q_\phi(\lambda,g)=\lambda(\phi(g))$.
Concretely $\Cliff_S(q_\phi)$ is the $S$-algebra generated by
$F^\svee$ and $G$ subject to the relations
\begin{align*}
\lambda_i \lambda_j + \lambda_j \lambda_i &= 0 = \lambda_i^2 
&\text{for $i,j = 1, \dots, m$;}\\
g_i g_j + g_j g_i &= 0 = g_i^2  &\text{for $i,j = 1, \dots, n$; and}\\
\lambda_i g_j + g_j \lambda_i &= x_{ij} &\text{for $i=1, \dots,
  m$, $j=1, \dots, n$}\,.
\end{align*}

\end{nsit}

\begin{nsit}{Quivers.}  
  Let $\Gamma$ be a quiver---a directed graph---on finitely many
  vertices $\lbrace 1, \dots, r\rbrace $.  Let $D$ be a commutative ring (below it
  will be $K$ or $S$).  Denote by $\Gamma_{ij}$ the free $D$-module
  with basis the set of paths in $\Gamma$ from vertex~$i$ to vertex~$j$,
  including the trivial path $e_u$ at each vertex $u$. The direct sum
  $D\Gamma = \bigoplus_{i,j} \Gamma_{ij}$ is naturally a $D$-algebra,
  the \emph{path algebra} of $\Gamma$, with multiplication
  $\Gamma_{jk}\otimes\Gamma_{ij} \to \Gamma_{ik}$ given by
  concatenation of paths where possible, and all other products
  trivial.  (Observe the indexing: we write our paths in functional
  order.) The paths $e_u$ are idempotent and $\sum_u e_u$ is the
  identity element in $D\Gamma$, conveniently denoted by~$1$.  Below
  we will also consider quivers with an infinite number of vertices
  (indexed from $-\infty$ to $\infty$).  In that case $D\Gamma$
  does not have a unit element, but the $e_u$ are local units.

Let $I \subseteq D\Gamma$ be a two-sided ideal.  The pair $(\Gamma,
I)$ is called a \emph{quiver with relations,} and the quotient
$D\Gamma/I$ its path algebra with relations. The relations $I$ will
often be understood and dropped from the notation.
\end{nsit}

\begin{nsit}{Quiverization.}
  If $A$ is a $\ZZ$-graded algebra then we define the \emph{infinite
    quiverization}\footnote{The knowledgeable reader will note
we are basically using the formalism of $\ZZ$-algebras here. See e.g.
\cite{Bondal-Polishchuk:1993}.}
 as the bigraded algebra without unit
    $Q_\infty(A)=\bigoplus_{i,j\in \ZZ} A_{j-i}$ with multiplication
    coming from the multiplication in $A$: $A_{k-j}\times
    A_{j-i}\to A_{k-i}$.  The term ``quiverization'' is meant
    to be informal, indicating that $Q_\infty(A)$ can often be
    advantageously represented as a path algebra of a quiver with
    relations on a set of vertices indexed by~$\ZZ$. If $M$ is a
    $\ZZ$-graded $A$-module then we may view $M$ as as left
    $Q_\infty(A)$-module through the action $A_{j-i}\times
    M_i\to M_j$.  We will denote this $Q_\infty(A)$-module by
    $Q(M)$.

For every $i\in \ZZ$ we have $1\in A_0= Q_\infty(A)_{ii}$. This is an
idempotent in $Q_\infty(A)$ which we denote by $e_i$.  The
\emph{quiverization $Q_r(A)$ of order $r$} of $A$ is defined as the
quotient $Q_\infty(A)/ \sum_{i\not\in [1,r]}
Q_\infty(A)e_iQ_\infty(A)$. It is easy to see that $Q(M)$ is a right
$Q_r(A)$-module provided the grading of $M$ is supported only in degrees
$1,\ldots, r$.  We can often represent $Q_r(A)$ naturally by a quiver
with vertices $[1,r]$.

The following lemma is trivial to prove.
\begin{lemma}
\label{lem:quiverizationprinciple}
The functor $M\leadsto Q(M)$ defines an equivalence of categories between
respectively:
\begin{enumerate}
\item The category of graded $A$-modules and the category of graded
  $Q_\infty(A)$-modules $N$ such that $N=\bigoplus_i e_i N$.
\item The category of graded $A$-modules whose support is 
concentrated in degrees $1,\ldots,r$ and the category of $Q_r(A)$-modules. 
\end{enumerate}
\end{lemma}

\end{nsit}
\begin{nsit}{The doubled Be\u\i linson quiver.}\label{sit:extBquiver}
  It is clear that $\Cliff_S(q_\phi)$ is bigraded by $\deg F=(1,0)$,
  $\deg G=(0,1)$. In this paper we consider two induced
  $\ZZ$-gradings. For the first one (labeled ``the $\ZZ$-grading'') we
  put $\deg F^\svee=-1$, $\deg G=1$. For the second one (``the
  $\NN$-grading'') we put $\deg F^\svee=\deg G=1$.

The \emph{quiverized Clifford algebra} on $F^\svee$ and $G$ is defined
as $C=Q_{m}(\Cliff_S(q_\phi))$ with $\Cliff_S(q_\phi)$ considered as
being graded by the $\ZZ$-grading.  Note that $C$ is still naturally
bigraded.

The $S$-algebra $C$ can be represented as the path algebra with relations
over $S$ of the  \emph{doubled Be\u\i linson quiver}:
\[
\Qtilde: \qquad
\Qtildelambdageer
\]
Note that $g_i,\lambda_j$ serve as the label for $m-1$ different
 arrows.  If there is confusion possible then we use notations like $p
 e_u$ or $e_v p$ to indicate explicitly the starting or ending point
 of the path $p$.

The $a,b$ graded piece $C_{ab}$ of $C$ consists of paths from $a$ to
$b$, thus $C_{ab} = e_b C e_a$.

The relations (with coefficients in $S$) on
$\tilde{\Q}$ are directly derived from those of $\Cliff_S(q_\phi)$:
\begin{align*}
\lambda_i \lambda_j + \lambda_j \lambda_i &= 0 = \lambda_i^2 
&\text{for $i,j = 1, \dots, m$;}\\
g_i g_j + g_j g_i &= 0 = g_i^2  &\text{for $i,j = 1, \dots, n$; and}\\
\lambda_i g_j + g_j \lambda_i &= x_{ij} &\text{for $i=1, \dots,
  m$, $j=1, \dots, n$}
\end{align*}
where we use the convention that whenever there are paths in such
relations that are not defined we silently drop them. This means that
the relation of the third type associated to vertex $1$ is in fact
$\lambda_i g_j=x_{ij}$ and the one associated to vertex $m$ is
$g_j\lambda_i =x_{ij}$.

These relations generate an ideal $\mathcal{J}$ in the path-algebra
$S\tilde{\Q}$ and we have $C=S\tilde{\Q}/\mathcal{J}$. 

\medskip

For further reference we note that $C$ has an involution
\begin{equation}
\label{eq:involution}
\lambda_i\mapsto \lambda_i, \quad g_j\mapsto g_j, \quad e_i\mapsto e_{m+1-i}
\end{equation}
which sends $C_{ab}$ to $C_{m+1-b,m+1-a}$. 
\end{nsit}

\begin{remark}
\label{rem:cubic}
If we prefer to do so we may work over the ground field $K$ instead of
over~$S$.  We find $ C=K\tilde{\Q}/\mathcal{J}' $ where $\mathcal{J}'$
is generated by the relations
\begin{align*}
\lambda_i \lambda_j + \lambda_j \lambda_i &= 0 = \lambda_i^2 
&\text{for $i,j = 1, \dots, m$;}\\
g_i g_j + g_j g_i &= 0 = g_i^2  &\text{for $i,j = 1, \dots, n$;}\\
\lambda_k(\lambda_i g_j + g_j \lambda_i) &= (\lambda_i g_j + g_j
\lambda_i)\lambda_k &\text{for $i,k=1, \dots, 
  m$, $j=1, \dots, n$; and}\\
g_l(\lambda_i g_j + g_j \lambda_i) &= (\lambda_i g_j + g_j
\lambda_i)g_l &\text{for $i=1, \dots, 
  m$, $j,l=1, \dots, n$.}
\end{align*}
The isomorphisms between the former presentation of $C$ and this one
are given by
\[
S\tilde{\Q}/\mathcal{J}\to
K\tilde{\Q}/\mathcal{J}'\colon \lambda_i\mapsto \lambda_i, 
g_j\mapsto g_j, x_{ij}\mapsto \lambda_i g_j + g_j \lambda_i
\]
and
\[
K\tilde{\Q}/\mathcal{J}'\to S\tilde{\Q}/\mathcal{J}
\colon 
\lambda_i\mapsto \lambda_i,
g_j\mapsto g_j\,.
\]
It follows that, when considered as a $K$-algebra, $C$ has cubic relations.
\end{remark}

\begin{nsit}{A Clifford action on $M$.}\label{sec:relCE}
We construct a natural map $C \to E=\End_R(M)$.  
To describe a map $C\to E$ we have to put a left $C$-module structure
on $M$, and according to Lemma~\ref{lem:quiverizationprinciple} it is
sufficient to construct an action of $\Cliff_S(q_\phi)$ on $M$.

 An $S$-endomorphism (or equivalently $R$-endomorphism) of $M=
\bigoplus_{a=1}^m M_a = \bigoplus_{a=1}^m \cok(\L^a \phi)$ is obtained
from a pair of morphisms $\alpha,\beta$ rendering the diagram
\begin{equation}\label{eq:matfact}
\xymatrix{
&
 \L \G \ar[r]^{\L \phi} \ar[d]_\beta & 
 \L \F \ar[r] \ar[d]^\alpha & 
 M \ar[r] \ar@{_..>}[d] 
 & 0\\
& 
 \L \G \ar[r]_{\L\phi} & 
 \L \F \ar[r] & 
 M \ar[r]
 & 0
}
\end{equation}
commutative (putting $\L^0\F=\L^0\G=S$ and $\L^0\phi=\id_S$). We
construct such $\alpha,\beta$ as (super-)differential operators on $\L
\F$ and $\L\G$.

\begin{enumerate}
\item For $\lambda \in \F^\svee$, define skew-derivations
  $\partial_\lambda\colon \L \F \to \L \F$ of degree $-1$ in $\F$ by
  (left) contraction $\lambda \Ydown -$; explicitly, for an element
  $f^{1} \wedge \cdots \wedge f^{a} \in \L^a\F$,
\[
\partial_\lambda(f^{1} \wedge \cdots \wedge f^{a}) = 
\sum_{j=1}^a (-1)^{j-1} \lambda(f^{j}) (f^{1} \wedge \cdots
\wedge \widehat{f^{j}} \wedge \cdots \wedge f^{a})\,.
\]
Then $\partial_\lambda$ extends as well to a skew derivation
$\partial_{\lambda\phi}\colon \L \G \to \L \G$.
Putting $(\alpha,\beta)=(\partial_{\lambda},\partial_{\lambda\phi})$
makes~(\ref{eq:matfact}) commute. Denote the induced endomorphism of $M$
again by $\partial_\lambda$.

\item For $g \in \G$, define $\theta_g\colon \L\G \to
  \L\G$ by the exterior multiplication $\theta_g(-) = g \wedge
  -$.  We have an induced map $\theta_{\phi(g)}\colon
  \L\F \to \L\F$. Putting
  $(\alpha,\beta)=(\theta_{\phi(g)},\theta_g)$ makes
  ~(\ref{eq:matfact}) commute. 
 We denote the induced
endomorphism of $M$ also by $\theta_g$.
\end{enumerate}
Write $\partial_i=\partial_{\lambda_i}$, $\theta_j=\theta_{g_j}$. It
is easy to see that we have
\begin{enumerate}
\item $\partial_i\partial_j + \partial_j\partial_i = 0 = \partial_i^2$
  and $\theta_i\theta_j + \theta_j\theta_i = 0 = \theta_i^2$; and
\item $\partial_i\theta_j + \theta_j\partial_i =
  \partial_i(\phi(g_j)) = x_{ij}$.
\end{enumerate}
and hence we have defined an action of $\Cliff_S(q_\phi)$ on $M$. 
\end{nsit}

We will prove below in Theorem~\ref{thm:qCliffisoEndM} that the
morphism $C\to \End_R(M)$ defined by the action above is an
isomorphism, sending $C_{ab}$ to $\Hom_R(M_b,M_a)$.  Our avenue of
proof once more proceeds by translating to geometry, where we define
an action of the Clifford algebra $C$ on the tilting bundle $\calt$.
We prove (Proposition~\ref{prop:actionsagree}) that the two actions
are compatible with the isomorphism $E \cong
\End_{\calo_\calz}(\calt)$ from Theorem~\ref{thm:mcm}, and then 
(Theorem~\ref{thm:expcan}) that this second action gives an
isomorphism $C \to \End_{\calo_\calz}(\calt)$.


\begin{nsit}{An $S$-presentation for $C$.}
In this section we prove a partial technical result (Lemma
\ref{lem:minresCij2}) 
which we will use in the proof of Theorem~\ref{thm:expcan}.
\end{nsit}

\begin{defn}\label{def:Qinfty} With $\lbrace \lambda_1, \dots,
  \lambda_m\rbrace $ and $\lbrace g_1, \dots, g_n\rbrace $ the fixed bases of $F^\svee$
  and $G$, let $\Q^\infty$ be the doubly infinite quiver over $S$
\[
\quiverQinftylabels
\]
with relations
\begin{align*}
\lambda_i \lambda_j + \lambda_j \lambda_i &= 0 = \lambda_i^2 \\
g_i g_j + g_j g_i &= 0 = g_i^2 \\
\lambda_i g_j + g_j \lambda_i &= x_{ij}\,.
\end{align*}
We define $C^\infty=Q(\Cliff_S(q_\phi))$.
Then $C^\infty$ is the $S$-path algebra of $\Q^\infty$ with relations
as above. Of course $C^\infty$ is again naturally graded by
$C^\infty_{ab} = e_b C^\infty e_a$, and surjects onto $C$.
\end{defn}

Verification of the following version of a Poincar\'e-Birkhoff-Witt
(PBW) basis for $C^\infty$ is routine (and follows formally from the
existence of a similar basis for $\Cliff_S(q_\phi)$). Recall that we
write paths in $\Q_\infty$ in functional order.

\begin{lemma}\label{lem:Cinftyfree}
The algebra $C^\infty$ is free as an $S$-module.  More precisely, a
basis for the graded piece $C^\infty_{ab}$ consists of paths
\begin{equation}\label{eq:pathCinfty}
e_b \lambda_{\beta_b}\lambda_{\beta_{b+1}}\cdots\lambda_{\beta_l}
g_{\alpha_l}g_{\alpha_{l-1}}\cdots g_{\alpha_a} e_a
\end{equation}
with $\alpha_a > \alpha_{a+1} > \cdots > \alpha_{l}$ and $\beta_l <
\beta_{l-1} < \cdots < \beta_{b}$.  
\end{lemma}
We will refer to writing an element of $C^\infty$ in terms of this
basis as the ``PBW expansion for the ordering 
$\lambda_m < \cdots < \lambda_1<
g_1 < \cdots < g_n$.
There is a similar PBW expansion
with the roles of $g_i$,$\lambda_j$ reversed.

\begin{prop}\label{prop:Disfree}  Let $D$ be the kernel
  of the surjection $C^\infty \to C$.  The graded piece $D_{ab}$
  is $S$-generated by two types of paths: those leaving
  $[1,m]$ to the right
\begin{equation}\label{eq:DpathR}
e_b \lambda_{\beta_b}\lambda_{\beta_{b+1}}\cdots\lambda_{\beta_l}
g_{\alpha_l}g_{\alpha_{l-1}}\cdots g_{\alpha_a} e_a
\end{equation}
with $l> m$, $\alpha_a > \alpha_{a+1} > \cdots > \alpha_{l}$, and
$\beta_l < \beta_{l-1} < \cdots < \beta_{b}$; and those leaving
$[1,m]$ to the left
\begin{equation}\label{eq:DpathL}
e_b g_{\alpha_b}g_{\alpha_{b-1}}\cdots
g_{\alpha_l}\lambda_{\beta_l}\lambda_{\beta_{l+1}}\cdots\lambda_{\beta_a}e_a
\end{equation}
with $l<1$, $\beta_a < \beta_{a-1} < \cdots < \beta_l$, and
$\alpha_l> \alpha_{l+1} > \cdots > \alpha_b$.
\end{prop}

\begin{proof} 
We need to prove that the paths \eqref{eq:DpathR} and
\eqref{eq:DpathL} generate $D_{ab}$.  To this end, we claim that with
the natural identifications $\L^k\F^\svee \subseteq C^\infty_{a,a-k}$
and $\L^k\G \subseteq C^\infty_{a,a+k}$ in mind,
\begin{equation}\label{eq:leavingtotherightisstable}
C^\infty_{lb} \cdot C^\infty_{al} \subseteq
\sum_{k\geq l} \L^{k-b} \F^\svee \cdot \L^{k-a} \G
\end{equation}
and, symmetrically,
\[
C^\infty_{lb} \cdot C^\infty_{al} \subseteq 
\sum_{k \leq l} \L^{k-b}\G \cdot \L^{k-a}\F^\svee\,.
\]
Indeed, by Lemma~\ref{lem:Cinftyfree}, any element of $C^\infty_{lb}
\cdot C^\infty_{al}$ is a linear combination of paths of the form
$e_b\bm{\lambda}\bm{g}e_l\bm{\lambda}'\bm{g}'e_a$, where
$\bm{\lambda},\bm{\lambda}'$, respectively $\bm{g},\bm{g}'$ represent
products of $\lambda_i$, respectively $g_j$.  The length of the path
$\bm{\lambda}$ is not less than $l-b$, while that of $\bm{g}'$ is not
less than $l-a$.  Applying Lemma~\ref{lem:Cinftyfree} to the product
$\bm{g\lambda}'$ then gives the first containment.  The other
follows similarly.  
\end{proof}

The presentation in Proposition \ref{prop:Disfree} is not minimal for
the $\NN$-grading on $S$. We next give a slightly smaller
presentation, which is sufficient for our proof of
Theorem~\ref{thm:qCliffisoEndM}, even though it is still not minimal.
For the best result see Proposition~\ref{prop:minresCij2}.

\begin{lemma}\label{lem:minresCij2}
The graded piece $C_{ab}$ has an $S$-free presentation of the form
\begin{equation}
\label{eq:prelimpresentation}
Q\oplus P_1 \xto{\rho} P_0 \to C_{ab} \to 0\,,
\end{equation}
where
\begin{itemize}
\item $P_0 = 
\bigoplus_{\max\lbrace a,b\rbrace \le k \le m}
\L_S^{k-b}\F^\svee
\otimes 
\L^{k-a}_S \G$
\item $P_1 = 
\bigoplus_{\begin{smallmatrix} 0\geq l\geq \max\lbrace a-m,b-m\rbrace 
\end{smallmatrix}}
\L^{b-l}_S\G\otimes \L^{a-l}_S \F^\svee $
\item $Q=\bigoplus_{\begin{smallmatrix} \max\lbrace a-m,b-m\rbrace > l\geq \max\lbrace a-m,b-n\rbrace 
\end{smallmatrix}}
\L^{b-l}_S\G\otimes \L^{a-l}_S \F^\svee$
\end{itemize}
and the map $\rho$ is the restriction of the inclusion of $D$ into
$C^\infty$. Furthermore $\ker(\rho\mid_{P_1})\subseteq S_{>0} P_1$.
\end{lemma}

\begin{proof}
  Our starting point is the free presentation of $C_{ab}$ given in
 Proposition \ref{prop:Disfree}. It takes the form (remember once
 again that paths are written in functional order) $e_b D
 e_a\hookrightarrow e_b C^\infty e_a$, where
\[
e_b C^\infty e_a=\bigoplus_{\max\lbrace a,b\rbrace \le k \le \min\lbrace a+n,b+m\rbrace }
\L_S^{k-b}\F^\svee\otimes \L^{k-a}_S \G
\]
and 
\begin{multline*}
e_bDe_a=\left(\bigoplus_{\begin{smallmatrix} m< k\le \min\lbrace a+n,b+m\rbrace 
\end{smallmatrix}}
\L^{k-b}_S\F^\svee\otimes \L^{k-a}_S \G\right)
\\
\oplus
\left(\bigoplus_{\begin{smallmatrix} 1> l\geq \max\lbrace a-m,b-n\rbrace 
\end{smallmatrix}}
\L^{b-l}_S\G\otimes \L^{a-l}_S \F^\svee\right)\,.
\end{multline*}
In the resulting presentation there is some cancellation, which
simplifies things to
\[
\bigoplus_{\begin{smallmatrix} 1> l\geq \max\lbrace a-m,b-n\rbrace 
\end{smallmatrix}}
\L^{b-l}_S\G\otimes \L^{a-l}_S \F^\svee
\to
\bigoplus_{\max\lbrace a,b\rbrace \le k \le m}
\L_S^{k-b}\F^\svee\otimes \L^{k-a}_S \G
\]
which is \eqref{eq:prelimpresentation}.

Now we prove the additional claim of the lemma.  Assume that we have 
\[
\rho\left(\sum_{\alpha,l=\max\lbrace a-m,b-m\rbrace }^0 s_{l,\alpha}
p_{l,\alpha}\right)=0
\] 
with $p_{l,\alpha}\in \L_S^{b-l}\calg\otimes \L^{a-l}_S \calf^\svee$
and $s_{l,\alpha}\in S$. This can be rewritten as an identity in
$C^\infty$
\begin{equation}
\label{eq:initialsum}
\sum_{\alpha,l=\max\lbrace a-m,b-m\rbrace }^0 s_{l,\alpha}
p_{l,\alpha}=\sum_{\beta,k\ge m+1}t_{k,\beta}q_{k,\beta} 
\end{equation}
with $q_{k,\beta}\in \L_S^{k-b}\calf^\svee \otimes
\L^{k-a}_S \calg$ and  
$t_{k,\beta}\in S$. We may assume that the $s_{l,\alpha}$,
$t_{k,\beta}$ are homogeneous 
for the $\NN$-grading.

Choose $l'$ maximal such that there exists $s_{l',\alpha}\neq 0$.  We
have to show that $s_{l',\alpha}\in S_{>0}$ for all $\alpha$
corresponding to this $l'$. Assume on the contrary that there is some
$\alpha'$ such 
that $s_{l',\alpha'}\not\in S_{>0}$.

By our restriction on $l$ we have $b-l\le m$, $a-l\le
m$ in the expression for $p_{l,\alpha}$. Right-multiplying \eqref{eq:initialsum} by a suitable product of
the $\lambda_j$ and left-multiplying by a suitable product of the
$g_i$, we obtain an identity (using
\eqref{eq:leavingtotherightisstable}) of paths starting and ending in
some vertex $v\in [1,m]$
\begin{equation}
\label{eq:noconstantterm}
\sum_{1\le i_1<\ldots<i_m\le n} s'_{i_1\ldots i_m} g_{i_1}\cdots
g_{i_m} \lambda_1\cdots \lambda_m 
=\sum_{\beta'}t'_{\beta'}q'_{\beta'}
\end{equation}
where $s'_{i_1\ldots i_m}\in S$ and at least one $s'_{i_1\ldots i_m}
\notin S_{>0}$,
$t'_{\beta}\in S$ and the $q'_{\beta'}$ are paths
leaving $[1,m]$ to the right as in \eqref{eq:DpathR}.

The PBW expansion of $g_{i_1}\cdots g_{i_m} \lambda_1\cdots \lambda_m$
in terms of paths going first to the right is of the form
\[
\pm [i_1\cdots i_m|1\cdots m]+(\text{an $S$-linear combination of paths
of positive length})
\]
where $[i_1\cdots i_m|1\cdots m]$ is the minor in $X$ with columns
$i_1,\ldots,i_m$.  

Substituting this into \eqref{eq:noconstantterm} and looking at 
constant terms we obtain an identity in $S$:
\[
\sum_{1\le i_1<\cdots<i_m\le n} \pm s'_{i_1\ldots i_m}[i_1\cdots
  i_m|1\cdots m]=0\,.
\]
This is only possible if all $s'_{i_1\ldots i_m}$ are in $S_{>0}$, yielding
a contradiction.
\end{proof}

\begin{nsit}{A Clifford action on the tilting bundle.}\label{sec:cliffordstuff}
Let $\calt = \bigoplus_a \calt_a = \bigoplus_a
p^{\prime\ast}\Omega^{a-1}(a)$ be the tilting bundle on $\calz$
defined in \S\ref{sec:tilting}. In this section we construct
an algebra morphism $C\to\End_{\calz}(\calt)$ which we show to be
an isomorphism afterwards. To construct the morphism it is sufficient
(according to Lemma \ref{lem:quiverizationprinciple}) to construct a
left action of $\Cliff_S(q_\phi)$ on $\calt$.

We have to give the action of the generators. For the action of
$F^\svee$ we use the composition
\begin{equation}
\label{eq:generator1}
\partial\colon F^\svee\otimes\Omega^{b-1}(b)\to \Omega^1(1)^\svee\otimes_{\PP}
\Omega^{b-1}(b)  \to \Omega^{b-2}(b-1)\,,
\end{equation}
where the first map is obtained as the dual of  the canonical map
\[
\Omega^1(1)\to F\otimes \calop
\]
introduced for example in \eqref{eq:defOmega-a}, while the second map is
contraction.

For the $G$-action we use the composition
\begin{multline}
\label{eq:generator2}
\theta\colon G\otimes p^{\prime\ast}\Omega^{b-1}(b) =q^{\prime \ast}\calg
\otimes_{\caloz} p^{\prime\ast}\Omega^{b-1}(b) \to
p^{\prime\ast}\Omega^1(1)\otimes_{\caloz}
p^{\prime\ast}\Omega^{b-1}(b)\\ 
=p^{\prime\ast}(\Omega^1(1)\otimes_{\calop} \Omega^{b-1}(b))
\to p^{\prime\ast}\Omega^{b}(b+1)
\end{multline}
where the first arrow is obtained from the description \eqref{eq:zdescription}
and the second arrow is multiplication.

One checks that the $F^\svee$- and $G$-actions combine to give
the requested action
\begin{equation*}
\Cliff_S(q_\phi)\otimes \calt\to \calt\,.
\end{equation*}
\end{nsit}

\begin{prop}\label{prop:actionsagree}
The morphisms $C \to E = \End_R(M)$ and $C \to
\End_{\calo_\calz}(\calt)$ defined in \S\ref{sec:relCE} and
\S\ref{sec:cliffordstuff} are compatible with the isomorphism
$\End_{\calo_\calz}(\calt) \to E$ of Theorem~\ref{thm:mcm}.
\end{prop}

\begin{proof}
From the construction in \S\ref{sec:relCE} we know that the
constructed action $C_{ab}\otimes M_a \to M_b$ lifts to an action
\begin{equation}
\label{eq:caction1}
\Cliff_S(q_\phi)_{b-a}\otimes \L^a\calf \to \L^b\calf\,.
\end{equation}
Likewise the same types of formulas show that
the action $C_{ab}\otimes p^{\prime\ast}\Omega^{a-1}(a)\to
p^{\prime\ast}\Omega^{b-1}(b)$ lifts to an action
\begin{equation}\label{eq:newaction1}
\Cliff_S(q_\phi)_{b-a}\otimes q^{\prime\ast} \L^a \calf\to 
q^{\prime\ast} \L^b \calf\,.
\end{equation}
It is now easy to see the \eqref{eq:caction1} and
\eqref{eq:newaction1} are compatible, whence the originals are
compatible by Theorem \ref{thm:geometric}.
%
\end{proof}

\begin{theorem}\label{thm:expcan}
The map $C\to \End_{\calo_\calz}(\calt)$ obtained by applying
Lemma~\ref{lem:quiverizationprinciple} to the action constructed in
\S\ref{sec:cliffordstuff} is an isomorphism.
\end{theorem}

\begin{proof}
  We have to show that $C_{ba}\to \Hom_{\calo_\calz}(\calt_b,\calt_a)$
  is an isomorphism.  From Lemma~\ref{lem:duality1} below together
  with the involution $C_{ba} \leftrightarrow C_{m+1-a,m+1-b}$ (see
  \eqref{eq:involution}), we easily deduce that we may assume $a+b\ge
  m+1$. {We make this assumption in the rest of the proof.}

As $S$-modules we have 
\[
\Cliff_S(q_\phi)_{a-b}=\bigoplus_{c} \L^c G\otimes
\L^{b-a+c}F^\svee\otimes S 
=\bigoplus_{c} \L^c \calg \otimes  \L^{b-a+c}\calf ^\svee\,.
\]
We equip $\Cliff_S(q_\phi)_{a-b}$ with a filtration $\mathfrak{F}$
obtained from the value of $c$, that is,
\[
\mathfrak{F}_{u} \Cliff_S(q_\phi)_{a-b}=\bigoplus_{c=0}^{u}
\L^c G\otimes \L^{b-a+c}F^\svee\otimes S 
=\bigoplus_{c=0}^{u}\L^c \calg \otimes  \L^{b-a+c}\calf ^\svee\,.
\]
We will start by proving that the induced map
\begin{equation}
\label{eq:simple1}
\mathfrak{F}_{m-b}\Cliff_S(q_\phi)_{a-b}\to
\Hom_{\calo_\calz}(p^{\prime\ast}\Omega^{b-1}(b), 
p^{\prime\ast} \Omega^{a-1}(a))
\end{equation}
is an epimorphism.

In \S\ref{sit:concrete} we have constructed an action by
$\calop$-linear derivations
\begin{equation}
\label{eq:firstaction}
\partial\colon F^\svee\otimes \KK\to \KK(-1)[1]\,.
\end{equation}
This extends to an action by $\caloz$-linear derivations
\[
\partial\colon F^\svee\otimes p^{\prime \ast}\KK\to p^{\prime
  \ast}\KK(-1)[1]\,.
\]
We produce an additional action 
\begin{equation}
\label{eq:secondaction}
\theta\colon G\otimes p^{\prime \ast}\KK \to p^{\prime \ast}\KK(1)[-1]
\end{equation}
by $p^{\prime\ast}\KK$-linearly extending  the $\KK$-linear map
\begin{align*}
G\otimes \KK
&\to G\otimes \Omega(1)^\svee\otimes_{\calo_\PP} \Omega(1)
\otimes_{\PP}\KK\\
&\to G\otimes \Omega(1)^\svee\otimes_{\calo_\PP}  \KK(1)[-1]\\
&\subset 
\caloz\otimes_{\calo_\PP} \KK(1)[-1] \qquad\qquad  \eqref{eq:defOmega-a} \\
&=p^{\prime \ast}\KK(1)[-1]
\end{align*}
where the second arrow is multiplication in the graded sheaf of
algebras $\KK$ via the inclusion $\Omega^1 \subset \KK^{-1}$. Since
the image of this inclusion consists of closed elements the resulting
multiplication is compatible with the differential. (Note: the
multiplication $F\otimes \KK\to \KK(1)[-1]$ is \emph{not} compatible
with the differential.)

One readily checks that \eqref{eq:firstaction} and
\eqref{eq:secondaction} combine to give an $\caloz$-linear action
\begin{equation}
\label{eq:cliffaction}
\Cliff_S(q_\phi)_s\otimes p^{\prime\ast} \KK \to
p^{\prime\ast} \KK(s)[-s]\,.
\end{equation}
Put $s=a-b$. We obtain an action
\[
\Cliff_S(q_\phi)_{a-b}\otimes  p^{\prime\ast} \KK (b)[-b+1]
\to p^{\prime\ast} \KK(a)[-a+1]
\]
which after truncating in homological degree zero becomes
\[
\Cliff_S(q_\phi)_{a-b}\otimes p^{\prime\ast} \KK_{\le b-1} (b)
\to p^{\prime\ast} \KK_{\le a-1}(a)
\]
so that we finally get a composition
\begin{multline}
\label{eq:simple}
\mathfrak{F}_{m-b}
\Cliff_S(q_\phi)_{a-b}\hookrightarrow\Cliff_S(q_\phi)_{a-b}\to
\\ 
\operatorname{RHom}_{\cald(\calz)}(p^{\prime\ast}\KK_{\le b-1} (b),
p^{\prime\ast} \KK_{\le a-1}(a))\cong
\Hom_{\calo_\calz}(\calt_b,\calt_a)\,. 
\end{multline}
It is easy to check that the second map coincides with the one
obtained from our action of $C$ on $\calt$. We will show
that~\eqref{eq:simple} is an epimorphism.

Using
the same methods as above we may define $\caloy$-linear actions
\begin{equation*}
\begin{split}
\partial&\colon F^\svee\otimes p^{\ast} \KK \to p^{\ast} \KK(-1)[1]\\
\theta&\colon G\otimes p^{\ast} \KK \to p^{\ast} \KK(1)[-1]
\end{split}
\end{equation*}
which are compatible with the natural map $p^\ast \KK\to j_\ast p^{\prime
\ast} \KK$. For example $\theta$ is obtained by extending
\begin{align*}
G\otimes \KK
&\to G\otimes F^\svee\otimes_{K} F
\otimes_{K}\KK\\
& \to G\otimes F^\svee\otimes \KK(1)[-1] \\
&\subset 
S\otimes_{\PP} \KK(1)[-1]\\
&=p^{\ast}\KK(1)[-1]\,.
\end{align*}
 Unfortunately $\theta$ is now not compatible with the
differential. However the commutator
\[
d_{\KK}\theta+\theta d_{\KK}\colon G\otimes p^{\ast} \KK\to
p^{\ast} \KK(1) 
\]
is given by multiplication with the cosection $\Phi\colon G\to
\caloy(1)$ defined in \eqref{eq:cosection}. Written compactly,
\[
d_{\KK}\theta+\theta d_{\KK}=\Phi\,.
\]
Let $\LL= (\textstyle{\L}_\caly (q^{*}\G \otimes_\caloy
p^{*}\calop(-1)), \partial_{\Phi(-1)})$ be the Koszul complex of
locally free $\calo_\caly$-modules resolving $j_*\calo_\calz$ which
was introduced in \eqref{eq:koszulcomplex}. Multiplication by elements
of $G$ defines an action
\[
\tilde{\theta}\colon G\otimes \LL\to \LL(1)[-1]
\]
which is again is not compatible with the differential. However one computes
\[
d_{\LL}\tilde{\theta}+\tilde{\theta} d_{\LL}=\Phi
\]
so that the combined actions
\begin{align*}
\partial\overset{\text{def}}{=}\partial_{13}\colon F^\svee\otimes
(\LL\otimes_{\caloy} p^{\ast} \KK) &\to \LL\otimes_{\caloy}
p^{\ast} \KK(-1)[1]\\ 
\Theta\overset{\text{def}}{=}\pm \tilde{\theta}_{12}\otimes
1+1\otimes\theta_{13}\colon  G\otimes (\LL\otimes_{\caloy} p^{\ast} \KK) 
&\to \LL\otimes_{\caloy} p^{\ast} \KK(1)[-1]
\end{align*}
commute with the total differential on the complex associated to the
double complex $\LL\otimes_{\calop} p^{\ast} \KK$. (Here the
subscripts indicate the factors of the tensor product to which the
maps apply.)

It is easy to see that these actions combine to give an action
\begin{equation}
\label{eq:paction}
\Cliff_S(q_\phi)_{s}\otimes (\LL\otimes_{\caloy} p^{\ast} \KK)
\to \LL\otimes_{\caloy} p^{\ast}\KK(s)[-s]
\end{equation}
which is compatible with the total differential and with the natural map
\[
\LL\otimes_{\caloy} p^{\ast}\KK\to j_\ast\calo_{\calz}\otimes_{\caloy} p^{\ast}\KK =
j_\ast p^{\prime\ast}\KK\,.
\]

Put $s=a-b$. Then \eqref{eq:paction} restricts to
a map
\begin{equation}
\label{eq:restrict}
\Cliff_S(q_\phi)_{a-b}\otimes p^\ast\KK_{\le b-1}(b)
\to \LL\otimes_{\caloy} p^\ast\KK_{\le a-1}(a)
\end{equation}
For $t\in \NN$ and $C$ a complex let $\sigma^{\ge t}C$ denote
the naive truncation of $C$ in cohomological degrees $\ge t$. 
Then \eqref{eq:restrict} restricts again to 
\[
\mathfrak{F}_{m-b}\Cliff_S(q_\phi)_{a-b}\otimes p^\ast\KK_{\le b-1}(b)
\to \sigma^{\ge -(m-b)}\LL\otimes_{\caloy} p^\ast\KK_{\le a-1}(a)
\]
We now obtain a commutative diagram
\[
\xymatrix{
\mathfrak{F}_{m-b}\Cliff_S(q_\phi)_{a-b} \ar[r]^(.325)\alpha \ar@{_(->}[dd]_\beta &
   \Hom_{\calo_\caly}(p^\ast\KK_{\le b-1}(b),\sigma^{\ge -(m-b)}\LL\otimes_{\caloy}
   p^\ast\KK_{\le a-1}(a)) \ar[d]^{\gamma_1} \\
&
   \operatorname{RHom}_{\cald{(\caly)}}(p^\ast\KK_{\le b-1}(b),\sigma^{\ge -(m-b)}\LL\otimes_{\caloy}
   p^\ast\KK_{\le a-1}(a)) \ar[d]^{\gamma_2} \\  
 \Cliff_S(q_\phi)_{a-b} \ar[r] \ar@{=}[d] &
   \operatorname{RHom}_{\cald(\caly)}(p^\ast\KK_{\le b-1}(b),\LL\otimes_{\caloy}
   p^\ast\KK_{\le a-1}(a)) \ar[d]^{\delta} \\  
\Cliff_S(q_\phi)_{a-b}\ar[r]_(.375)\epsilon &
   \operatorname{RHom}_{\cald(\calz)}(p^{\prime\ast}\KK_{\le b-1}(b),
   p^{\prime\ast}\KK_{\le a-1}(a))  
}
\]
where the horizontal arrows are obtained from the Clifford algebra
actions and the vertical arrows are the natural ones. The commutativity
of the lower square follows from the above discussion. 

Looking back at \eqref{eq:simple} we see that we have to show that
$\epsilon\beta$ is an epimorphism on degree zero cohomology. So we
have to show that $\delta\gamma_2\gamma_1\alpha$ is an epimorphism on
degree zero cohomology.

The fact that $\LL$ is a resolution of $j_\ast \caloz$ and formal
adjointness arguments imply that $\delta$ is a quasi-isomorphism (in
fact this is the basis of the proof of Theorem
\ref{thm:hdi}).

We claim that $\gamma_2$ is an epimorphism on degree zero cohomology. To this end we look
at the distinguished triangle
\[
\sigma^{\ge -(m-b)}\LL\otimes_{\caloy}
   p^\ast\KK_{\le a-1}(a))
\rightarrow \LL\otimes_{\caloy}
   p^\ast\KK_{\le a-1}(a))\rightarrow \sigma^{< -(m-b)}\LL\otimes_{\caloy}
   p^\ast\KK_{\le a-1}(a))\rightarrow
\]
It is sufficient to prove that
\[
 \Hom_{\cald(\caly)}\left(p^\ast\KK_{\le b-1}(b), \sigma^{< -(m-b)}\LL\otimes_{\caloy}
   p^\ast\KK_{\le a-1}(a)\right)=0
\]
which in turn follows from
\[
 \Hom_{\cald(\caly)}\left(p^\ast\KK_{\le b-1}(b), \LL^{-c}[c]\otimes_{\caloy}
   p^\ast\KK_{\le a-1}(a)\right)=0
\]
for $c>m-b$. 
To prove this last equation we note that 
\[
\Hom_{\cald(\caly)}\left(p^\ast\KK_{\le b-1}(b), \LL^{-c}[c]\otimes_{\caloy}
p^\ast\KK_{\le a-1}(a)\right)= \L^c G \otimes H^c(\PP,\calm^b_a(-c))
\otimes S
\]
The required vanishing now follows from Theorem
\ref{thm:directimage}(2,4,5).

We also claim that $\gamma_1$ is a quasi-isomorphism. This follows
immediately from \S\ref{sit:E2prop}(e) which states that the
sheaf-Homs between the terms of $p^\ast\KK_{\le b-1}(b)$ and $
\sigma^{\ge (m-b)}\LL\otimes_{\caloy} p^\ast\KK_{\le a-1}(a)$ have no
higher cohomology.

Finally we claim  that $\alpha$ is a
quasi-isomorphism.
To prove this we filter the complex $\Hom_{\calo_\caly}(p^\ast\KK_{\le
  b-1}(b),\sigma^{\ge -(m-b)}\LL\otimes_{\caloy} p^\ast\KK_{\le a-1}(a))$ by
 the degrees in the $\LL$-complex and we equip
$\Cliff_S(q_\phi)$ with the filtration $\mathfrak{F}$ defined above.

Taking associated graded complexes we find that we have to show that
\begin{align*}
\L^c G\otimes \L^{a-b+c} F^\svee\otimes S&\to
\Hom_{\caloy}\left(p^\ast\KK_{\le b-1}(b),\LL^{-c}[c]\otimes p^\ast\KK_{\le
  a-1}(a)\right)\\ &=\L^c G\otimes 
\Hom_{\calop}\left(\KK_{\le b-1}(b),\KK_{\le a-1}(a)(-c)[c]\right)
\otimes S
\end{align*}
is a quasi-isomorphism for $c\le m-b$. One verifies that up to sign
this is in fact the map $\id\otimes \partial\otimes\id$ where
$\partial$ is as defined in \eqref{eq:minidelta}. To finish the proof
that \eqref{eq:simple} is an epimorphism it is now sufficient to
invoke Lemma \ref{lem:concrete}.

\medskip

At this point we know that $C_{ba}\to \Hom_{\calo_\calz}(\calt_b,\calt_a)$
is an epimorphism. We will proceed to show that it is an isomorphism.
With notations as in Lemma~\ref{lem:minresCij2} (swapping $a$ and $b$)
we may construct a commutative diagram
\begin{equation}\label{eq:finaldiagram}
\xymatrix{
P_1\oplus Q \ar[r]^(.6)\rho \ar[d]_{\exists}^{(\alpha_1,\alpha_2)} &  
P_0 \ar[r] \ar@{=}[d] &
C_{ba} \ar@{->>}[d] \ar[r] &
0\\
P_1 \ar[r] & 
P_0 \ar[r] &
\Hom_{\calz}(\calt_b,\calt_a) \ar[r] &
0
}
\end{equation}
where $P_0=\mathfrak{F}_{m-b}\Cliff_S(q_\phi)_{a-b}$ is as in
\eqref{eq:simple1}.  The upper exact sequence is obtained from Lemma
\ref{lem:minresCij2}. The arrow $P_0\rightarrow
\Hom_{\calz}(\calt_b,\calt_a)$ is defined as the composition 
$P_0\rightarrow C_{ba} \rightarrow 
\Hom_{\calz}(\calt_b,\calt_a)$. By the
second and third row of \eqref{table:mus} with $c=0$ (also using 
the assumption $a+b\ge m+1$) we know the minimal resolution
of $\Hom_{\calz}(\calt_b,\calt_a)$, which tells us that we can
complete the lower row as we did. 

Then  the existence of
$(\alpha_1,\alpha_2)$ follows but its properties are a priori unknown.
Nonetheless we claim that $\alpha_1$ must be an isomorphism. Assume this is
not the case. Choose two sets of homogeneous bases
$(x_i)_{i=1,\ldots,N}$, $(y_i)_{i=1,\ldots,N}$ for $P_1$ ordered in
ascending degree. Let $A$ be the matrix of $\alpha_1$ with respect to
these bases. Since $A$ is not invertible, easy degree considerations
show that after change of basis $A$ may be put in the form
\[
\begin{pmatrix}
1      & \cdots &   0    &   0    & A_{1 t+1}  & \cdots & A_{1N}\\
\vdots & \ddots & \vdots & \vdots &  \vdots  &        & \vdots \\
0      & \cdots &  1     &   0    & A_{t,t+1}  & \cdots & A_{tN}\\
0      & \cdots &   0    &   0    & A_{t+1 t+1}  & \cdots & A_{1N}\\
\vdots & \ddots & \vdots & \vdots &  \vdots  &        & \vdots \\
0      & \cdots &  0     &   0    & A_{N,t+1}  & \cdots & A_{NN}\\
\end{pmatrix}\,.
\]
It follows that $P_1$, as a graded $S$-module, may be decomposed as
$P_1=P_1'\oplus P''_1$ with $P''_1\cong S(-u)$ for $u=\deg x_t$ such
that the restriction of $\alpha_1$ to $P''_1$ is zero. It then follows
from~\eqref{eq:finaldiagram} that $\rho\mid_{P''_1}=0$ as well. In
other words $P''_1\subseteq \ker (\rho\mid_{P_1})$.  Since
$P''_1\not\subseteq S_{>0}P_1$ this contradicts Lemma
\ref{lem:minresCij2}.

Hence $\alpha_1$ is an isomorphism and as a result $(\alpha_1,\alpha_2)$
is an epimorphism. Then diagram \eqref{eq:finaldiagram} easily yields
that $C_{ba}\to \Hom_{\calo_\calz}(\calt_b,\calt_a)$ is an isomorphism. 
\end{proof}

The following lemma was used.
\begin{lemma}\label{lem:duality1} 
There is a commutative diagram
\[
\xymatrix{
\Cliff_S(q_\phi)_{a-b} \ar[r] \ar_\alpha[d] & 
 q'_\ast p^{\prime\ast} \calm^b_a \ar[d]^\beta \\
\Cliff_S(q_\phi)_{a-b} \ar[r] & q'_\ast p^{\prime\ast} \calm^{m+1-a}_{m+1-b}
}
\]
where the horizontal maps are those in Theorem~\ref{thm:expcan},
$\alpha$ is obtained from the involution on $\Cliff_S(q_\phi)$ which
is the identity on $\calf^\svee\oplus \calg$, and $\beta$ is obtained
from the isomorphism $\calm^b_a\cong \calm^{m+1-a}_{m+1-b}$ exhibited
in Lemma~\ref{lem:iden1}.
\end{lemma}

\begin{proof}
The isomorphism $\calm^b_a\cong \calm^{m+1-a}_{m+1-b}$ in Lemma
\ref{lem:iden1} is derived from the non-degenerate pairing~\eqref{eq:pairing}
\[
-\wedge-\colon \Omega^{a-1}(a)\otimes_{\calop}
\Omega^{m-a}(-a)\to \Omega^{m-1}
\]
and likewise the induced isomorphism $p^{\prime\ast}\calm^b_a\cong
 p^{\prime\ast}\calm^{m+1-a}_{m+1-b}$ can be obtained from the induced
 pairing
\[
-\wedge-\colon p^{\prime\ast}\Omega^{a-1}(a)\otimes_{\calop}
p^{\prime\ast}\Omega^{m-a}(-a)\to p^{\prime\ast}\Omega^{m-1}
\]
It is therefore sufficient to show that for $\lambda\in F^\svee$ and
$g\in G$ the actions of $\partial_\lambda$ and $\theta_g$ as defined
in \S\ref{sec:cliffordstuff} are self-adjoint for this pairing. 
This is an easy exercise which we leave to the reader.
\end{proof}

Combining the above theorem with Proposition~\ref{prop:actionsagree}
we have the main result of this section.

\begin{theorem}\label{thm:qCliffisoEndM}
The endomorphism algebra $E = \End_R(M)$ is isomorphic to the
quiverized Clifford algebra $C$.
\end{theorem}

\section{The Commutative Desingularization as a Moduli Space}
Having completed the proofs of the statements contained in Theorems
\ref{thm:A}-\ref{thm:C} in the Introduction we now include some
miscellaneous sections. In this section we show that the canonical
commutative desingularization $\calz$ of $\Spec R$ can be obtained as
a fine moduli space for certain representations over the
non-commutative one.

Specifically, we prove in Theorem~\ref{thm:representQtilde} that
$\calz$ represents the functor of flat families of representations
$W$ of $\Qtilde$ which have dimension vector $(1, m-1, \binom{m-1}{2},
\dots, 1)$ and which are generated by $W_m$.  We then identify the
points in $\calz$ corresponding to the simple representations $W$ as
those lying over the non-singular locus of $\Spec R$.
\begin{nsit}{Quiver representations.}
\label{sec:quiverization}
A $K$-representation, $V$, of a quiver $\Gamma$ associates a
(finite-dimensional) $K$-vector space $V_i$ to each vertex $i$ of
$\Gamma$ and a linear map $V(a)\colon V(i) \to V(j)$ for each arrow
$a\colon i\to j$.  A homomorphism $f$ of representations from $V$ to
$V'$ is given by a collections of linear maps for each vertex
$f(i)\colon V(i) \to V'(i)$ so that the obvious diagram commutes.  The
category $\Rep(\Gamma)$ of representations is an abelian category.
The dimension vector of $V$, a function from the vertices of $\Gamma$
to the natural numbers, assigns to $i$ the $K$-rank of $V(i)$.  The
representations of $\Gamma$ with a fixed dimension vector $\theta =
(\theta(i))_i$ are parametrized by the vector space $\prod_{i\to j}
\Hom_K(V(i),V(j))$, and thus the isomorphism classes of
representations $V$ with dimension vector $\theta$ are in one-one
correspondence with the orbits under the action of $\prod_{i}
\GL_{\theta(i)}(K)$.

These notions clearly generalize to the case where $K$ is an arbitrary
commutative ring and each $V(i)$ is a free $K$-module of finite rank.
\end{nsit}

\begin{nsit}{Baby case.}
  The Be\u\i linson algebra associated to a vector space $F$ of rank
  $m$ over the field $K$ is the order-$m$ quiverization (see
  \S\ref{sec:quiverization}) $Q_m(\L F^\svee)$ of the exterior algebra
  of $F^\svee$.

  The Be\u\i linson algebra can be represented as the path algebra of
  the Be\u\i linson quiver
\[
\Q:\qquad
{\xymatrix@C=5pc{
1 &
2 \ar@<2ex>[l]_{\vdots}|{\lambda_1} \ar@<-2ex>[l]|{\lambda_m}  & 
\cdots \cdots \ar@<2ex>[l]_{\vdots}|{\lambda_1} \ar@<-2ex>[l]|{\lambda_m} & 
m \ar@<2ex>[l]_{\vdots}|{\lambda_1} \ar@<-2ex>[l]|{\lambda_m} 
}}
\]
equipped with the anti-commutativity relations $\lambda_i\lambda_j +
\lambda_j\lambda_i = 0 = \lambda_i^2$.  The category
$\Rep(\Q)$ is equivalent to the category of graded left
$\L F^\svee$-modules with support in degrees $1, \dots, m$ (see
Lemma~\ref{lem:quiverizationprinciple}).

For an arbitrary commutative $K$-algebra $A$ we let $\calr(A)$ be the set of
isomorphism classes $W$ of representations of $\mathsf Q$ of the form
\[
W:\qquad
\xymatrix@C=5pc{
W_1 &
W_2 \ar@<2ex>[l]_{\vdots}|{\lambda_1} \ar@<-2ex>[l]|{\lambda_m}  & 
\cdots \cdots \ar@<2ex>[l]_{\vdots}|{\lambda_1} \ar@<-2ex>[l]|{\lambda_m} & 
W_m \ar@<2ex>[l]_{\vdots}|{\lambda_1} \ar@<-2ex>[l]|{\lambda_m} 
}
\]
such that each $W_a$ is a projective $A$-module of rank
$\binom{m-1}{a-1}$, and $W$ is generated by $W_m = A$.  

For a projective $A$-module $P$ of rank $m-1$ and  a \emph{split 
monomorphism} $\alpha\colon P \to F \otimes A$, define a representation
$W_\alpha \in \calr(A)$ by
\[
(W_\alpha)_a = \L^{m-a}_AP^\svee
\]
for $a = 1, \dots, m$, with $P^\svee = \Hom_A(P,A)$. Define the action of
$\lambda \in F^\svee$ on $W_\alpha$ by the left exterior
multiplication
\[
\alpha^\svee(\lambda) \wedge - \colon  \L^{m-a}_A P^\svee \to
\L^{m-a+1}_A P^\svee\,,
\]
where $\alpha^\svee\colon F^\svee\otimes A \onto P^\svee$ is the $A$-dual of
$\alpha$.

\begin{lemma}\label{lem:allrepsQ}
  Every $W \in \calr(A)$ is of the form $W_\alpha$ for a uniquely
  determined $A$-projective $P$ of rank $m-1$ and split monomorphism
  $\alpha\colon P \to F \otimes A$.
\end{lemma}

\begin{proof}
Let $W \in \calr(A)$.  Viewed as a left module over $(\L
F^\svee)\otimes A = \L_A(F^\svee\otimes A)$, $W$ is generated by $W_m =
A$.  This gives in particular a surjective homomorphism
\[
\pi\colon  F^\svee\otimes A \to W_{m-1}\,.
\]
If $W = W_\alpha$ then $W_{m-1} = \L^1_A P^\svee = P^\svee$, and thus
$\alpha = \pi^\svee$ can be reconstructed from $W$, giving uniqueness.

For $W$ arbitrary, put $I = \ker\pi$.  As $W$ is generated by
$W_{m}$ we find that $W$ is a quotient of $\L_A((F^\svee \otimes A)/I)
= \L_A W_{m-1}$.  Since $W$ and $\L_A W_{m-1}$ have the same
rank, we see that $W \cong \L_A W_{m-1}$ is of the form $W_\alpha$.
\end{proof}

  With $\PP=\PP(F^\svee)$ once more the projective space of linear
  forms, let $\calu=\Omega^1(1)=\ker(F\otimes \calo_{\PP}\to
  \calop(1))$ be the tautological bundle.  Any split monomorphism
  $P\to F\otimes A$ with $P$ of rank $m-1$ is uniquely
  obtained as a pullback of $\calu \to F\otimes \calo_{\PP}$
  across an $A$-point $\eta\colon \Spec A \to \PP$ of $\PP$. Combining
  this with Lemma \ref{lem:allrepsQ} we obtain the following
  corollary.

\begin{cor}\label{cor:representQ}
The functor $\calr$ is representable by $\PP(F^\svee)$; equivalently,
$\PP(F^\svee)$ is a fine moduli space for $\calr$.  The universal
bundle is given by $\calb_0 = \L_{\PP(F^\svee)} \calu^\svee$, where
$\lambda \in F^\svee$ acts via $\partial^\svee(\lambda)\wedge-$.\qed
\end{cor}
\end{nsit}

\begin{nsit}{Representations of the quiverized Clifford
	algebra.}\label{nsit:repsQtilde}
Reintroduce now the second $K$-vector space $G$ of rank
$n$, with its fixed basis $\lbrace g_1, \dots, g_n\rbrace $, and consider again
from~\S\ref{sit:extBquiver} the doubled Be\u\i linson quiver on
$F^\svee$ and $G$
\[
\Qtilde:\qquad
{\xymatrix@C=5pc{
1 \ar@/_/[r]|{g_1} \ar@/_/@<-4ex>[r]^{\vdots}|{g_n}  & 
2 \ar@/_/[r]|{g_1} \ar@/_/@<-4ex>[r]^{\vdots}|{g_n}
\ar@/_/[l]_{\vdots}|{\lambda_1} \ar@/_/@<-4ex>[l]|{\lambda_m}  &  
\cdots \cdots \ar@/_/[l]_{\vdots}|{\lambda_1} \ar@/_/@<-4ex>[l]|{\lambda_m}
\ar@/_/[r]|{g_1} \ar@/_/@<-4ex>[r]^{\vdots}|{g_n}&  
m  \ar@/_/[l]_{\vdots}|{\lambda_1} \ar@/_/@<-4ex>[l]|{\lambda_m}
}}
\]
with relations as before.  Again let $C$ be its path algebra. 

For an arbitrary commutative $K$-algebra $A$, let $\tilde\calr(A)$ consist of
those isomorphism classes of representations
\[
W:\qquad
\xymatrix@C=5pc{
W_1 \ar@/_/[r]|{g_1} \ar@/_/@<-4ex>[r]^{\vdots}|{g_n}  & 
W_2 \ar@/_/[r]|{g_1} \ar@/_/@<-4ex>[r]^{\vdots}|{g_n}
\ar@/_/[l]_{\vdots}|{\lambda_1} \ar@/_/@<-4ex>[l]|{\lambda_m}  &  
\cdots \cdots \ar@/_/[l]_{\vdots}|{\lambda_1} \ar@/_/@<-4ex>[l]|{\lambda_m}
\ar@/_/[r]|{g_1} \ar@/_/@<-4ex>[r]^{\vdots}|{g_n}&  
W_m  \ar@/_/[l]_{\vdots}|{\lambda_1} \ar@/_/@<-4ex>[l]|{\lambda_m}
}
\]
such that each $W_a$ is a projective $A$-module of rank
$\binom{m-1}{a-1}$, and $W$ is generated as a left $C$-module by
$W_{m}=A$. 

\begin{prop}\label{prop:xijscalars}
Let $W \in \tilde\calr(A)$.  Then the central elements $x_{ij} \in
C$ act as scalars (elements of $A$) on $W$.  Furthermore, $W$
is generated by $W_m$ as a left module over $\L_A (F^\svee
 \otimes A)$.
\end{prop}
\begin{proof}
Each homogeneous $A\otimes C$-linear endomorphism of $W$ is
determined by its action on $W_m$.  From the fact that $W_m=A$, we
deduce that every such endomorphism is given by multiplication by some
element of $A$.  In particular, this holds for multiplication by
$x_{ij}$.

Any element of $C$ can be written as a linear combination of products
$e_b\bm{\lambda g x}e_a$, where $\bm \lambda$, $\bm g$, and $\bm x$ are
products of $\lambda_k$, $g_l$, and $x_{ij}$.  As each $g_l$ acts with
degree $+1$, $g_l W_m = 0$.  It follows that $W$ is generated by $W_m$
over $(\L F^\svee)\otimes A = \L_A(F^\svee\otimes A)$ alone.
\end{proof}

\begin{sit}
Suppose now we are given a projective $A$-module $P$ of rank
$m-1$, and a pair of homomorphisms
\[
\alpha\colon  P \to F\otimes A\,, \qquad
\beta\colon  P^\svee \to G^\svee \otimes A
\]
with $\alpha$ a split monomorphism.  Define $W_{\alpha\beta} \in
\tilde\calr(A)$ to have
\[
(W_{\alpha\beta})_a = \L^{m-a}_AP^\svee
\]
as before, with $\lambda \in F^\svee$ again acting via
$\alpha^\svee(\lambda)\wedge -$, and with $g\in G$ acting via the
contraction $\beta^\svee(g) \Ydown -$.  Explicitly, $g$ sends $u^{1}
\wedge \cdots \wedge u^{m-a}$ to 
\[
\sum_{j=1}^{m-a} (-1)^{j-1}u^{j}(\beta^\svee(g))\left(u^{1} \wedge \cdots \wedge
\widehat{u^{j}} \wedge \cdots \wedge u^{m-a}\right) \quad \in
\L^{m-a-1}_AP^\svee\,. 
\]
\end{sit}
\begin{prop}\label{prop:allrepsQtilde}
Every $W \in \tilde\calr(A)$ is of the form $W_{\alpha\beta}$ for a
uniquely determined projective $P$ of rank $m-1$ and a pair of homomorphisms
\[
\alpha\colon P \to F \otimes A \quad,\quad \beta\colon P^\svee \to
G^\svee\otimes A
\]
with $\alpha$ a split monomorphism.
\end{prop}
\begin{proof}
Any representation class $W \in \tilde\calr(A)$ can be viewed as an
object of $\calr(A)$ by simply ignoring the rightward-pointing arrows
of $\Qtilde$.  By Lemma~\ref{lem:allrepsQ}, such an
object is necessarily of the form $W_\alpha = \L_A P^\svee$ for some
$P$ and some monomorphism $\alpha\colon  P \into F \otimes A$.  It remains
only to construct $\beta$.

The central elements $x_{ij} = \lambda_i g_j + g_j \lambda_i \in C$ act on
$W$ as multiplication by certain scalars $a_{ij} \in A$.  Applying
this to the generator $1 \in A = W_m$, we obtain
\[
a_{ij} = g_j \lambda_i\,,
\]
so that in particular each $g_j$ acts as the left super-$S$-derivation
on $\L_A P^\svee$ sending $\alpha^\svee(\lambda_i)$ to $a_{ij}$.  Hence
the action of $G$ on $\L_A P^\svee$ is provided by a homomorphism
$\gamma\colon G \otimes A \to P$, which dualizes to a map $\beta\colon
P^\svee \to G^\svee \otimes A$ such that $\gamma(g)$ is given by
contraction with $\beta^\svee(g)$ for each $g \in G$.  This shows that $W
\cong W_{\alpha\beta}$.  As in Lemma~\ref{lem:allrepsQ}, both $\alpha$
and $\beta$ can be reconstructed from $W$.
\end{proof}

Let $\calz$ again be the Springer desingularization of $\Spec R$.  As
in~\S\ref{nsit:resRjOZ}, we write
\[
\calz = \underline{\Spec} \left(\Sym_{\PP(F^\svee)}(\calu^\svee \otimes G)\right)\,.
\]
The bundle $\L_{\PP}\calu^\svee \otimes_\PP \Sym_\PP(\calu^\svee\otimes G)$
carries a natural $C$-action where $\lambda \in
F^\svee$ acts via $\partial^\svee(\lambda) \wedge -$ and $g \in G$
sends a section $e$ of $\calu^\svee$ to $e\otimes g$ and fixes
$\calu^\svee\otimes G$.  Denote this latter super-derivation by
$g\Ydown-$.  Letting $\calb$ be the $\caloz$-module determined by
$\L_{\PP}\calu^\svee\otimes_\PP\Sym_\PP(\calu^\svee\otimes G)$, we see that
$\calb = p'^* \calb_0$, where $\calb_0 = \L_{\PP} \calu^\svee$ is as in
Corollary~\ref{cor:representQ} and $p'\colon \calz \to \PP$ is the
projection.  Of course $\calb$ is still a $C$-module. 

\begin{theorem}\label{thm:representQtilde}
The functor $\tilde\calr$ is representable by $\calz$.  The universal
bundle is given by $\calb = p'^*(\L_{\PP} \calu^\svee)$.
\end{theorem}

\begin{proof}
An $A$-point of $\calz$ consists of two pieces of data.  The
first of these is a point $\eta\colon \Spec A \to \PP$, and 
 we obtain from the canonical map $\calu\to F\otimes \calo_{\PP}$ 
a split monomorphism
\[
\partial_\eta\colon  \calu_\eta \to F \otimes A 
\]
with $\calu_\eta$ an $A$-projective of rank $m-1$.  The other
information carried by a point of $\calz$ is an $A$-point $\xi\colon
\Spec A \to \underline{\Spec} \Sym_{\PP}(\calu^\svee_\eta \otimes
G)$.  Such a point corresponds to an $A$-linear map $\calu^\svee_\eta
\otimes G \to A$, which by adjunction yields a homomorphism
$\beta\colon \calu^\svee_\eta \to G^\svee \otimes A$.

Thus the $A$-points of $\calz$ are in one-one correspondence with the
pairs $(\alpha,\beta)$, i.e., with the elements of
$\tilde\calr(A)$.  This proves that $\calz$ represents $\tilde\calr$.
It is easy to see that the induced actions of $F^\svee$ and $G$ on
$\calb_\xi = (\calb_0)_\eta$ define an isomorphism $\calb_\xi \cong
W_{\alpha\beta}$.
\end{proof}
\end{nsit}

\begin{nsit}{Simple representations.}
Our next task is to identify the points of $\calz$ corresponding to
the simple representations $W \in \tilde\calr(A)$.  We shall see that
they are precisely those points lying over the non-singular locus of
$\Spec R$.  We first record an easy lemma.
\end{nsit}

\begin{lemma}\label{lem:WgenbyW1}
Assume that $A = K$.  Then $W \in \tilde\calr(K)$ is simple if and
only if $W$ is generated by $W_1$. 
\end{lemma}
\begin{proof} If we consider only the action of the $\lambda_i$ then
$W=\L P^\svee$. We see that any subrepresentation of $W$ contains
its socle $W_m=\L^{m-1}_A P^\svee$. Hence if $W_1$ generates $W$ then
this subrepresentation 
must be everything. 
\end{proof}
\begin{lemma}\label{lem:simpleCmods}
The following are equivalent for $W = W_{\alpha\beta} \in
\tilde\calr(K)$:
\begin{enumerate}
\item $W$ is a simple left $C$-module;
\item $\beta\colon  P^\svee \to G^\svee$ is a monomorphism.
\end{enumerate}
\end{lemma}
\begin{proof}
The perfect pairing 
\[
\L^{m-a}P^\svee \times \L^{a-1}P^\svee \to \L^{m-1}P^\svee \cong A
\]
defines an isomorphism
\[
W_a = \L^{m-a}_AP^\svee \cong \left(\L^{a-1}_AP^\svee\right)^\svee \cong
\L^{a-1}P\,.
\]
For any $g \in G$, then, the diagram
\[
\xymatrix{
\L^{m-a}P^\svee \ar[r]^\cong \ar[d]_{\beta^\svee(g)\Ydown -} & 
   \L^{a-1}P \ar[d]^{\beta^\svee(g)\wedge-} \\
\L^{m-a-1}P^\svee \ar[r]^(.6)\cong & \L^{a}P
}
\]
is commutative.  We see from Lemma~\ref{lem:WgenbyW1} that $W$ is
generated by $W_1$ if and only if
\[
\beta^\svee(g)\Ydown-\colon  \L^{m-1}P^\svee \otimes G^\svee \to
  \L^{m-2}P^\svee
\]
is surjective, if and only if $\beta\colon P^\svee \to G^\svee$ is injective.
\end{proof}

\begin{prop}
A representation $W_{\alpha\beta} \in \tilde\calr(K)$ is simple if and
only if the corresponding point in $\calz$ lies over the non-singular
locus of $\Spec R$.
\end{prop}

\begin{proof}
Recall that the projection $q'\colon \calz \to \Spec R$ is an
isomorphism over the non-singular locus of $\Spec R$.  One checks that
the composition $\calz \xrightarrow{q'} \Spec R \into \Spec S \cong
F\otimes G^\svee$ sends a point of $\calz$, viewed as a pair of
homomorphisms $(\alpha, \beta)$ as above, to the composition
\[
K \to P\otimes P^\svee \xto{\alpha\otimes\beta}F\otimes G^\svee\,.
\]
Thus a point of $\calz$ corresponds to a simple $C$-module if, and
only if, $\alpha \otimes \beta$ has rank $n-1$, which occurs exactly
when it lies over the non-singular locus of $\Spec R$.
\end{proof}

\section{Explicit Minimal Presentations}
\label{sec:explicit}
In this section we will write down explicit minimal $S$-presentations
for the Cohen-Macaulay modules $\Hom_R(M_a,M_b)$. By Theorem
\ref{thm:qCliffisoEndM} this amounts to giving an $S$-free
presentation of $C_{ab}$.  By the involution $C_{ab} \leftrightarrow
C_{m+1-b,m+1-a}$ we see that we may as usual assume $a+b\ge m+1$.
Below we will show that \eqref{eq:prelimpresentation} yields a minimal
presentation of $C_{ab}$ provided we drop the projective $Q$.
Furthermore we give an explicit matrix representation for~$\rho$.

In characteristic zero our presentation can be block diagonalized
yielding a decomposition of $\Hom_R(M_a,M_b)$ into certain
maximal Cohen-Macaulay modules of lower rank.

\begin{nsit}{A star product.}\label{sit:starproduct}
We first recall a well-known formula of Gerstenhaber and Schack.
Assume that $\psi_1, \dots, \psi_n, \theta_1, \dots, \theta_n$ are
commuting nilpotent derivations on a commutative algebra $A$
containing $\QQ$.  Then, denoting by $\m\colon A \otimes A \to A$ the
multiplication in $A$, there is an associated 
associative product
\[
u*v = \m(e^{\psi_1\otimes\theta_1 + \cdots
  \psi_n\otimes\theta_n}(u\otimes v))
\]
on $A$.  It is easy to see that this formula generalizes to the graded
case.

Applying this formula with $A = \L_S(\F^\svee \oplus\G)$
and 
\[
\psi_i = \dd{g_i}\,, \qquad \qquad \theta_i = - \sum_{j=1}^m
x_{ji}\dd{\lambda_j}
\]
for $i=1, \dots, n$, yields a multiplication on $A$ via
\begin{equation}\label{eq:Deltastar}
u*v = \m(e^{-\Delta}(u\otimes v))
\end{equation}
where
\[
\Delta = \sum_{
\begin{smallmatrix}i=1,\ldots,n\\
j=1,\ldots,m
\end{smallmatrix}}x_{ji} \dd{g_i} \otimes \dd{\lambda_j}\,.
\]
\end{nsit}
\begin{lemma}
The star product on $A = \L_S( \F^\svee \oplus
\G)$ gives $A$ the structure of a quadratic $S$-algebra
generated by the symbols $\lambda_1, \dots, \lambda_m, g_1, \dots,
g_n$ subject to the relations
\begin{align*}
\lambda_k * \lambda_l &= \lambda_k \lambda_l = -\lambda_l \lambda_k =
-\lambda_l * \lambda_k,&\lambda_k\ast\lambda_k&=\lambda_k^2=0,\\ 
g_k * g_l &= g_k g_l = - g_l g_k = - g_l * g_k, &g_k\ast g_k&=g_k^2=0,\\
g_k * \lambda_l&= g_k \lambda_l +x_{kl} = -\lambda_l * g_k + x_{kl}\,.
\end{align*}
In other terms, $(A, *)$ is isomorphic to the Clifford algebra $C$ on
$F^\svee$ and $G$.
\end{lemma}
We quickly show that in this particular case (\ref{eq:Deltastar}) is
defined over $\ZZ$ and thus is true in arbitrary characteristic.  To
this end we have to compute $\Delta^t$. We find
\begin{align*}
\Delta^t&=\sum x_{j_1i_i}\cdots x_{j_ti_t}
 \dd{g_{i_t}}\cdots\dd{g_{i_1}}\otimes
  \dd{\lambda_{j_1}}\cdots\dd{\lambda_{j_t}}\\
&=t! \sum_{j_1<\cdots <j_t} x_{j_1i_i}\cdots x_{j_ti_t}
 \dd{g_{i_t}}\cdots\dd{g_{i_1}}\otimes
  \dd{\lambda_{j_1}}\cdots\dd{\lambda_{j_t}}\\
&=t!
\mathop{\sum_{i_1< \cdots<i_t}}_{j_1<\cdots<j_t} [i_1\cdots
	i_t\;|\;j_1\cdots j_t] 
  \dd{g_{i_t}}\cdots\dd{g_{i_1}}\otimes
  \dd{\lambda_{j_1}}\cdots\dd{\lambda_{j_t}} 
\end{align*}
where the peculiar arrangement of indices is to eliminate some signs and
where $[i_1\cdots
  i_t\;|\;j_1\cdots j_t]$ is the (unsigned) determinant of the $(t\times
  t)$-submatrix of $X$ consisting of the rows indexed $i_1, \dots,
  i_t$ and columns indexed $j_1, \dots, j_t$.

It follows that if we set 
\[
\Delta^{(t)} = \frac{\Delta^t}{t!} = 
 \mathop{\sum_{i_1< \cdots<i_t}}_{j_1<\cdots<j_t} [i_t\cdots
	i_1\;|\;j_t\cdots j_1] 
  \dd{g_{i_t}}\cdots\dd{g_{i_1}}\otimes
  \dd{\lambda_{j_1}}\cdots\dd{\lambda_{j_t}}\,,
\]
then the star product on $A = \L_S(\F^\svee\oplus\G)$ is given by
\[
\m\circ\left(1-\Delta + \Delta^{(2)} - \cdots\right)\,.
\]
Return now to the free $S$-presentation of $C_{ab}$ given by
Lemma \ref{lem:minresCij2}.  We have the
following simplification of this presentation
\begin{prop}\label{prop:minresCij2} If $a+b\ge m+1$ then
$C_{ab}$ has a minimal $S$-free presentation of the form
\begin{equation}
\label{eq:reducedpresentation}
P_1 \xto{\rho} P_0 \to C_{ab} \to 0\,,
\end{equation}
where
\begin{itemize}
\item $P_0 = 
\bigoplus_{\max\lbrace a,b\rbrace \le k \le m}
\L^{k-a}_S \G
\otimes 
\L_S^{k-b}\F^\svee$.
\item $P_1 = 
\bigoplus_{\begin{smallmatrix} 0\geq l\geq \max\lbrace a-m,b-m\rbrace 
\end{smallmatrix}}
\L^{b-l}_S\G\otimes \L^{a-l}_S \F^\svee $
\item $\rho_{lk} = 
\begin{cases}(\Delta^{(a+b-k-l)})_{lk}&\text{if $a+b-k-l\ge 0$, and}\\
0&\text{otherwise.}
\end{cases}$.
\end{itemize}
\end{prop}

\begin{proof}
  Our starting point is the free presentation of $C_{ab}$ given in
 \eqref{eq:prelimpresentation}. It takes the form
\begin{equation}
\label{eq:toobigversion}
\bigoplus_{\begin{smallmatrix} 1> l\geq \max\lbrace a-m,b-n\rbrace 
\end{smallmatrix}}
\L^{b-l}_S\G\otimes \L^{a-l}_S \F^\svee
\to
\bigoplus_{\max\lbrace a,b\rbrace \le k \le m}
\L_S^{k-b}\F^\svee\otimes \L^{k-a}_S \G
\end{equation}
where $\rho$ is obtained by expanding paths that go first to the left
and then to the right in terms of paths that do the opposite.

Now we borrow some ingredients from the proof of Theorem~\ref{thm:expcan}.
Writing \eqref{eq:toobigversion} in the form
\[
Q\oplus P_1\xrightarrow{\rho} P_0
\]
as in \eqref{eq:prelimpresentation} we deduce from the fact that
$\alpha_1$ is shown to be invertible in \eqref{eq:finaldiagram} that
$\rho$ and $(\rho\mid_{P_1})\colon P_1\to P_0$ represent the same
$S$-module.  This shows that $C_{ab}$ has a presentation as in
\eqref{eq:reducedpresentation}.
Furthermore the resulting matrix entry
\[
\L^{b-l}_S\G\otimes \L^{a-l}_S \F^\svee
\to 
\L_S^{k-b}\F^\svee\otimes \L^{k-a}_S \G
\]
can be deduced by working in $(\L_S(\F^\svee\oplus \G),\ast)$.  We
find that it is the composition
\begin{multline}
\L^{b-l}_S\G\otimes \L^{a-l}_S \F^\svee
  \xrightarrow{(-1)^{a+b-k-l}\Delta^{(a+b-k-l)}}
  \L^{k-a}_S \G \otimes  \L_S^{k-b}\F^\svee\\
\xrightarrow{(-1)^{(k-b)(k-a)}}
\L_S^{k-b}\F^\svee\otimes \L^{k-a}_S \G\,.
\end{multline}
The presentation given in the statement of the
proposition is deduced from this by pre- and postcomposing with
invertible diagonal matrices (with diagonal entries in $\lbrace \pm 1\rbrace $).

The presentation is minimal if and only if $a+b-k-l\ge 1$ for all
allowable $k,l$. It is enough to test this for $k,l$ maximal, i.e.\
$k=m$, $l=0$.  Then $a+b-k-l=a+b-m$ which is positive if and only if
$a+b\ge m+1$.
\end{proof}

\begin{example} Assume that $m=n=5$, $a=b=4$. Then
\begin{align*}
P_0&=\bigoplus_{4\le k \le 5}
\L^{k-4}_S \G
\otimes 
\L_S^{k-4}\F^\svee=K\oplus \G\otimes \F^\svee\\
P_1&=\bigoplus_{\begin{smallmatrix} 0\geq l\geq -1
\end{smallmatrix}}
\L^{4-l}_S\G\otimes \L^{4-l}_S \F^\svee 
=\L^{5}_S\G\otimes \L^{5}_S \F^\svee  
\oplus
 \L^{4}_S\G\otimes \L^{4}_S \F^\svee 
\end{align*}
and the matrix form of the presentation is 
\begin{equation*}
{
\begin{pmatrix}
\Delta^{(5)}&\Delta^{(4)}\\
\Delta^{(4)}&\Delta^{(3)}
\end{pmatrix}
}
\end{equation*}
 It will be clear to the reader that over $\QQ$ this presentation can
be diagonalized further. We will say more on this below.
\end{example}

\begin{nsit}{Characteristic zero.}
In this section we assume $\operatorname{char}K=0$.  For
$\alpha,\beta\ge 0$ with $\alpha+\beta<m$, define $C^{\alpha\beta}$ to
be the cokernel of
\[
\Delta^{(m-\alpha-\beta)}\colon  \L_S^{m-\beta}\G\otimes
\L_S^{m-\alpha} \F^\svee
\to
 \L_S^{\alpha}\G\otimes
\L_S^{\beta} \F^\svee\,.
\]
\end{nsit}

\begin{prop}
Assume $\operatorname{char}K=0$ and that $a+b\ge m+1$.  Then
\begin{enumerate}
\item The $C^{\alpha\beta}$ are maximal Cohen-Macaulay $R$-modules.
\item We have a decomposition
\[
C_{ab}=\bigoplus_{\max\lbrace a,b\rbrace \le p\le m} C^{p-a,p-b}\,.
\]
\end{enumerate}
\end{prop}

\begin{proof}
According to Proposition~\ref{prop:minresCij2}, the map $\rho$ written
as a matrix has the form
\begin{equation}
\label{formalmatrix}
\rho=
\begin{pmatrix}
\Delta^{(r)} & \Delta^{(r-1)} & \cdots & \\
\Delta^{(r-1)} & \cdots & &\\
\vdots&&&\vdots\\
&&\cdots& \Delta^{(s)}
\end{pmatrix}
=
\begin{pmatrix}
\frac{\Delta^r}{r!}& \frac{\Delta^{r-1}}{(r-1)!}&\cdots & \\
\frac{\Delta^{r-1}}{(r-1)!}&\cdots&&\\
\vdots&&&\vdots\\
&&\cdots&\frac{\Delta^s}{s!}
\end{pmatrix}\,.
\end{equation}
Here $\Delta^{(r)}$ represents the map from $\L^{b-l}_S\G\otimes
\L^{a-l}_S \F^\svee$ to $\L^{k-a}_S \G \otimes \L_S^{k-b}\F^\svee$ for
$k=\max\lbrace a,b\rbrace $ and $l=\max\lbrace a,b\rbrace -m$. Thus
$r=a+b-2\max\lbrace a,b\rbrace +m=-|a-b|+m$.

Similarly  $\Delta^{(s)}$ represents the map
$\L^{b-l}_S\G\otimes \L^{a-l}_S \F^\svee$
to $\L^{k-a}_S \G
\otimes 
\L_S^{k-b}\F^\svee$ for $k=m$, $l=0$. Thus
$s=a+b-m$.

It order to manipulate \eqref{formalmatrix} we write it \emph{formally}
as 
\[
\rho=
\begin{pmatrix}
\Delta^{r/2}&\cdots & 0\\
\vdots &\ddots &\vdots\\
0&\cdots & \Delta^{s/2}
\end{pmatrix}
\begin{pmatrix}
\frac{1}{r!}& \frac{1}{(r-1)!}&\cdots & \\
\frac{1}{(r-1)!}&\ddots&&\\
\vdots&&&\vdots\\
&&\cdots&\frac{1}{s!}
\end{pmatrix}
\begin{pmatrix}
\Delta^{r/2}&\cdots & 0\\
\vdots &\ddots &\vdots\\
0&\cdots & \Delta^{s/2}
\end{pmatrix}\,.
\]
Let $A$ be the middle scalar matrix. According to
Lemma~\ref{diagonalization} below we have $A=PDP^t$ where $D$ is a
non-singular diagonal matrix and $P$ is upper triangular with $1$'s on
the diagonal, both with rational entries.

Let $\tilde{P}$ be obtained from $P$ by replacing $P_{ij}$ by
$\Delta^{j-i}P_{ij}$.  Then after a bit of manipulation we obtain the
following (non-formal) expression for $\rho$.
\[
\rho=
\tilde{P}
\begin{pmatrix}
D_{rr}\frac{\Delta^r}{r!}& 0&\cdots & \\
0&\ddots&&\\
\vdots&&&\vdots\\
&&\cdots&D_{ss}\frac{\Delta^s}{s!}
\end{pmatrix}
\tilde{P}^t
\]
As $\tilde{P}$ is invertible, this shows that $C_{ab}$ indeed has a
decomposition as indicated in the statement of the proposition.  If
follows that $C^{\alpha\beta}$ is a maximal Cohen-Macaulay $R$-module
if $C^{\alpha\beta}$ occurs as a summand among one of the $C^{ab}$.
Given $\alpha,\beta\ge 0$, with $\alpha+\beta<m$, we put $p=m$ so that
$a=m-\alpha$, $b=m-\beta$. Then $a+b=2m-(\alpha+\beta)\ge m+1$, as
required.
\end{proof}
\begin{example} The following matrix gives the decomposition of $C_{ab}$ for
$m=3$ (and $n$ arbitrary).
\[
\begin{pmatrix}
C^{00} & C^{10} & C^{20}\\
C^{01} & C^{00}\oplus C^{11} & C^{10}\\
C^{02} & C^{01} & C^{00}
\end{pmatrix}
\]
The cases $a+b\ge m+1=4$ are covered by the proposition. For the other
cases we perform the involution $(a,b)\mapsto (m+1-b,m+1-a)=(4-b,4-a)$. 
\end{example}

The following lemma is used in the next lemma, which was used in the
above proof.
\begin{lemma} Let $u\ge 2t$ and let $A$ be the $(t\times t)$-matrix over $\QQ$
\[
A_{ij}=\frac{1}{(u-i-j)!}
\]
with $1\le i,j\le t$. Then $\det A\neq 0$.
\end{lemma}
\begin{proof}
Put
$
B=(u-2)! A
$.
Then $B$ is equal to
\small
\[
\begin{pmatrix}
1&x&x(x-1)&\cdots & x(x-1)\cdots (x-t+2)\\
x&x(x-1)&x(x-1)(x-2)&\cdots &x(x-1)\cdots (x-t+1)\\
x(x-1)&x(x-1)(x-2)&x(x-1)(x-2)(x-3)&\cdots &x(x-1)\cdots (x-t)\\
\vdots & \vdots &\vdots &\ddots &\vdots
\end{pmatrix}
\]
\normalsize
with $x=u-2$.
 Then
\begin{align*}
\det B&=x\cdot x(x-1)\cdot x(x-1)(x-2)\cdots x(x-1)\cdots (x-t+2) \det C\\
&=x^{t-1} (x-1)^{t-2}\cdots (x-t+2) \det C
\end{align*}
with $C$ being equal to
\[
\begin{pmatrix}
1&x&x(x-1)&\cdots & x(x-1)\cdots (x-t+2)\\
1&x-1&(x-1)(x-2)&\cdots &(x-1)\cdots (x-t+1)\\
1&x-2&(x-2)(x-3)&\cdots &(x-2)\cdots (x-t)\\
\vdots & \vdots &\vdots &\ddots &\vdots
\end{pmatrix}\,.
\]
If we put $x_i=x-i$ then $C$ can be written as 
\[
\begin{pmatrix}
1&x_0&x_0(x_0-1)&\cdots & x_0(x_0-1)\cdots (x_0-t+2)\\
1&x_1&x_1(x_1-1)&\cdots &x_1\cdots (x_1-t+2)\\
1&x_2&x_2(x_2-1)&\cdots &x_2\cdots (x_2-t+2)\\
\vdots & \vdots &\vdots &\ddots &\vdots
\end{pmatrix}
\]
which using column operations can be turned into a Vandermonde determinant.
Hence
\[
\det C=\prod_{0\le i<j \le t-1} (x_j-x_i)=\prod_{0\le i<j \le t-1}
(i-j)\neq 0\,. 
\]
\end{proof}
\begin{lemma} \label{diagonalization} Let $A$ be as in the previous
  lemma. Then $A=PDP^t$ with $D$ 
diagonal and $P$ upper triangular with $1$'s on the diagonal.
\end{lemma}
\begin{proof} We view $A$, being a symmetric matrix, as a quadratic
  form.  Diagonalizing it in the usual way, starting with the last
  variable, we see that we need $\det (A_{ij})_{p\le i,j\le t}\neq 0$
  for $p=1,\ldots,t$. This follows from the previous lemma.
\end{proof}

\section{Minimal Resolutions of the Simples in Characteristic
  Zero}\label{sect:char0simples} 

In this final section we require $K$ to be a field of characteristic
zero.  Other than that, we keep the established
notation.  Our aim in this section is to compute the $\Ext$-groups
among the graded simple modules over the non-commutative
desingularization $E$, and so obtain the shapes of their minimal
graded free resolutions.

\begin{nsit}{The main result.}
We follow the notation of Weyman~\cite{Weyman:book} for Schur modules
$L_\alpha$ corresponding to partitions $\alpha = (\alpha_1, \dots,
\alpha_q)$.  Let $\Gamma(m,n)$ be the set of Young diagrams
(identified, as usual, with partitions) having at most $m$ rows and
$n$ columns. The conjugate partition $\alpha'$ is obtained by
reflection across the line $y=-x$.

A \emph{convex square} of a diagram $\alpha \in \Gamma(m,n)$ is a
square with coordinates $(r, \alpha_r)$ such that $\alpha_{r+1} <
\alpha_r$. For a convex square $(r,c)$ in $\alpha$, let $R_r(\alpha)$ be the
partition obtained from $\alpha$ by dropping the $r^\text{th}$ row.
Similarly, $C_c(\alpha)$ is obtained by dropping $\alpha$'s
$c^\text{th}$ column.  For example, we have indicated below the convex
squares for the partition $(421)$.
\[
\includegraphics*[scale=0.8]{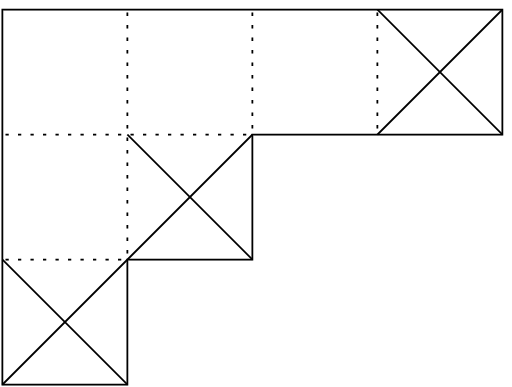}
\]
We obtain three corresponding pairs of partitions
$(C_c(\alpha),R_r(\alpha))$, namely $((321),(21))$, $((311),(41))$, and
$((31),(42))$. 
\end{nsit}

For $a=1,\ldots,m$ let $P_a=\Hom_R(M_a,M)$ be the corresponding 
graded projective left $E$-module and let $S_a$ be the associated
graded simple module. We have:
\begin{theorem}\label{thm:extsimples}
  Assume $\operatorname{char} K=0$. For simple right $E$-modules $S_a$
  and $S_b$, we have
\[
\Ext_E^t(S_b,S_a) \cong \bigoplus L_{C_c(\alpha)} F \otimes
L_{R_r(\alpha)'} G^\svee\,,
\]
where the direct sum is taken over all partitions $\alpha \in
\Gamma(m,n)$ such that $|\alpha| = t+1$, and over all convex squares
$(r,c)$ in $\alpha$ such that $-a+b = -r+c$.
\end{theorem}
The proof of this theorem will occupy the remainder of the section.

\begin{example}
We can evaluate the sum above for small values of $t$, obtaining the
first few terms of the resolution of $S_a$:
\footnotesize
\[
\xymatrix@C=1pc{
P_a  &
\mbox{\ensuremath{
\begin{array}{c}
P_{a-1}(-1)\otimes F^\svee \\ 
{}\oplus{} \\ 
P_{a+1}(-1) \otimes G
\end{array}
}}
\ar[l] &
\mbox{\ensuremath{
\begin{array}{c}
P_{a-2}(-2) \otimes \SS^2 F^\svee \\ 
{}\oplus{} \\ 
P_{a-1}(-3)\otimes \L^2 F^\svee \otimes G\\
{}\oplus{} \\ 
P_{a+1}(-3) \otimes F^\svee\otimes \L^2 G\\
{}\oplus{} \\ 
P_{a+2}(-2) \otimes \SS^2 G
\end{array}
}}
\ar[l] &
\mbox{\ensuremath{
\begin{array}{c}
P_{a-3}(-3) \otimes \SS^3 F^\svee \\ 
{}\oplus{} \\ 
P_{a-2}(-4) \otimes L_{21} F^\svee \otimes G\\
{}\oplus{} \\ 
P_{a-1}(-5) \otimes \L^3 F^\svee \otimes \SS^2 G\\
{}\oplus{} \\ 
P_{a}(-4) \otimes \L^2 F^\svee \otimes \L^2 G\\
{}\oplus{} \\ 
P_{a+1}(-5) \otimes \SS^2 F^\svee \otimes \L^3 G\\
{}\oplus{} \\
P_{a+2}(-4) \otimes F^\svee \otimes L_{21} G\\
{}\oplus{} \\
P_{a+3}(-3) \otimes \SS^3 G 
\end{array}
}}
\ar[l]
}
\]
\normalsize where we understand $P_i =0$ if $i \notin [1,m]$.  From
this resolution we can read off the generators and relations of
$C\cong E$.  Of course, the result is consistent with
Remark~\ref{rem:cubic}. The interpretation of the higher terms in the
resolution remains open.
\end{example}

\begin{nsit}{Translation into geometry.}
As a matter of notational convenience in this section, we dualize and
work with the \emph{twisted tangent bundle} $\calq:= \calu^\svee =
(\Omega^1(1))^\svee$ on $\PP = \PP(F^\svee)$ defined by exactness of
the sequence
\begin{equation*}\label{eq:defcalq}
0 \to \calop(-1) \to F^\svee \otimes \calop \to \calq \to 0\,.
\end{equation*}
With the same argument as in  Theorem~\ref{tilting}  it follows
that $\calm' = p'^*(\L \calq)$ is also a tilting bundle on
$\calz$.  In particular, we have the exact equivalence of
categories
\begin{equation*}
\rHom_{\calo_\calz}(p'^*(\L \calq), -) \colon 
\cald^b(\coh(\calz)) \xto{\phantom{space}} \cald(E)
\end{equation*}
since $\End_{\calo_\calz}(p'^*(\L \calq))\cong 
\End_{\calo_\calz}(\calt)^\text{op}=
E^\text{op}$.
This equivalence sends each $p'^*(\L^a \calq)$, $a = 0,
\dots, m-1$, to the graded projective left $E$-module $P_{a+1}$.
\end{nsit}

\begin{lemma}\label{lem:identsimples} Let $u\colon \PP\to \calz$ be
the zero section of the vector bundle $p'\colon \calz\to \PP$ (see
\S\ref{nsit:resRjOZ}).  The object in $\cald^b(\coh(\calz))$
corresponding to the simple module $S_{a+1}$ is $u_\ast\calop(-a)[a]$.
\end{lemma}
\begin{proof}
We must show that $\Ext_{\calo_\calz}^t(p'^*\L^b\calq,
u_\ast\calop(-a)[a])$ is one-dimensional if $t=0$ and $a=b$, and
vanishes otherwise.  By adjunction it suffices to prove
\[
\Ext_{\calo_\PP}^t\left(\L^b\calq,\calop(-a)\right) =H^t(\PP, \Omega^b(b-a))=
\begin{cases}
K & \text{if } t=a=b, \text{and}\\
0 & \text{otherwise.}
\end{cases}
\]
Computing
\[
\Omega^b(b-a)=\calm^m_{b+1}(-a-1)\otimes |F|^\svee\,,
\]
we finish the proof by invoking Theorem \ref{thm:directimage}. 
\end{proof}

Hence in order to prove Theorem \ref{thm:extsimples} it is sufficient
to compute 
\[
\Ext_{\calo_\calz}^t\left(u_\ast\calop(-b)[b], u_\ast\calop(-a)[a]\right) =
\Ext_{\calo_\calz}^{t-b+a}\left(u_\ast\calop(-b), u_\ast\calop(-a)\right)\,.
\]
To this end we prove something more general.

\begin{prop}\label{prop:Hodge}
Let $\calu, \calv$ be objects in $\cald^b(\coh(\PP))$.  Then 
\[
\Ext_{\calo_\calz}^t\left(u_\ast\calu, u_\ast\calv\right) = 
\bigoplus_{s} \Ext_{\calo_\PP}^{t-s}\left(\L^{s}(\calq\otimes G)\otimes_\PP
\calu, \calv\right)\,.
\]
\end{prop}

\begin{proof}
We may assume that $\calu$ is a bounded complex of locally free
$\calop$-modules.  The locally free resolution of $u_\ast\calu$ as
$\caloz$-module is then given by
\[
\cdots 
\to \L^2(\calq\otimes G)\otimes_{\calo_\PP} \calu \otimes_{\calo_\PP} \caloz
\to \calq\otimes G \otimes_{\calo_\PP} \calu \otimes_{\calo_\PP} \caloz 
\to \calu\otimes_{\calo_\PP} \caloz \to 0\,.
\]
It follows that $\R\cHom_{\calo_\calz}(\calu,\calv)$ is equal in
$\cald^b(\coh(\PP))$ to 
the complex
\small
\[
0 
\to \cHom_{\calo_\PP}(\calu, \calv) 
\to \cHom_{\calo_\PP}(\calq\otimes G\otimes_{\calo_\PP} \calu, \calv)
\to \cHom_{\calo_\PP}(\L^2(\calq\otimes G) \otimes_{\calo_\PP} \calu, \calv)
\to \cdots\,.
\]
\normalsize
We note however that the center of $\GL(G)$ acts with different
weights on the terms of this complex.  It follows that the
maps are necessarily all zero, whence
\[
\R\cHom_{\calo_\calz}(\calu, \calv) = \bigoplus_s
\cHom_{\calo_\PP}\left(\L^s(\calq \otimes G)\otimes_{\calo_\PP} \calu,
\calv\right)[-s]\,.
\]
This implies the form claimed.
\end{proof}

\begin{proof}[Proof of Theorem~\ref{thm:extsimples}]
From Lemma~\ref{lem:identsimples} and Prop.~\ref{prop:Hodge} we obtain
\begin{align*}
\Ext_E^t(S_b,S_a) &=
\Ext_{\calo_\calz}^{t-b+a}\left(u_\ast\calop(-b+1)[b-1],u_\ast\calop(-a+1)[a-1]\right)
\\ &= \bigoplus_s H^{t-b+a-s}\left(\PP, \L^s(\calq\otimes G)^\svee(b-a)\right)\,.
\end{align*}
Expanding $\L^s(\calq\otimes G)$ according to the Cauchy formula (this
is the first time we use  $\operatorname{char} K=0$)
\[
\L^s(\calq\otimes G) = \bigoplus_{|\alpha|=s} L_\alpha\calq \otimes
L_{\alpha'}G
\]
we find
\[
\Ext_E^t(S_b,S_a) 
= \bigoplus_\alpha H^{t-b+a-|\alpha|} (\PP, (L_\alpha
\calq)^\svee(b-a)) \otimes L_{\alpha'}G^\svee\,.
\]
To continue, we apply Serre duality:
\[
H^{t-b+a-|\alpha|}(\PP, (L_\alpha \calq)^\svee(b-a)) = 
H^{m-1-t+b-a+|\alpha|}(\PP,(L_\alpha\calq)(a-b-m))^\svee \otimes
|F^\svee|\,.
\]
Using a straightforward application of Bott's theorem (see 
the discussion after the current proof) the direct sum can
now be written as
\begin{align}\label{eq:rimhook}
\Ext_E^t(S_b,S_a) 
&= 
\mathop{\bigoplus_{\alpha <_{m-a+b} \beta}}_{l(\beta-\alpha)
  = m-t-a+b+|\alpha|} L_\beta F \otimes L_{\beta'}G^\svee \otimes
|F^\svee| \\
&= 
\mathop{\bigoplus_{\alpha <_{m-a+b} \beta}}_{c(\beta-\alpha)
  = t-|\alpha|+1} L_\beta F \otimes L_{\beta'}G^\svee \otimes
|F^\svee|
\end{align}
where the notation $\alpha <_s \beta$ means that $\beta-\alpha$ is a
\emph{rim hook} (or \emph{border strip}) of length $s$ ending at the
$m^\text{th}$ row. Recall that a rim hook is a connected skew tableau
not containing any $(2\times 2)$-squares.  We write $l(\beta-\alpha)$ for the
number of rows in $\beta-\alpha$, and $c(\beta-\alpha)$ for the number
of columns.

The formula~\eqref{eq:rimhook} can be expressed symmetrically as 
\begin{equation}\label{eq:symmetricrimhooks}
 \Ext_E^t(S_b,S_a) = 
\mathop{\mathop{\bigoplus_{\mu \cong_{c,r} \nu}}_{t-|\nu|=c-1}}_{-a+b=-r+c-1}
L_\mu F \otimes L_{\nu'} G^\svee\,.
\end{equation}
In this sum $\mu$ runs over partitions with at most $m-1$ rows and $m$
columns, while $\nu$ runs over those with at most $m$ rows and $m-1$
columns.  The notation $\mu \cong_{c,r} \nu$ indicates that $\nu$
contains an embedded $r\times c$ rectangle as shown
\[
\includegraphics*[scale=0.6]{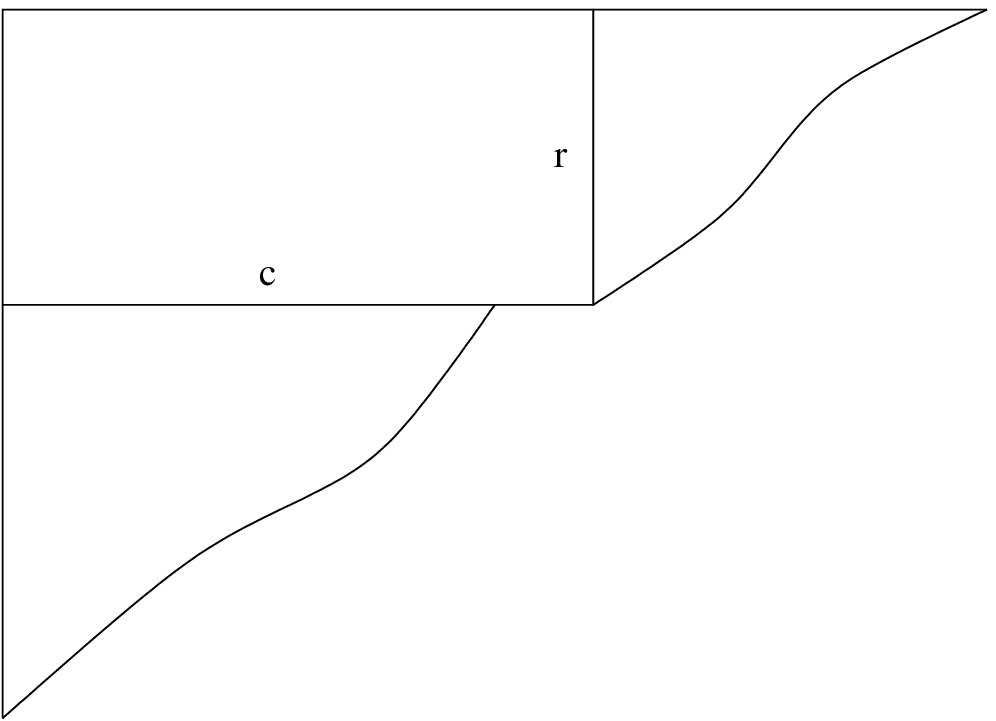}
\]
with $r \geq 0$, $c> 0$, and $\mu$ is obtained by replacing the
rectangle by an $(r+1)\times (c-1)$ rectangle.
\[
\includegraphics*[scale=0.6]{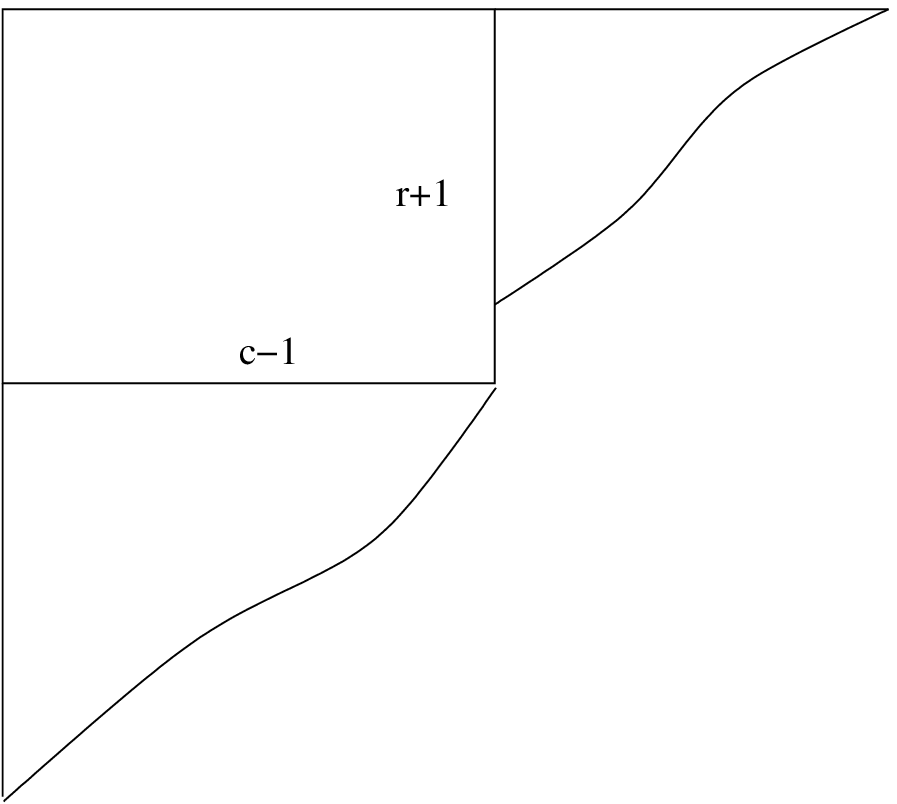}
\]
It is now easy to obtain the statement of Theorem~\ref{thm:extsimples}
from~\eqref{eq:symmetricrimhooks}, completing the proof.
\end{proof}

\begin{sit}
In the previous proof we have used Bott's theorem for which we provide
a brief reminder to the reader. Let $G$ be a reductive group and let
$T\subset B\subset P\subset G$ be respectively a maximal torus $T$, a
Borel subgroup $B$ and a parabolic subgroup $P$. For a dominant weight
$\theta\in X(T)$ let $L^G_\theta$ be the corresponding simple
$G$-representation.

Taking fibers in $[P]\in G/P$ provides an equivalence between rational
$P$-rep\-re\-sen\-ta\-tions and $G$-equivariant quasi-coherent sheaves on
$G/P$. Denote this equivalence by $\tilde{?}$. Let
$H=P/\rad P$ be the reductive part of $P$ and let
$L_\chi^H$ be the simple $H$-representation associated to a
$H$-dominant weight $\chi\in X(T)$.  We view $L^H_\chi$ as a
$P$-representation.

Bott's theorem computes
the cohomology of $\tilde{L}^H_\chi$ as follows 
\begin{equation}
\label{eq:bott}
H^i(G/P,\tilde{L}^H_\chi)
=\begin{cases}
L^G_\theta&\parbox{8cm}{\offinterlineskip \lineskip 3pt if there
  exists a (necessarily unique) $w\in W$ such that  
$\theta=w\cdot \xi$ is $G$-dominant and $l(w)=i$}
 \\
0&\text{otherwise.}
\end{cases}
\end{equation}
where $W$ is the Weyl group of $G$ and where $w\cdot
\xi=w(\xi+\rho)-\rho$ is the twisted Weyl group action, with $\rho$ as
usual being half the sum of the positive roots.

\medskip

Now in the setting of this paper choose an identification
$F^\svee=K^m$ and let $G=\GL_m(K)$. 
 Then $\PP(F^\svee)=G/P$ where $P$ is the stabilizer of the point
$p=(0,\ldots,0,1)$. Let $T=\lbrace \operatorname{diag}(t_1,\ldots,t_m)\rbrace \subset G$
be the diagonal torus. We view $t_1,\ldots,t_m$ as characters of $T$.

The roots of $G$ are $t_i t_j^{-1}$, $i\neq j$, with the positive
roots being those for which $i>j$ (in this setting the negative roots
are the non-zero weights of $\operatorname{Lie}(B)$).  The
$G$-dominant weights are of the form $t_1^{\alpha_1}\cdots
t_m^{\alpha_m}$ with $\alpha_1 \geq \dots \geq \alpha_m$. Thus the
dominant weights $\alpha$ are partitions with at most $m$ rows and one
has $L_\alpha^G=L_\alpha F^\svee$.  The (twisted) action of the Weyl
group is generated by the reflections
\begin{equation}
\label{eq:twisted}
s_i\colon  t_i^{\alpha_i}t_{i+1}^{\alpha_{i+1}} \mapsto
t_i^{\alpha_{i+1}-1}t_{i+1}^{\alpha_i+1}\,.
\end{equation}
The $G$-equivariant exact sequence 
\[
0\to \calop(-1)\to F^\svee\otimes \calop\to
\calq\to 0
\]
yields a $P$-equivariant exact sequence
\[
0\to \calop(-1)_p\to F^\svee\to
\calq_p\to 0
\]
with $\dim \calop(-1)_p=1$, $\dim \calq_p=m-1$. Such an exact sequence
is unique and must be isomorphic to 
\[
0\to K\to K^m\to K^{m-1}\to 0
\]
where the first non-trivial map is the injection into the last factor
and the second non-trivial map is the projection onto the first
$m-1$ factors. This means that $\calop(-1)=\tilde{L}^H_{t_m}$,
$\calq=\tilde{L}^H_{t_1}$ where $H=\GL_{m-1}(K)\times \GL_1(K)$.
Looking at the stalk in $p$ we also compute that for a partition
$\alpha$ with at most $m-1$ rows we have $
L_\alpha\calq(-s)=\tilde{L}^H_{t_1^{\alpha_1}\cdots
  t_{m-1}^{\alpha_{m-1}}t_m^s} $.  Hence to compute the cohomology of
$L_\alpha\calq(-s)$ using \eqref{eq:bott} we must try to flatten the
factor $t_m^s$ in the weight $t_1^{\alpha_1}\cdots
t_{m-1}^{\alpha_{m-1}}t_m^s$ using the twisted Weyl group action
\eqref{eq:twisted}. We see that this is only possible if there is a
partition $\beta$ with $m$ rows such that $\beta-\alpha$
is a rim hook with $s$ boxes and the number of reflections we need
in that case is one less than the number of rows in $\beta-\alpha$.
This completes the derivation of~\eqref{eq:rimhook}.
\end{sit}

 \bibliographystyle{amsplain}

 \newcommand{\arxiv}[2][AC]{\mbox{\href{http://arxiv.org/abs/#2}{\sf arXiv:#2 [math.#1]}}}
 \newcommand{\oldarxiv}[2][AC]{\mbox{\href{http://arxiv.org/abs/math/#2}{\sf arXiv:math/#2 [math.#1]}}}
 \providecommand{\MR}[1]{\mbox{\href{http://www.ams.org/mathscinet-getitem?mr=#1}{#1}}}
 \renewcommand{\MR}[1]{\mbox{\href{http://www.ams.org/mathscinet-getitem?mr=#1}{#1}}}
\providecommand{\bysame}{\leavevmode\hbox to3em{\hrulefill}\thinspace}
\providecommand{\MRhref}[2]{%
 \href{http://www.ams.org/mathscinet-getitem?mr=#1}{#2}
}
\def\cprime{$'$}

\end{document}